\documentclass[reqno,english,11pt]{amsart}
\usepackage{times}

\usepackage{amsmath,amsfonts,amssymb,graphicx,amsthm,enumerate,url,mathabx}
\usepackage[noadjust]{cite}
\usepackage{stmaryrd}
\usepackage{mathrsfs,booktabs,tabularx,setspace}
\usepackage{xifthen,xcolor,tikz}
\usetikzlibrary{arrows}
\usetikzlibrary{decorations.pathmorphing,patterns,shapes,calc,decorations}
\usetikzlibrary{decorations.pathreplacing}
\usepackage{mathtools,upgreek}
\usepackage[final]{showkeys} 

\definecolor{refkey}{gray}{.75}
\definecolor{labelkey}{gray}{.5}
\usepackage[algo2e,boxed,vlined,algoruled]{algorithm2e}

\usepackage[letterpaper]{geometry}
\usepackage[colorinlistoftodos]{todonotes}
\presetkeys{todonotes}{inline, color=green}{}

\usepackage[shortlabels]{enumitem}
\usepackage{accents}
\usepackage{float}

\usepackage[colorlinks=true]{hyperref}
\colorlet{DarkGreen}{green!50!black}
\colorlet{DarkGray}{gray!60!black}
\usepackage[capitalise]{cleveref}
\usepackage{subcaption}

\numberwithin{equation}{section}

\newcommand{\ignore}[1]{}

\renewcommand{\epsilon}{\varepsilon}

\newcommand{\one}{\mathbf{1}}

\usepackage{color}
 \definecolor{refkey}{gray}{.5}
 \definecolor{labelkey}{gray}{.5}
\definecolor{light}{gray}{.9}

\newtheorem{theorem}{Theorem}[section]
\newtheorem*{theorem*}{Theorem}
\newtheorem{lemma}[theorem]{Lemma}

\newtheorem{claim}[theorem]{Claim}
\newtheorem{proposition}[theorem]{Proposition}

\newtheorem{fact}[theorem]{Fact}
\newtheorem{corollary}[theorem]{Corollary}

\theoremstyle{definition}{

\newtheorem{definition}[theorem]{Definition}

\newtheorem*{definition*}{Definition}

\newtheorem{question}[theorem]{Question}
\newtheorem{remark}[theorem]{Remark}

}

\newcommand{\Z}{\mathbb Z}

\newcommand{\cB}{\ensuremath{\mathcal B}}
\newcommand{\cC}{\ensuremath{\mathcal C}}

\newcommand{\cE}{\ensuremath{\mathcal E}}
\newcommand{\cF}{\ensuremath{\mathcal F}}

\newcommand{\cL}{\ensuremath{\mathcal L}}

\newcommand{\cP}{\ensuremath{\mathcal P}}

\newcommand{\cS}{\ensuremath{\mathcal S}}

\newcommand{\cV}{\ensuremath{\mathcal V}}

\newcommand{\llb }{\llbracket}
\newcommand{\rrb }{\rrbracket}
\newcommand{\Ext}{{\mathsf{Ext}}}
\newcommand{\Int}{{\mathsf{Int}}}

\newcommand{\elem}{{\mathsf{el}}}
\newcommand{\out}{{\mathsf{out}}}
\newcommand{\tr}{{\mathsf{tr}}}
\newcommand{\rn}{{\mathsf{rn}}}
\newcommand{\rnelem}{{\mathsf{rn.el}}}

\newcommand{\mac}{{\textrm{mac}}}
\newcommand{\mic}{{\textrm{mic}}}
\newcommand{\dlp}{\mathtt{d}^+}
\newcommand{\dlm}{\mathtt{d}^-}
\newcommand{\dl}{\mathtt{d}}

\newcommand{\gap}{\text{\tt{gap}}}

\newcommand{\tmix}{t_{\textsc{mix}}}

 \renewcommand{\epsilon}{\varepsilon}

\DeclareMathOperator{\diam}{diam}

\DeclareMathOperator{\sgn}{sgn}

\newcommand{\tv}{{\textsc{tv}}}

\makeatletter
\newcommand{\superimpose}[2]{%
  {\ooalign{$#1\@firstoftwo#2$\cr\hfil$#1\@secondoftwo#2$\hfil\cr}}}
  
\newcommand{\sbullet}{%
  \hbox{\fontfamily{lmr}\fontsize{.4\dimexpr(\f@size pt)}{0}\selectfont\textbullet}}

\makeatother

\newcommand{\red}{\textsc{red}}
\newcommand{\blue}{\textsc{blue}}

\title{Metastability cascades and prewetting in the SOS model}

\author{Reza Gheissari}
\address{R.\ Gheissari\hfill\break
Department of Mathematics \\ Northwestern University }
\email{gheissari@northwestern.edu}

\author{Eyal Lubetzky}
\address{E.\ Lubetzky\hfill\break
Courant Institute\\ New York University\\
251 Mercer Street\\ New York, NY 10012, USA.}
\email{eyal@courant.nyu.edu}

\begin{document}

\maketitle

\begin{abstract}
We study Glauber dynamics for the low temperature $(2+1)$D Solid-On-Solid model on a box of side-length $n$ with a floor at height $0$ (inducing entropic repulsion) and a competing bulk external field $\lambda$ pointing down (the prewetting problem).
In 1996, Cesi and Martinelli showed  that if the inverse-temperature $\beta$ is large enough, then along a decreasing sequence of critical points $(\lambda_c^{(k)})_{ k=0}^{K_\beta}$ the dynamics is torpid: its inverse spectral gap is $O(1)$ when $\lambda \in (\lambda_c^{(k+1)},\lambda_c^{(k)})$ whereas it is $\exp[\Theta(n)]$ at each  $\lambda_c^{(k)}$ for each $k\leq K_\beta$, due to a coexistence of rigid phases at heights $k+1$ and $k$. Our focus is understanding (a) the onset of metastability as $\lambda_n\uparrow\lambda_c^{(k)}$; and (b) 
the effect of an unbounded number of layers, as we remove the restriction $k\le K_\beta$, and even allow for $\lambda_n\to 0$ towards the $\lambda = 0$ case which has $O(\log n)$ layers and was studied by Caputo~et~al.~(2014). 
We show that for any $k$, possibly growing with $n$, the inverse gap is $\exp[\tilde\Theta(1/|\lambda_n-\lambda_c^{(k)}|)]$ as $\lambda\uparrow \lambda_c^{(k)}$ up to distance $n^{-1+o(1)}$ from this critical point, due to a metastable layer at height $k$ on the way to forming the desired layer at height $k+1$.
By taking $\lambda_n = n^{-\alpha}$ (corresponding to $k_n\asymp \log n$), this also interpolates down to the behavior of the dynamics when $\lambda =0$. We complement this by extending the fast mixing to all $\lambda$ uniformly bounded away from $(\lambda_c^{(k)})_{k=0}^\infty$. Together, these results provide a sharp understanding of the predicted infinite sequence of dynamical phase transitions governed by the layering phenomenon.  
\end{abstract}

\section{Introduction}\label{sec:introduction}

Consider the low-temperature Solid-On-Solid (SOS) model on $\Lambda_n = \llb -\frac n2 ,\frac n2\rrb^2$, with zero boundary conditions, a floor at height zero, and an external field $\lambda = \lambda_n\ge 0$ pointing downward: The model assigns to a height function $\varphi:\Lambda_n \to \llb 0,n \rrb $ (viewed as a surface inside a cube of side length $n$) the probability 
\begin{align}\label{eq:SOS-measure}
    \mu_{n,\lambda}(\varphi) \, \propto \, \exp\Big( - \beta \sum_{u\sim v} |\varphi_u - \varphi_v| - \lambda \sum_{v} \varphi_v\Big)\,,
\end{align}
where $\beta>0$ is the inverse-temperature, we let $u\sim v$ denote nearest-neighbor adjacency between the vertices $u,v$, and $\varphi_x$ is taken to be $0$ for all $x\notin\Lambda_n$  as per the boundary conditions.
The study of this family of models in the statistical physics literature goes back to Burton, Cabrera and Frank~\cite{BCF51} in 1951 in dimension $(2+1)$ and to Temperley~\cite{Temperley52} in 1952 in dimension $(1+1)$  as models of crystal formation/growth and approximations to the low-temperature plus/minus interface in the Ising model. 

The model at $\lambda=0$ and no floor (the integer height function $\varphi$ can also take negative values) is associated with a \emph{roughening phase transition} exclusive to dimension $2+1$: for some $\beta_{\textsc r}>0$ (numerical simulations suggest $\beta_{\textsc r}\approx 0.806$)
the SOS surface is rough at $\beta\leq\beta_{\textsc r}$ (e.g.,~the variance of $\varphi_o$ diverges with~$n$) whereas it is rigid (the variance is $O(1)$) for $\beta>\beta_{\textsc r}$ (see~\cite{FrSp81a,FrSp81b} and~\cite{BW82} for $\beta\ll 1$ and $\beta\gg1$, resp., and~\cite{Lammers22} proving they form a dichotomy). When setting the surface above a floor (a hard constraint $\varphi\geq 0$) at the low temperature regime ($\beta$ large), it exhibits \emph{entropic repulsion}~\cite{BEF86}: the average height is propelled from $O(1)$ to $\Theta(\log n)$, despite the energetic cost charged along the zero boundary conditions, in order to gain entropy via downward-pointing spikes. When $\lambda = \lambda_n>0$, the field induces an additional downwards force on the interface, inducing a non-trivial competition with the entropic repulsion, known as the \emph{prewetting regime}.

We study Glauber dynamics for the SOS model from \cref{eq:SOS-measure}---a continuous-time single-site Markov chain that, on one hand, gives a simple local recipe for sampling $\mu_{n,\lambda}$, and on the other hand, serves as a natural physical model for the evolution of a random surface towards the SOS measure at equilibrium.
When the SOS surface is concentrated at height $h$ (e.g.,  $h=\Theta(\log n)$ when $\lambda=0$), understanding the rate of convergence of this dynamics to $\mu_{n,\lambda}$ is closely related to the time it takes the surface to successively create $h$ macroscopic layers (each grown out of local droplets) starting, e.g., from the all-zero configuration.

More precisely, the Glauber dynamics assigns every site in $\Lambda_n$ a rate-1 Poisson clock, and every time such a clock rings, it updates the height $\varphi_v$ at the corresponding site $v$ to one of $\{\varphi_v-1 \vee 0,\varphi_v,\varphi_v+1\wedge n\}$, weighted by the conditional probabilities induced on these by $\mu_{n,\lambda}$ given $\{\varphi_u :u\neq v\}$. Said probabilities (and by extension, the update rule at  $v$) depend only on $\{\varphi_u: u = v\mbox{ or }u\sim v\}$ due to the 
nearest-neighbor interactions in \cref{eq:SOS-measure}, and by construction, the chain is reversible w.r.t.\ its stationary distribution $\mu_{n,\lambda}$. A standard way to measure its rate of convergence to equilibrium is the spectral-gap of its generator, denoted $\gap_{n,\lambda}$, which governs the mixing time in $L^2$-distance, and its dependence on the system size $n$ as $n\to\infty$.

With no external field ($\lambda=0$), Caputo et al.~\cite{CLMST14} showed that the inverse gap of
Glauber dynamics for the $(2+1)$D SOS model above a floor satisfies $\gap_{n,\lambda}^{-1} = \exp[\Theta(n)]$, as the surface encounters a sequence of metastable states towards equilibrium: starting from the all-zero state, it takes it time doubly-exponential in~$k$---approximately $\exp[c e^{4\beta k}]$---to create a new macroscopic layer at height $k$, with the bottlenecks at the final levels of height $\sim\frac1{4\beta}\log n$ dominating the mixing time and costing $\exp[c n]$.

In the presence of a bulk external field $\lambda>0$ pointing downward, the competition between the field and the entropic repulsion may lower the typical surface height. For $\lambda>0$ fixed, this lowers the SOS surface to height $k=k(\lambda)$, where as we vary $\lambda\downarrow 0$, the preferred height $k(\lambda)$ of the surface will grow as $\frac{1}{4\beta}\log \frac{1}{\lambda}$. Due to the rigidity of the surface, this induces a sequence of critical points $\lambda_c^{(k)}\approx e^{ - 4\beta k}$ at which the two heights $k+1,k$ are both (equally) stable. (A similar infinite sequence of critical points along which the interface height diverges occurs in the \emph{wetting problem}, where in lieu of the bulk external field $\lambda$ the surface is tilted by $\lambda \sum_{v}\one_{\varphi_v=0}$, rewarding only sites that are pinned to height $0$. A detailed understanding of this was developed by Lacoin~\cite{Lacoin18,Lacoin20} confirming the predicted phenomena~\cite{Chalker}. See also~\cite{Alexander11,FY23} and the excellent survey by Ioffe and Velenik~\cite{IV18} for more on the layering phenomena associated with wetting/pinning.)
In our prewetting framework, at sufficiently low temperatures (large enough $\beta$), an extensive study of the statics and Glauber dynamics for this model\footnote{The setup of~\cite{CeMa1,CeMa2} technically considers $\varphi:\Lambda_n\to \Z_{\geq 0}$, i.e., no upper bound on the maximal height. For $\lambda$ bounded away from~$0$ uniformly in $n$, as they were considering, results with and without a ceiling readily transfer back and forth.} by Cesi and Martinelli~\cite{CeMa1,CeMa2} in 1996 established a finite number of dynamical phase transitions $(\lambda_c^{(k)})_{k=0}^{K_\beta}$.

\begin{theorem*}[\cite{CeMa2}]
Fix $\beta>0$ large enough, and consider
Glauber dynamics w.r.t.\ $\mu_{n,\lambda}$. There is a sequence of critical points $(\lambda_c^{(k)})_{k=0}^{K_\beta}$, where $K_\beta = \lfloor \exp(\frac{\beta}{20000})\rfloor$, such that for every $k=0,\ldots,K_{\beta}$:
\begin{enumerate}[1.]
    \item If $\lambda\in (\lambda_c^{(k+1)},\lambda_c^{(k)})$ then $\gap_{n,\lambda}^{-1}=\Theta(1)$ (moreover, under any boundary condition);
    \item If $\lambda=\lambda_c^{(k)}$ then $\gap_{n,\lambda}^{-1} = \exp[\Theta( n)]$ under free boundary conditions;
\end{enumerate}
where the implicit constants depends on both $\beta$ and $\lambda$.
\end{theorem*}

Our aim here is to 
(a) understand the interpolation between the two behaviors given by~\cite{CeMa1,CeMa2}---exponentially slow mixing at $\lambda=\lambda_c^{(k)}$ as opposed to fast mixing at $\lambda\neq\lambda_c^{(k)}$, as we take $\lambda=\lambda_n\to \lambda_c^{(k)}$; and (b) obtain the complete picture of an \emph{infinite} sequence of dynamical phase transitions by removing the restriction on $K_\beta$, and en route, show the interpolation from the $\lambda$ fixed case to the $\lambda=0$ case studied in~\cite{CLMST14}. 

It ought to be noted that the analysis of the model in~\cite{CeMa1,CeMa2} is quite involved, and itself was built on (and refined) a highly nontrivial study of the model in infinite-volume by Dinaburg and Mazel~\cite{DiMa94} (where a full infinite sequence of critical values  $(\lambda_c^{(k)})_{k=0}^\infty$ for limiting free energies was derived). As mentioned in~\cite{CeMa1} as well as in the survey~\cite[\S4.1.1]{IV18}, it was unclear whether extending the results of~\cite{CeMa1,CeMa2} to $K_\beta\uparrow\infty$ (ideally up to $k= \Theta(\log n)$ where the $\lambda=0$ behavior is observed) was technical or would required different ideas. (An outline of how some of the equilibrium results in~\cite{CeMa1} may be pushed beyond the restriction $k\leq K_\beta$ was sketched by Lebowitz and Mazel~\cite{LeMa96}, albeit still for for $\lambda$ and $k$ fixed independently of $n$.)

\begin{figure}
    \centering
\begin{tikzpicture}
   \node (fig1) at (0,0) {    \includegraphics[width=.9\textwidth, height = 2.5in]{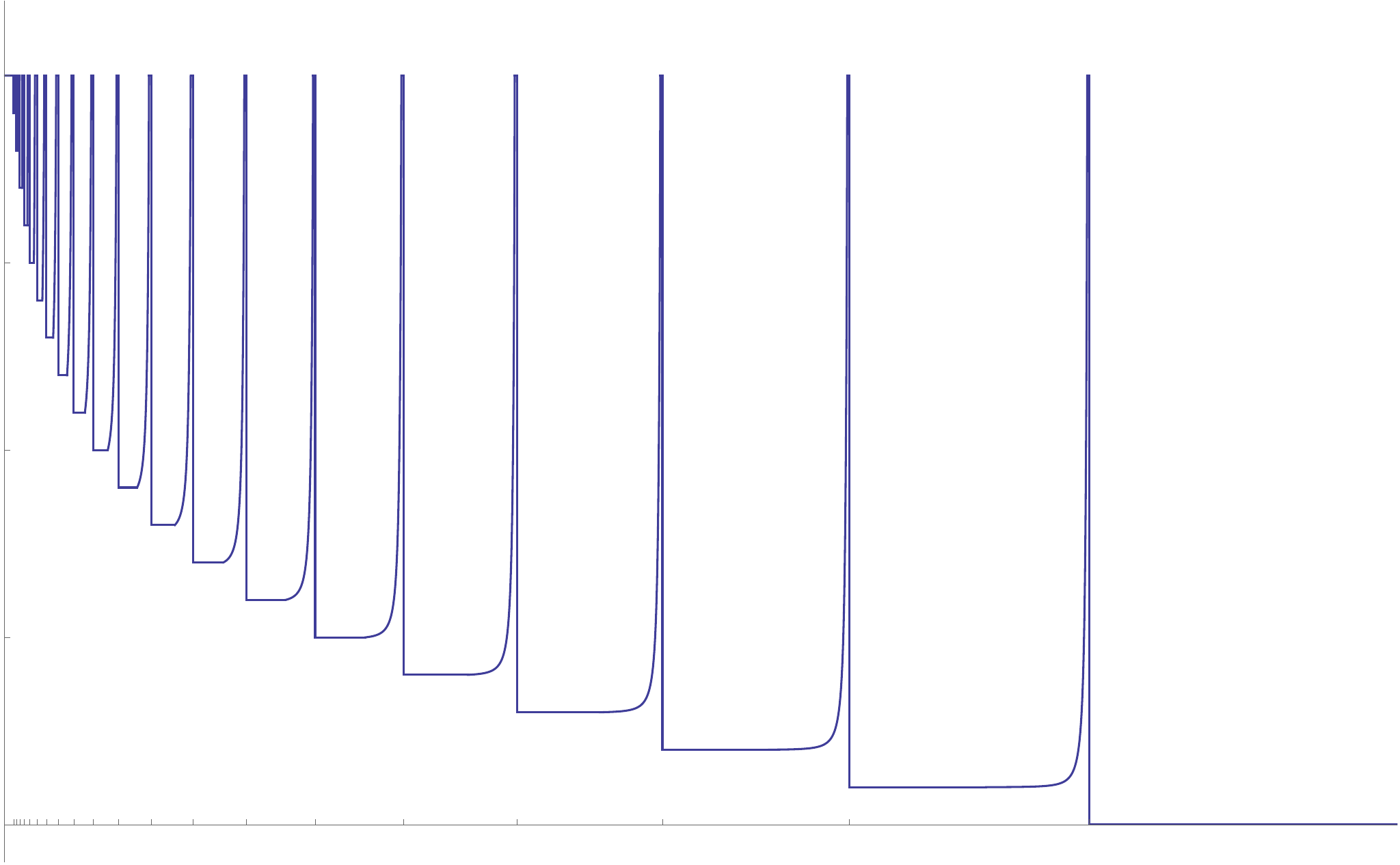}};
   \node[black,font=\small] at (-7.85,2.65) {$e^{n}$};
   \node[black,font=\small] at (-7.85,0) {$e^{\sqrt n}$};
   \node[black,font=\small] at (-7.85,-1.5) {$e^{n^{\alpha}}$};

        \node[black,font=\small] at (-7.85,1.3) {$\vdots$};
        \node[black,font=\small] at (-7.85,-2) {$\vdots$};

     \node[black,font=\small] at (4.25,-3.15) {$\lambda_c^{(k)}$};
     \node[black,font=\small] at (1.7,-3.15) {$\lambda_c^{(k+1)}$};
   \node[black,font=\small] at (-0.3,-3.15) {$\cdots$};
     \node[black,font=\small] at (-1.8,-3.15) {$\cdots$};
        \node[black,font=\small] at (-5.5,-3.2) {$\cdots$};

     \node[black,font=\small] at (-3.95,-3.15) {$\lambda_c^{{\tiny{(\frac{\alpha}{4\beta}\log n)}}}$};
     \node[black,font=\small] at (-7.15,-3.15) {$\frac{1}{n}$};
     \node[black,font=\small] at (-6.45,-3.18) {$\frac{1}{\sqrt{n}}$};

   \end{tikzpicture}
   \vspace{-0.1in}
    \caption{The inverse gap $\gap_{n,\lambda_n}^{-1}$ under zero boundary conditions plotted against~$\lambda$, demonstrating the diverging sequence of metastability cascades as $\lambda\downarrow 0$. The slowdowns near the critical points are expected to be asymmetrical due to the boundary condition.}
    \label{fig:zero-bc-expected}
\end{figure}
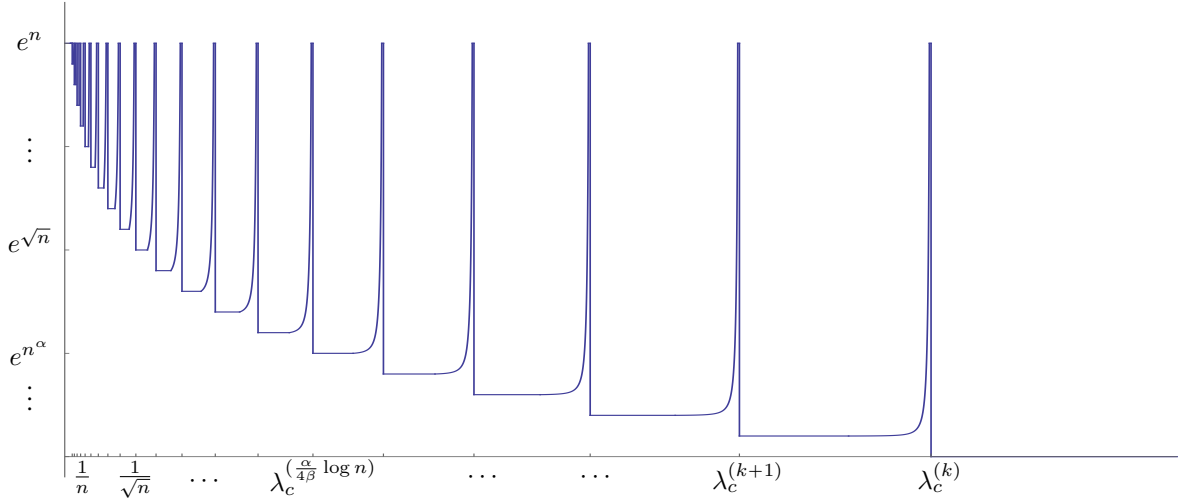

Our main theorem establishes the full infinite sequence of metastability windows of $\lambda$ about $(\lambda_c^{(k)})_{k=0}^\infty$, with the distance of  $\lambda_n$ to the nearest critical point determining the rate of the exponential slowdown. Namely, the rate of metastability is determined by the time it takes a surface that is predominantly at height $k$ to generate a critical droplet at the desired height $k+1$. The size needed for the droplet to be stable is in turn determined by an isoperimetric tradeoff and the difference of $\lambda$ to $\lambda_c^{(k)}$. (A similar metastability induced by the time to grow a critical droplet is seen in the dynamics of a low-temperature 2D Ising model with $(+)$-boundary and a competing small $(-)$ external field, as was studied in~\cite{SchShl-Ising-field-dynamics-2D}: see also~\cite{Sch-Ising-field-dynamics,Sch-Ising-field-dynamics-survey}). 
\begin{theorem}\label{mainthm:dynamical}
    Consider the SOS Glauber dynamics at $\beta>\beta_0$ large enough and a bulk external field $\lambda_n>0$ with zero boundary as per \cref{eq:SOS-measure}. There exists a sequence of critical points $(\lambda_c^{(k)})_{k=0}^\infty$ such that the following holds. Let \begin{equation}\label{eq:dl}
    \dlp(x) := \min_k\{ (\lambda_c^{(k)}-x): \lambda_c^{(k)}\ge x\}\,, \quad \text{and} \quad \dl(x) := 
 \min_k\{ |\lambda_c^{(k)}-x|\}\,.
\end{equation}
For every $\epsilon>0$ there exist constants $c(\epsilon),C(\beta,\epsilon)>0$ such that for all $n$,    
        \begin{enumerate}
    \item \label{it-thm-1-O(1)} If $\lambda_n$ is such that $\dl(\lambda_n)>\epsilon$ then $C^{-1}\le \gap_{n,\lambda_n}^{-1}\leq C$.
    
    \item \label{it-thm-1-O(1/f)} If $\lambda_n$ is such that $\frac{c\beta \log n}{n}\leq \dlp(\lambda_n)\leq \epsilon$ then\footnote{This restriction can in fact be replaced with $\frac{c\beta k}{n}\le \dlp(\lambda_n)\le \epsilon$ where $k= \operatorname{argmin}_k\{ (\lambda_c^{(k)}-x): \lambda_c^{(k)}>x\}$.} 
        \begin{equation}\label{eq:mainthm-gap-1/f} \frac{1}{C} \vee \frac{1}{n^2}\exp\Big[ \frac{1}{C\dlp(\lambda_n)}\Big] \leq \gap_{n,\lambda_n}^{-1} \leq \exp\Big[\frac{C(\log n)^3}{\dl(\lambda_n)}\Big]\,.
        \end{equation}
    \end{enumerate}
\end{theorem}

The following corollary sharply characterizes the onset of metastability in the two regimes which served as our earlier motivation: (a) as $\lambda_n\uparrow \lambda_c^{(k_n)}$, and (b) as $\lambda_n\downarrow 0$. 

\begin{corollary}\label{maincor:n^alpha}
In the setting of \cref{mainthm:dynamical}, for any fixed $0<\alpha<1$, in either of the following situations:
\begin{enumerate}[(a)]
    \item\label{it:cor-fixed-k} $\lambda_n = \lambda_c^{(k)} - \Theta(n^{-\alpha})$ for any $1\le k\le o(\log n)$,
    \item \label{it:cor-0} $\lambda_n = \Theta(n^{-\alpha})$ such that $\dl(\lambda_n)\gtrsim n^{-\alpha}$ (e.g., the sequence $\lambda_n = e^{-\beta - 4\beta k_n}$ for $k_n = \lfloor \frac{\alpha}{4\beta} \log n\rfloor$), 
\end{enumerate}
we have $\gap_{n,\lambda_n}^{-1} = \exp[\tilde\Theta(n^{\alpha})]$.
\end{corollary}

In \cref{it:cor-fixed-k} of \cref{maincor:n^alpha}, the typical interface will be localized at height $k+1$, but under $k$ boundary conditions, only height-$(k+1)$ layers of size greater than $n^{\alpha}$ are desirable (smaller ones do not have enough room for the entropic repulsion to propel them to height $k+1$). As the dynamics climbs from below, say initialized at all-zero, this generates a bottleneck, whereby there is an exponential in $n^{\alpha}$ waiting time to randomly form such an atypically large droplet. In \cref{it:cor-0}, initialized from all-zero, there is a cascade of metastable transitions with times exponential in $(\lambda_c^{(k)}- \lambda_n)_{k\ge 0}$ to go from $k$ to $k+1$, with the final one to get to $k_n=\lfloor\frac{\alpha}{4\beta}\log n\rfloor$ dominating the mixing time. See \cref{fig:zero-bc-expected} for a rough depiction of the above.  

As evident from the fact that \cref{eq:mainthm-gap-1/f} features $\dlp(\lambda_n)$ in the lower bound yet $\dl(\lambda_n)$ in the upper bound (measuring the distance to the closest critical point $\lambda_c^{(k)}$ in the latter yet only to the closest $\lambda_c^{(k)}$ above $\lambda_n$ in the former), we expect the mixing time to be asymmetric in the window about each $\lambda_c^{(k)}$ (see \cref{fig:zero-bc-expected}). Indeed, when $\dl(\lambda_n) \ll \dlp(\lambda_n)$, e.g., $\lambda_n\downarrow \lambda_c^{(k)}$, the typical interface is rigid about height $k$, but in a region with boundary conditions $k+1$, only height-$k$ layers of size greater than $\dl(\lambda_n)^{-1}$ will be desirable. This would induce an exponential slowdown if the boundary conditions were above height $k$, but the height-$0$ boundary conditions enable the demolition of the height-$k$ layer from the outside inwards (via droplet shrinkage). This dynamical picture---in particular, if we identify heights $\{\leq k\}$ as minus spins and heights $\{>k\}$ as plus spins---is in analogy with (and more complicated than) the dynamics of the low-temperature Ising model with minus boundary conditions, initialized from all-plus (possibly with a small negative field of $\dl(\lambda_n)$). The latter is a notoriously hard open problem in the mixing time literature: see~\cite{LMST} for the best known bounds. In particular, for fixed $k$, exactly at $\lambda_c^{(k)}$, and with zero boundary conditions, one would expect the mixing time to be polynomially fast, with a $\Theta(n^2)$ inverse gap.  We discuss the delicate expected behavior in this regime in more detail in \cref{subsec:open-problems}. 

When considering the SOS model under \emph{periodic} (instead of zero) boundary conditions---that is, $\Lambda_n$ in \cref{eq:SOS-measure} is replaced by $(\mathbb Z/n\mathbb Z)^2$---the above asymmetry is no longer present, as illustrated by the next theorem. (See \cref{fig:torus-mixing-time} for a depiction of the bounds of \cref{mainthm:dynamical-torus} symmetric about $(\lambda_c^{(k)})_{k=0}^{\infty}$.) 

\begin{theorem}\label{mainthm:dynamical-torus}
    In the setting of \cref{mainthm:dynamical} but with periodic boundary conditions, so long as $\lambda_n \ge \frac{\log n}{n}$, for every $\epsilon>0$ there exists $C(\beta,\epsilon)$ such that for all $n$, 
    \begin{enumerate}
        \item \label{it-torus-O(1)} If $\lambda_n$ is such that $\dl(\lambda_n)\ge \epsilon$, then $C^{-1} \le \gap^{-1}_{n,\lambda_n}\le C$.
        \item \label{it-torus-O(1/f)} If $\lambda_n$ is such that $\dl(\lambda_n) \le \epsilon$, then 
        \begin{align*}
           \frac{1}{C} \vee \frac{1}{n^2} \exp\Big[\frac{1}{C} \Big( n \wedge \frac{1}{\dl(\lambda_n)}\Big)\Big] \le \gap_{n,\lambda_n}^{-1}\le \exp\Big[C \Big( n \log n \wedge \frac{(\log n)^3}{\dl(\lambda_n)}\Big)\Big]\,.
        \end{align*}
    \end{enumerate}
\end{theorem}

\begin{figure}
    \centering
    \vspace{-0.1in}
    \begin{tikzpicture}
   \node (fig1) at (0,0) {
\includegraphics[width=0.9\textwidth, height = 2.5in]{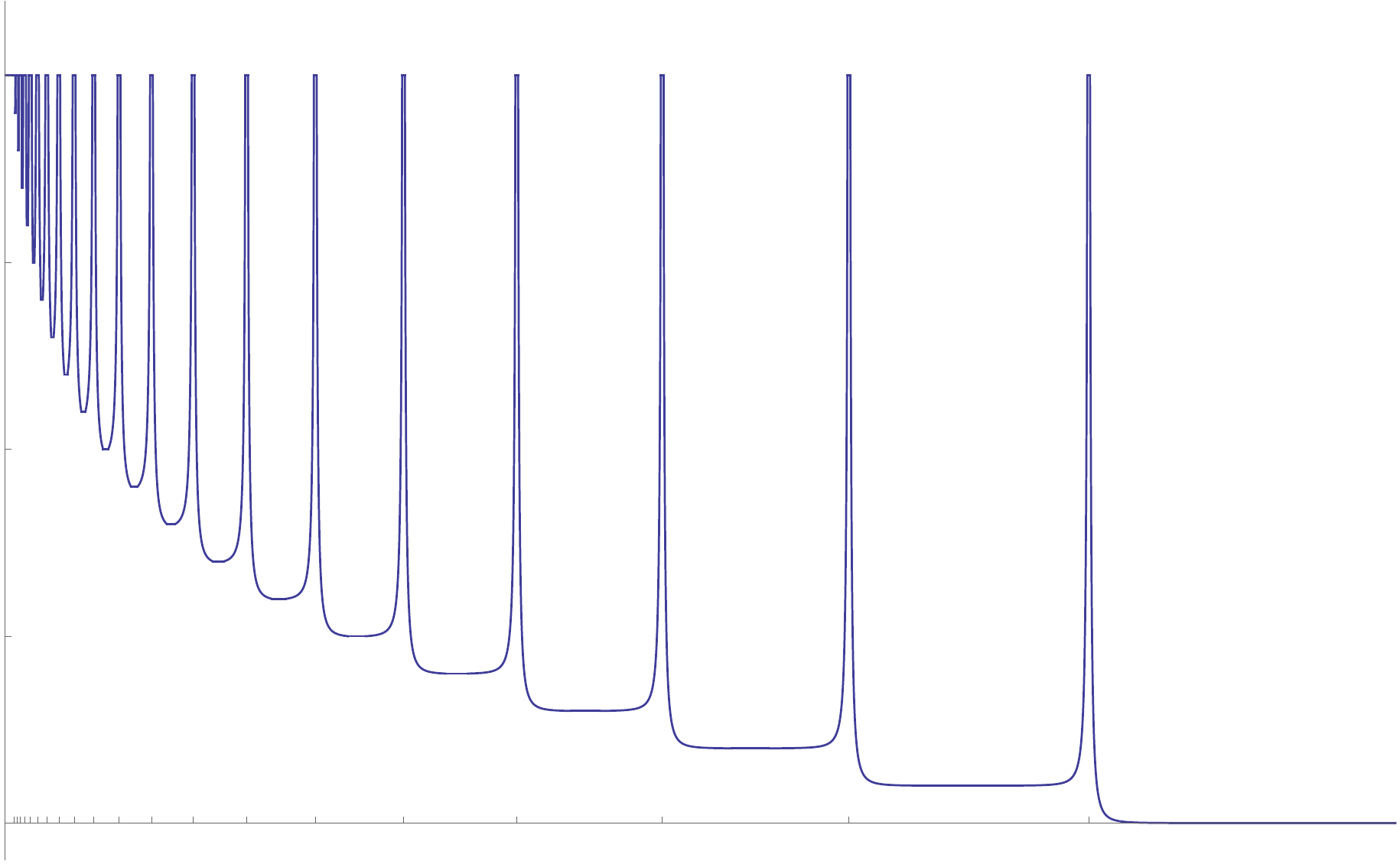}};
   \node[black,font=\small] at (-7.85,2.65) {$e^{n}$};
   \node[black,font=\small] at (-7.85,0) {$e^{\sqrt n}$};
   \node[black,font=\small] at (-7.85,-1.5) {$e^{n^{\alpha}}$};

        \node[black,font=\small] at (-7.85,1.3) {$\vdots$};
        \node[black,font=\small] at (-7.85,-2) {$\vdots$};

     \node[black,font=\small] at (4.25,-3.15) {$\lambda_c^{(k)}$};
     \node[black,font=\small] at (1.7,-3.15) {$\lambda_c^{(k+1)}$};
   \node[black,font=\small] at (-0.3,-3.15) {$\cdots$};
     \node[black,font=\small] at (-1.8,-3.15) {$\cdots$};
        \node[black,font=\small] at (-5.5,-3.2) {$\cdots$};

     \node[black,font=\small] at (-3.95,-3.15) {$\lambda_c^{{\tiny{(\frac{\alpha}{4\beta}\log n)}}}$};
     \node[black,font=\small] at (-7.15,-3.15) {$\frac{1}{n}$};
     \node[black,font=\small] at (-6.45,-3.18) {$\frac{1}{\sqrt{n}}$};
\end{tikzpicture}
\vspace{-0.15in}
    \caption{The inverse gap $\gap_{n,\lambda_n}^{-1}$ for the model under periodic boundary conditions plotted against~$\lambda$. The critical slowdowns are symmetric, as per the bounds of \cref{mainthm:dynamical-torus}.}
    \vspace{-0.05in}
    \label{fig:torus-mixing-time}
\end{figure}
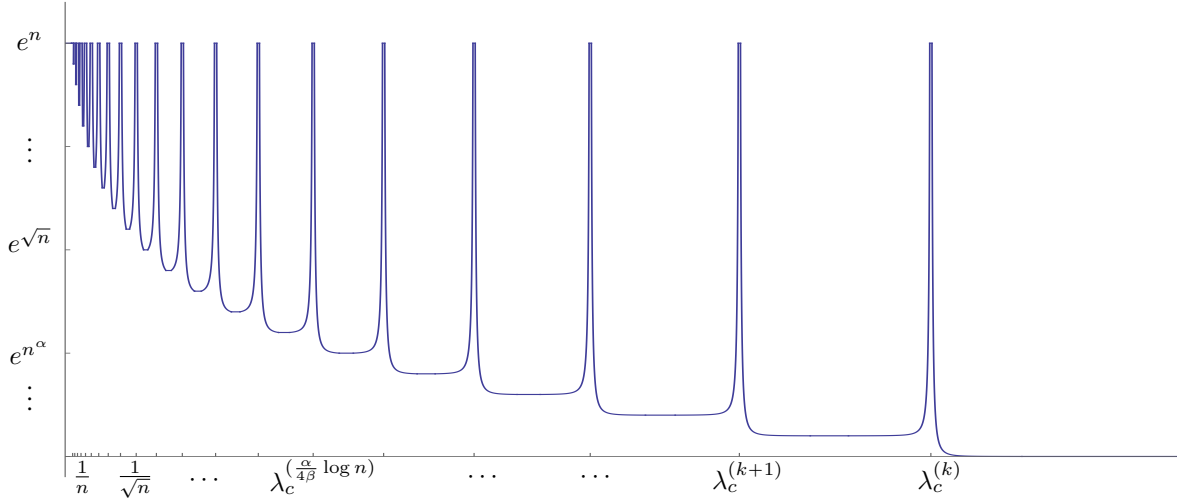

The case of free boundary conditions, which was the setting in which~\cite{CeMa2} showed $\exp[\Theta(n)]$ mixing at $(\lambda_c^{(k)})_{k=0}^{K_\beta}$, is similar to the torus in that the boundary conditions do not favor one of $\{k,k+1\}$. If $\lambda_n$ is uniformly bounded away from $0$, the factor of $\log n$ in the upper bound of \cref{it-torus-O(1/f)} of \cref{mainthm:dynamical-torus} can similarly be dropped (see \cref{rem:getting-rid-of-the-log}), thus recovering for a fixed $\lambda$, bounded away from $0$, the same bounds for the torus as obtained in~\cite{CeMa2} for free boundary conditions. Moreover, \cref{mainthm:dynamical-torus} allows $\lambda$ to be taken arbitrarily close to $0$, as well as $\lambda_n$ approaching $0$ or a critical point as~$n\to\infty$. While we did not pursue this, it should be possible to modify the arguments of this paper to the case of free boundary conditions.

We conclude this part with a brief comment on other possible boundary conditions. In analogy to the low-temperature 2D Ising model, the metastable windows can be quite sensitive to the choice of boundary conditions. If the boundary conditions were high, rather than at $0$, they would favor height $k+1$ over $k$ near $\lambda_c^{(k)}$, and \cref{it-thm-1-O(1/f)} of \cref{mainthm:dynamical} would apply with $\dlp$ replaced by $\dlm = \min_k \{(x-\lambda_c^{(k)}): \lambda_c^{(k)}\le x\}$. Dobrushin-type boundary conditions, say $\{\le k\}$ on half the boundary and $\{> k\}$ on the other half, should induce interfaces and rich static and dynamic behaviors: see the discussion of these interfaces as $\dl(\lambda_n)\downarrow 0$ in \cref{subsec:open-problems}. 
Finally, the $O(1)$ inverse gap bound of \cref{it-thm-1-O(1)} in \cref{mainthm:dynamical} applies for all fixed boundary conditions, so long as $n$ is sufficiently large as a function of $\epsilon$, extending the fast mixing results of \cite{CeMa2} (which needed $\lambda>\lambda_c^{(K_\beta)}$) to hold for $\lambda$ arbitrarily close to $0$ (while kept $\Omega(1)$ in $n$). 

\subsection{Proof ideas}\label{subsec:proof-ideas}
The proofs in our paper are divided into two parts, the equilibrium portion comprising \cref{sec:elementary,sec:infinite-sequence-lambda,sec:spatial-mixing} and the dynamical portion comprising \cref{sec:mixing-time-upper-bounds,sec:mixing-time-lower-bounds}. 

Our analysis of the equilibrium estimates differs from the works \cite{DiMa94,CeMa1} first and foremost via its focus on \emph{level lines} (viewing the SOS configuration as a contour ensemble, each associated to a level using appropriate boundary conditions). The aforementioned previous works on the $\lambda>0$ setting studied \emph{cylinders} as the main objects---maximal stacks of contours (all having the same projection on $\Z^2$). Operations on contours (such as Peierls-type maps that shift a single level line up or down) can in general allow a much more refined analysis than operations on cylinders, as was demonstrated in the work~\cite{CLMST14} for the $\lambda=0$ case. 
This different perspective does introduce additional complications, yet is essential in our extension of the previous results to any height $k$ (as opposed to $k\leq K_\beta$). 
Namely, by combining probabilistic arguments on the set of contours (based on coarse graining, monotonicity, and Peierls maps) with the (already quite sophisticated) cluster-expansion based techniques from \cite{DiMa94,CeMa1} (adapted from cylinders to our setting of contours in \cref{sec:elementary,sec:infinite-sequence-lambda}), we arrive at the following understanding of the shape of an SOS surface with zero boundary conditions in the presence of an external field $\lambda$. Note the subtle relation between the height $k$ the surface reaches, the distance $\dlp$ of the external field to its nearest critical point, and the domain size~$n$.  

\begin{theorem}\label{mainthm:equilibrium}
For every $\beta>\beta_0$ there exist some $C>0$, an $\epsilon_\beta$ (going to zero as $\beta \to\infty$), and an infinite sequence of critical points $(\lambda_c^{(k)})_{k=0}^\infty$ such that, if $\lambda_n\geq \frac {\log n}n$ satisfies $\lambda_n\in (\lambda_c^{(k_n+1)},\lambda_c^{(k_n)})$ 
    then with $\mu_{n,\lambda_n}$-probability at least $1-e^{ - n/C}$, 
    \begin{enumerate}
        \item if $n\ge C \beta k_n/\dlp(\lambda_n)$, all but an $\epsilon_\beta$ fraction of the sites will be at height $k_n+1$.
        \item If $n\le 1/(C\dlp(\lambda_n))$, all but an $\epsilon_\beta$ fraction of the sites in $\Lambda_n$ will be at height $k_n$.
    \end{enumerate}
\end{theorem}

We in fact prove stronger equilibrium results in \cref{sec:spatial-mixing}, showing \emph{spatial mixing} estimates between SOS models with different boundary conditions that hold on boxes of side-length $n_0 \gtrsim k_n/\dl(\lambda_n)$, yet fail when $n_0 \lesssim 1/\dl(\lambda_n)$.
Note that the analysis of contours, beyond its effectiveness in capturing the behavior of the surface in regimes of $\lambda$ unaddressed by~\cite{CeMa1}, opens the door to a host of questions addressing the refined geometry of the individual level lines (notably the top one): e.g., understanding their macroscopic scaling limits and the random fluctuations around those as $\lambda$ approaches $\lambda_c^{(k)}$; see \cref{subsec:open-problems} for concrete examples.

Turning to the dynamical arguments, our proofs of the exponential lower and upper bounds are guided by the following mechanism for equilibration, which we describe here for the case of zero boundary condition and $\lambda_n = \lambda_c^{(k_n)} - \dlp(\lambda_n)$. If we let $n\ge n_0 = C\beta k_n/\dlp(\lambda_n)$, then on a domain of size $n$, the equilibrium surface is predominantly at height $k_n + 1$ per \cref{mainthm:equilibrium}. On the other hand, if the dynamics is initialized at heights $\{\le k_n\}$ it needs to grow a layer at height $k_n +1$ of size at least $n_0$ until that layer would become thermodynamically stable and could grow to encompass a $(1-\epsilon_\beta)$-fraction of $\Lambda_n$ (if the layer has diameter smaller than $n_0$, the exponential cost in its boundary dominates the free energy benefit of the area it confines being at height $k_n+1$ over $k_n$). Since the Glauber dynamics makes single-site updates, it in particular needs to first form a layer of diameter between $n_0/2$ and $n_0$, an event that has probability $e^{ - n_0/C}$. Once such a droplet has been formed, without waiting much longer, it will have formed around every site in $\Lambda_n$, and therefore the $k_n+1$ layer will in fact encompass almost all of $\Lambda_n$, leading to rapid equilibration. 

For the lower bounds on $\gap_{n,\lambda_n}^{-1}$, the above intuition is used to construct a bottleneck event for Glauber dynamics (see \cref{sec:mixing-time-lower-bounds}). For the upper bound, in \cref{sec:mixing-time-upper-bounds}, we guide the dynamics towards equilibrium, using censoring and monotonicity to reduce the mixing time on $\Lambda_n$ to the mixing time on graphs of cut-width at most $n_0$, where $n_0$ is the smallest scale at which the model exhibits weak spatial mixing. While reductions of mixing times to weak spatial mixing are common in the literature (e.g., the classical work~\cite{MaOl1}), boundary conditions pose difficulties when only weak spatial mixing, but not strong spatial mixing, holds. We handle this difficulty by also showing weak spatial mixing estimates in annuli of thickness $n_0$, whose mixing time will also only be exponential in $n_0$, and which is used to couple near the boundary.   

Finally, when $\lambda$ is uniformly bounded away from $(\lambda_c^{(k)})_{k=0}^{\infty}$ (i.e., $\dl(\lambda_n)\ge \epsilon>0$) the $O(1)$ bound on $\gap^{-1}_{n,\lambda_n}$ follows similarly to~\cite{CeMa2}; first, we use that in 2D systems, weak spatial mixing implies strong spatial mixing~\cite{MOS-WSM-SSM}, then we use that strong spatial mixing implies an $O(1)$ inverse gap. Note that the dependence of this bound on $1/\dl(\lambda_n)$ blows up in a doubly exponential manner whereas the upper bounds of \cref{it-thm-1-O(1/f)} of \cref{mainthm:dynamical} and \cref{it-torus-O(1/f)} of \cref{mainthm:dynamical-torus} are only exponential.

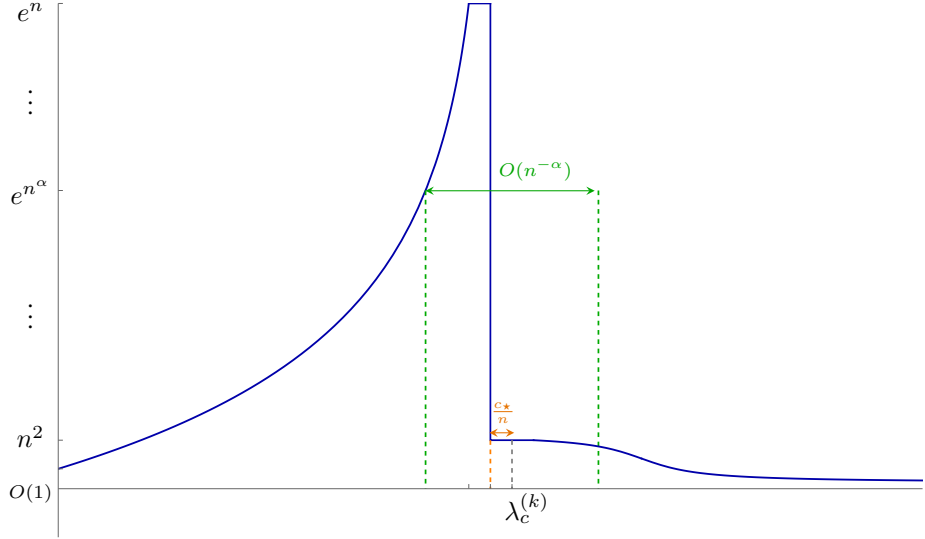
\begin{figure}
        \centering
        \vspace{-0.15in}
\begin{tikzpicture}
   \node (fig1) at (0,0) {    \includegraphics[width=.7\textwidth]{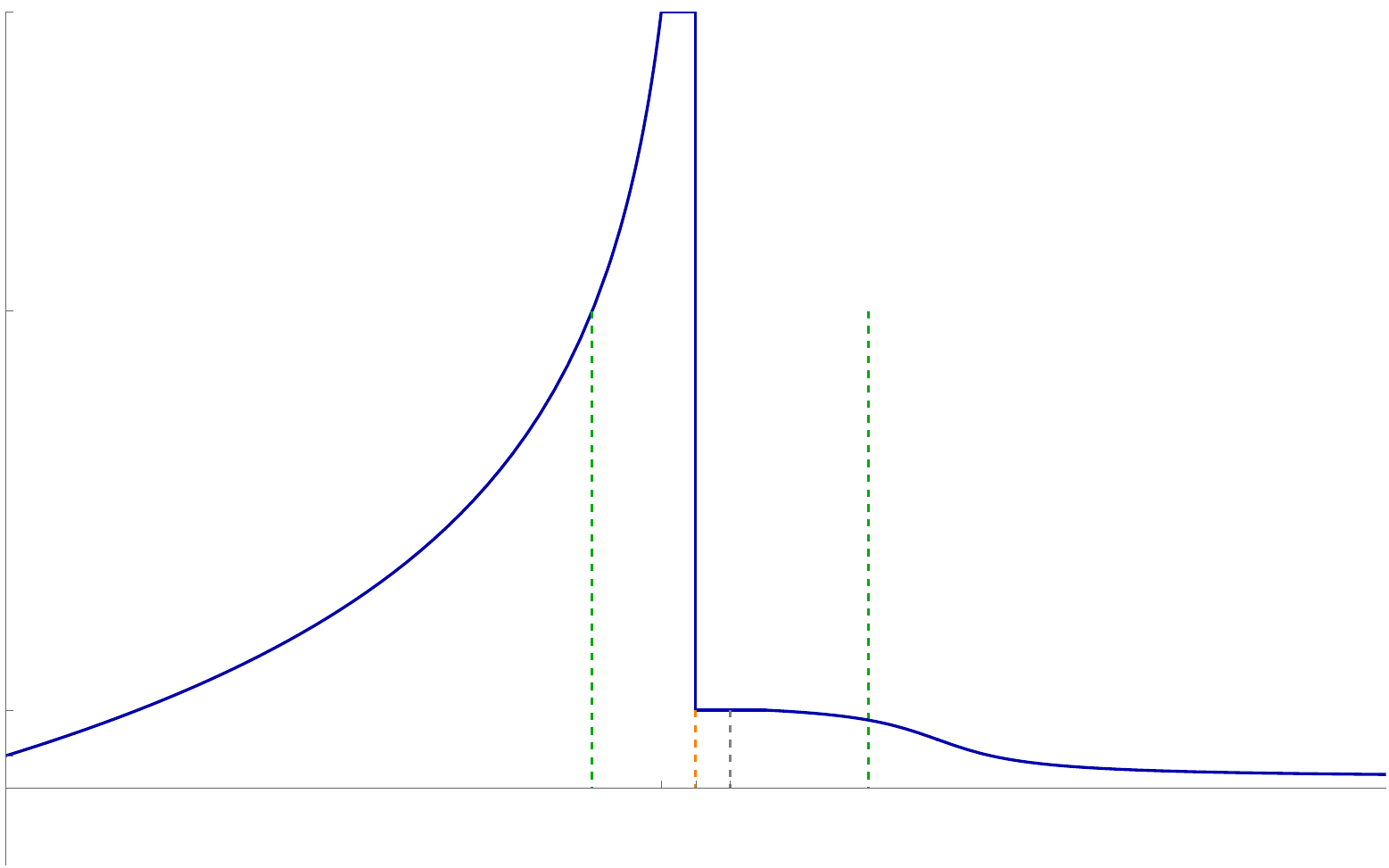}};
     \node[black,font=\small] at (.5,-3.2) {$\lambda_c^{(k)}$};
        \node[black,font=\small] at (-6.1,3.4) {$e^{n}$};
   \node[black,font=\small] at (-6.1,1) {$e^{n^{\alpha}}$};
      \node[black,font=\small] at (-6.1,-2.25) {$n^2$};
      \node[black,font=\tiny] at (-6.1,-3) {$O(1)$};

        \node[black,font=\small] at (-6.1,2.3) {$\vdots$};
        \node[black,font=\small] at (-6.1,-.55) {$\vdots$};

        \draw[stealth-stealth,green!65!black,font=\tiny] (-.85,1.02) -- node[pos=0.65,above] {$O(n^{-\alpha})$} (1.4,1.02);

        \draw[stealth-stealth,orange!90!black,font=\tiny] (0,-2.2) -- node[pos=0.55,above] {$\frac{c_\star}n$} (.33,-2.2);
        
 \end{tikzpicture}
   \vspace{-0.2in}
    \caption{The predicted behavior of the inverse gap with zero boundary conditions when $\lambda_n = \lambda_c^{(k)}\pm o(1)$ for fixed $k\geq 0$. \Cref{mainthm:dynamical} establishes the curve left of the discontinuity (up to poly-logarithmic factors); the right side of the curve should be governed by stable layers shrinking via mean-curvature flow, leading to an $O(n^2)$ relaxation time at $\lambda_c^{(k)}$.}
    \label{fig:phase-transition-window}
    \vspace{-0.15in}
\end{figure}

\subsection{Open problems}\label{subsec:open-problems}

\subsubsection*{The critical window under zero boundary conditions}
Consider the dynamical phase transition about $\lambda_c^{(k)}$ for some fixed $k\geq 0$.
Our main results established the full dynamical phase transition in $\log\gap_{n,\lambda_n}^{-1}$ (up to a poly-log($n$) factor) in the torus (\cref{mainthm:dynamical-torus}). However, as noted above, the case of zero boundary conditions is significantly more involved, and our results in \cref{mainthm:dynamical} are only sharp for $\lambda$ below $\lambda_c^{(k)}$; we expect that the difference between the two settings would be apparent already when viewing $\lambda=\lambda_c^{(k)}$, where the inverse-gap under zero boundary conditions should no longer be $\exp[\Theta(n)]$, but instead have order $n^2$ (the conjectured~\cite{FisherHuse}  order of the inverse-gap 
in Ising under all-plus boundary conditions). More precisely, we would expect that as $\lambda_n\uparrow \lambda_c^{(k)}$, the regime where the inverse-gap is $\exp[\Theta(n)]$---which is reached as $\lambda_n$ draws within distance $O(1/n)$ from $\lambda_c^{(k)}$---
would terminate discontinuously at $\lambda_n=\lambda_c^{(k)}-c_\star/n$ for some $c_\star(k,\beta)$, reminiscent of the critical point appearing in the work of Schonmann and Shlosman~\cite{SchShl-Ising-field-dynamics-2D} on Glauber dynamics for the 2D Ising model under an external field.
At this $\lambda_n$, and as $\lambda_n$ further increases up to $\lambda_n = \lambda_c^{(k)}+O(1/n)$, the inverse-gap should be polynomial in $n$, and thereafter should smoothly drop to $O(1)$ as $\lambda_n-\lambda_c^{(k)}$ increases to a constant order (see a possible depiction of this behavior in \cref{fig:phase-transition-window}). A first step towards verifying this conjectured behavior would be to establish sub-exponential mixing at $\lambda=\lambda_c^{(k)}$.

\subsubsection*{Ferrari--Spohn limit in critical prewetting}
Recall that the SOS model above a floor with $\lambda=0$ is known to have a typical height of order $\log n$; it is also known~\cite{CLMST12,CLMST14,CLMST16} to be typically rigid about a single level, where the corresponding level-line contour has cube-root fluctuations around its deterministic scaling limit. In particular, as said limit coincides with the sides of the box near their centers, it is believed that in the setting of~\cref{eq:SOS-measure} with $\lambda=0$, taking the intersection of the top level-line with an $n^{2/3}\times n^{2/3}$ box $R_n$ about the center of the bottom boundary of $\Lambda_n$, and then rescaling its width by $n^{2/3}$ and its height by~$n^{1/3}$, should give a nontrivial random process in the $n\to\infty$ limit. Unfortunately, the attraction/repulsion effects between the $\log n$ level lines, as well as  against the boundary of $\Lambda_n$, pose a formidable challenge to the analysis of the model. So far it is only known that, in an $n^{2/3+\epsilon}\times n^{2/3+\epsilon}$ box $R'_n$ about the center of the bottom boundary,
 the maximum vertical distance of the top level line from $\partial\Lambda_n$ 
is at least $n^{1/3-\epsilon'}$ and at most~$n^{1/3+\epsilon'}$ (the aforementioned cube-root fluctuations).
Continuous $(1+1)$D models approximating this ensemble of SOS level lines in $R_n$ have been studied in detail, notably the ensemble of non-crossing Brownian excursions with an area tilt~\cite{CIW19a,CIW19b,DLZ23,CG23}, where the area $A_k$ under the $k$-th curve tilts the probability by a factor of $\exp(-\mathfrak{b}^k A_k)$ for some $\mathfrak b>1$ (in the SOS model with $\lambda =0$, the entropic repulsion reward becomes stronger at the level $k$ decreases, corresponding to a factor $\mathfrak b = e^{4\beta}$).
Note that a single Brownian excursion with an area tilt has the law of a Ferarri--Spohn~\cite{FS05} diffusion.
In the discrete setting, the longstanding prewetting problem in the 2D Ising model (akin to the SOS problem studied here, yet involving only a single level line) was very recently settled: the cube-root fluctuations of $n^{1/3+o(1)}$ at the critical scale $\lambda_n=\Theta(1/n)$ were established by Velenik~\cite{Velenik04} in 2004; Ganguly and the first author~\cite{GaGh20} established tightness of the fluctautions at scale~$n^{1/3}$; and the Ferrari--Spohn limit law was finally established by Ioffe et al.~\cite{IOSV21}.

In light of these results, and given the new results from this work, it would be interesting to examine the microscopic fluctuations of the top level line of the SOS model at $\lambda_n=\lambda_c^{(k)}-\Theta(1/n)$, corresponding to critical prewetting between levels $k$ and~$k+1$. There, despite the existence of $k+1$ interacting contours, one would expect the behavior to be effectively dictated by the top contour alone (as this value of $\lambda$ is bounded away from the critical points corresponding to  lower level lines), making it  more amenable to analysis.

\begin{question}
Fix $\beta>\beta_0$, an integer $k\geq 0$ and some constant $c>0$, and consider the $(2+1)$D SOS measure $\mu_{n,\lambda_n}$ from \cref{eq:SOS-measure} for $\lambda_n = \lambda_c^{(k)}-c/n$.
Let $R_n$ be the square whose bottom edge overlaps the bottom boundary $\partial\Lambda_n$ along  $\llb-n^{2/3},n^{2/3}\rrb$, take the ($k+1$)-level line of SOS restricted to $R_n$, and rescale its width by $n^{2/3}$ and height by $n^{1/3}$. Does its law converge to a Ferrari--Spohn diffusion on~$[-1,1]$?
\end{question}
A variant of this question, which is somewhat simpler yet should still exhibit a Ferarri--Spohn limit, is to consider  $\mu_{n,\lambda_n}$ under boundary conditions at height $k$ (as opposed to zero) for $\lambda_n = \lambda_c^{(k)}-c/n$. That setting should be essentially the same as the case of $\lambda_c^{(0)}$, where there is typically at most $1$ macroscopic level line. A useful first step towards these questions would be to understand the microscopic features of the $k$-to-$(k+1)$ contours at $\lambda_c^{(k)}$, e.g., \'{a} la the tools developed for Ising interfaces in \cite{DKS} and \cite{PV97}.

\subsubsection*{SOS near the roughening point and 3D Ising}

We conclude with two challenging fronts. First, in light of the recently obtained~\cite{Lammers22} sharpness of the roughening transition for the SOS model, one would want to extend the above results to all $\beta>\beta_{\textsc r}$. For $\beta$ so close to the roughening point, not only do cluster expansion techniques (used in this paper) become non-applicable, so do the Peierls maps on contours (which were the starting point of the entropic repulsion analysis of~\cite{CLMST14} for instance). Second, given recent progress on understanding the entropic repulsion phenomenon for the Ising interface~\cite{GL20,GL21}, one would hope to establish a similar metastability induced by layering for the Ising Glauber dynamics.  To that end, a prerequisite would be to pinpoint the typical height of the interface in the prewetting problem for 3D Ising.

\section{Notation and preliminaries}\label{sec:prelim}
In this section, we introduce much of the notation we will use throughout the paper. We also introduce the basic contour representation of SOS configurations, and include a few basic but important estimates that follow from the monotonicity of the model in the external field. 

\subsection{Notational disclaimers}
Throughout the paper, we imagine $\beta_0$ to be a large absolute constant, so that $\beta>\beta_0$ indicates $\beta$ sufficiently large (independent of everything else). We are only interested in $\lambda \le 1$, say, since for $\beta$ large, even the first critical point $\lambda_c^{(0)}$ will be less than $1$, and monotonicity arguments can easily reduce analysis of larger $\lambda$ to $\lambda =1$. Finally, we use $C$ in various place to indicate an absolute constant, and $\epsilon_\beta$ a sequence going to zero as $\beta\to\infty$, both of which may change from line to line. 

\subsection{Underlying geometry}
The underlying geometry we work with are subsets of $\mathbb Z^2$, with edge sets given by subsets of nearest neighbor edges $E(\mathbb Z^2)$. For a domain $V \subset \mathbb Z^2$, we use $E(V)$ to denote the edges in $E(\mathbb Z^2)$ between two vertices in $V$. When these subsets are squares, we use the notation $\Lambda_n = \llb -\frac{n}{2},\frac{n}{2}\rrb^2$ with corresponding edge-set $[-\frac{n}{2},\frac{n}{2}]^2 \cap E(\mathbb Z^2)$. 
For a general domain $V \subset \mathbb Z^2$, let 
$$\partial_i V = \{x\in V: x\sim \mathbb Z^2 \setminus V\} \qquad \text{and}\qquad \partial_o V = \{x\in \mathbb Z^2\setminus V: x\sim V\}\,,$$  
where the subscripts $i,o$ are meant to indicate inner and outer boundary. The edge boundary is $$\partial_e V = \{e= vw: v\in \partial_i V, w\in \partial_o V\}\,,$$
and we can also denote the strict interior of a set $V$ by $\mathring{V} = V \setminus \partial_i V$. 

\subsection{The SOS measure}
We define the SOS model on general domains, with generic floors (and ceilings) and describe some important properties of the model. 

Consider a finite domain $V \subset \mathbb Z^2$. A boundary condition is an assignment of heights to the vertices of $\partial_o V$. A general SOS distribution that we consider may also have floors and ceilings that vary depending on the vertex; for that, we let $(\mathbf{a},\mathbf{b}) = ((a_v)_{v\in V},(b_v)_{v\in V})$ be such that $a_v <b_v$ for all $v$, and we call $a_v$ the floor and $b_v$ the ceiling for a vertex. 

For an inverse temperature $\beta>0$, external field $\lambda\ge 0$, the SOS distribution on $V$ with boundary conditions $\phi\in \mathbb Z^{\partial_o V}$ and ceiling/floor collection $(\mathbf{a},\mathbf{b})$ is the distribution 
\begin{align}\label{eq:general-SOS-measure}
    \mu_{\lambda,\beta,V}^{\mathbf{a},\mathbf{b},\phi}(\varphi)  = \frac{1}{Z_{\lambda,\beta,V}^{\mathbf{a},{\mathbf{b},\phi}}} e^{ -  H(\varphi)} \prod_{v\in V} \mathbf{1}_{a_v \le \varphi_v \le b_v}\,,
    \end{align}
    where $Z_{\lambda,\beta,V}^{\mathbf{a},{\mathbf{b},\phi}}$ is a normalizing factor to make it a probability measure, and 
    \begin{align}
    H(\varphi) = \beta \sum_{v\sim w} |\varphi_v - \varphi_w| + \lambda \sum_v \varphi_v\,.
\end{align}
where in the presence of boundary conditions $\phi$, the first sum includes a sum over all edges between $V$ and $\mathbb Z^2\setminus V$ in which case $\varphi_w$ is taken to be $\phi_w$. If we write $\mu[X]$ for a random variable $X$, we mean the expectation under the event $\mu$. 

In order to simplify notation, we will tend to drop $\mathbf{a},\mathbf{b}$ and $\lambda,\beta$ from the subscripts and superscripts of the measure to write $\mu_{\phi,V}$. We retain the relevant sub/superscripts when they are varying, but typically these will be fixed throughout a proof. The baseline floor/ceiling choices to have in mind will be $a_v \equiv 0$ and $b_v \equiv \infty$ for all $v\in V$, and we comment soon on other ones the reader might encounter. At the above level of generality, the SOS model satisfies a few essential properties that we recall. 

\subsubsection{Monotonicity and the FKG property}
Endow the space of SOS configurations with the natural partial order (induced by pointwise total orders). The above monotonicity statement in boundary conditions,  floors/ceilings, and field, is standard; we include a proof for completeness in \cref{sec:general-monotonicity-fkg}. 

\begin{lemma}\label{lem:SOS-monotonicity}
    Fix any $\beta>0$. Suppose that $\lambda'\le \lambda$, $a_v\le a_v'$, $b_v\le b_v'$ and $\phi \le \phi'$. Then 
    \begin{align*}
        \mu_{\lambda,V}^{\mathbf{a},\mathbf{b},\phi} \preceq \mu_{\lambda',V}^{\mathbf{a}',\mathbf{b}',\phi'}\,.
    \end{align*}
    where here and throughout the paper, $\preceq$ denotes stochastic domination. 
\end{lemma}

\subsubsection{Domain Markov property}
The next essential property satisfied by the model is the domain Markov property, that describes the law of the SOS measure conditional on its values on some subset of $V$, as a new SOS measure with some other boundary conditions. Namely, if we fix a subset $W \subset V$, and let $\zeta$ be a configuration on $W$ (consistent with the floors and ceilings $\mathbf{a},\mathbf{b}$), then 
    \begin{align*}
        \mu_{V}^\phi ( \varphi(V\setminus W) \in \cdot \mid \varphi(W) = \zeta)  \stackrel{d}=\mu_{V\setminus W}^{(\phi,\zeta)}\,,
    \end{align*}
    where the boundary conditions $(\phi,\zeta)$ are the concatenation of $\eta,\zeta$ on $\partial_o(V\setminus W)$. 

We end with an observation that if boundary conditions are above/below the floor/ceiling they can be treated as being exactly at the height of the floor/ceiling. Namely, the distribution $\mu_{V}^{\mathbf{a},\mathbf{b},\phi}$ is identical to $\mu_V^{\mathbf{a},\mathbf{b},\phi'}$ where for each vertex $w\in \partial_o V$, the boundary condition 
\begin{align}\label{eq:boundary-condition-equivalence}
\phi'_w = \begin{cases}
    \max_{v\sim w} b_v & \quad \phi_w\ge \max_{v\sim w} b_v \\ \min_{v\sim w} a_v & \quad \phi_w \le \min_{v\sim w}a_v \\ \phi_w & \qquad \text{else}
\end{cases}\,.\end{align}
It  thus suffices to always consider boundary conditions bounded between $\min_{v\in \partial_i V}a_v$ and $\max_{v\in \partial_i V} b_v$.

\subsection{A few basic estimates}

We present here a few basic results that will be helpful at various points in the paper. 
The first will be a bound on the maximum that is derived from comparison to the $\lambda =0$ SOS measure, and crude bounds on the maximum therein.  
(Much more refined understanding of the maxima was obtained in \cite{CLMST14}). For all $k\ge \frac{1}{4}\log |V|$, all $\beta$ large,   

\begin{align}\label{eq:easy-SOS-maximum}
    \mu_{V}^{\mathbf{a},\mathbf{b},\phi} \big(\max_{v\in V} \varphi_v \ge \max\{\|\mathbf{a}\|_\infty, \|\phi\|_\infty\} + k\big) \le  e^{ - \beta k}\,.
\end{align}
\begin{proof}
By monotonicity, it suffices to upper bound the probability that $\max_{v\in V}\varphi_v \ge k/2$ under the measure with $\lambda =0$, no ceiling, and with boundary conditions that are identically $\max\{\|\mathbf{a}\|_\infty,\|\phi\|_\infty\} + k/2$. By a union bound over the $|V|$ vertices, the probability $Ce^{ - 4\beta (k/2)}$ of a downward deviation by $k/2$ (this follows from a Peierls argument so long as $\beta\ge 1$, say, per \cite[Proposition 3.9]{CLMST14}), this measure is coupled to the one with no floor at $0$, except with probability $C |V| e^{ - 2\beta k}$. By a similar union bound, that probability that the no-floor sample has a maximum that goes up from the boundary conditions by an additional $k/2$ is at most $C |V|  e^{ - 4\beta (k/2)}$. Together with the fact that $k\ge \log |V|$, we obtain \cref{eq:easy-SOS-maximum}. 
\end{proof}

Various of our estimates in the first few sections of the paper will go through comparison and manipulation of partition functions to obtain convergent cluster expansions corresponding to the SOS model with boundary conditions that are at the ``correct" height. The next two bounds are simple a priori bounds on ratios of partition functions with different (constant) boundary conditions. 

\begin{lemma}\label{lem:preliminary-partition-function-bound}
    For every $V$ and every $k<m$, we have 
    \begin{align*}
        \frac{Z_{V}^{\mathbf{a},\mathbf{b},k}}{Z_{V}^{\mathbf{a}',\mathbf{b}',m}} \le e^{\lambda (m-k)|V|}\,.
    \end{align*}
    so long as $[a_v +m-k, b_v +m-k] \subset [a_v',b_v']$ for all $v$.
\end{lemma}
\begin{proof}
    Every configuration contributing to $Z_{V}^{\mathbf{a},\mathbf{b},k}$ can be mapped to one in $Z_{V}^{\mathbf{a}',\mathbf{b}',m}$ by shifting it by $m-k$. The change in the weights from this map is evidently exactly $e^{ \lambda(m-k)|V|}$. 
\end{proof}

The following shows a non-quantitative monotonicity of ratios of partition functions (the top corresponding to a lower measure than the bottom) as one increases the external field. 

\begin{lemma}\label{lem:monotonicity-part-function-ratio}
    For every $V$, every $\lambda$, every $\phi\le \phi'$, and any floor/ceiling pair $\mathbf{a},\mathbf{b}\le \mathbf{a}',\mathbf{b}'$, 
    \begin{align*}
        \frac{d}{d\lambda} \frac{Z_{\lambda,V}^{\mathbf{a},\mathbf{b},\phi}}{Z_{\lambda,V}^{\mathbf{a}',\mathbf{b}',\phi'}} \ge 0\,.
    \end{align*}
\end{lemma}

\begin{proof}
    Since the derivative of $\log(x)$ is positive on $\mathbb R_+$, it suffices to show the claimed non-negativity of 
   \begin{align*}\frac{d}{d\lambda}\log\Big(\frac{Z_{\lambda,V}^{\mathbf{a},\mathbf{b},\phi}}{Z_{\lambda,V}^{\mathbf{a}',\mathbf{b}',\phi'}}\Big) =
        \mu_{\lambda,V}^{\mathbf{a}',\mathbf{b}',\phi'} \big[\sum_{v} \varphi_v \big] - \mu_{\lambda,V}^{\mathbf{a},\mathbf{b},\phi} \big[\sum_v \varphi_v \big]\,.
    \end{align*}
    By the stochastic domination from \cref{lem:SOS-monotonicity} and the fact that $\sum_v \varphi_v$ is an increasing function, this difference is non-negative as claimed. 
\end{proof}

\subsection{The common choices of boundary conditions and floors/ceilings}
For much of the paper, we will be focused on a specific class of boundary conditions and floors/ceilings. The baseline choice of floors and ceilings in this paper is $a_v \equiv 0$ and $b_v \equiv \infty$ except when specified otherwise. The special class of boundary conditions are those which are constant $\phi \equiv k$ for some $k\ge 0$. The special class of floors/ceilings will  modify the baseline choice along $\partial_i V$ via a \emph{boundary signing} $\eta \in \{+,-,\emptyset,f\}^{\partial_i V}$.

\begin{definition}\label{def:boundary-signing}
    If the boundary condition $\phi$ is constant, $\phi \equiv k$, say, then the \emph{boundary signing} indicates the choices of floors/ceilings where for each $v\in \partial_i V$,
    \begin{itemize}
        \item if $\eta_v= +$ then $a_v = k$ and if $\eta_v = -$, then $b_v =k$
        \item if $\eta_v = \emptyset$, then $a_v = b_v =k$
        \item if $\eta_v = f$, then $a_v, b_v$ are the default choices of $0$ and $\infty$ respectively.  
    \end{itemize}
\end{definition}

Until \cref{sec:spatial-mixing}, constant $\phi$ and floors/ceilings given by $a_v \equiv 0, b_v \equiv \infty$ and boundary signings will be the exclusive family of boundary conditions and floors/ceilings we work with and therefore we give it a shortened notation.
Define   $\mu_{\eta,k,V}$ as the SOS measure with boundary conditions that are identically $k$, and boundary signing $\eta$ modifying $a_v\equiv 0$, $b_v =\infty$ along $\partial_i V$ according to \cref{def:boundary-signing}. Let $Z_{\eta,k,V}$ be the corresponding partition function.  
In this special context, \cref{lem:preliminary-partition-function-bound} translates to the following. 
\begin{corollary}\label{cor:preliminary-partition-function-bound}
    For every $V$ and every $k<m$ and every boundary signing $\eta$,  
    \begin{align*}
        \frac{Z_{\eta,k,V}}{Z_{\eta,m,V}} \le e^{\lambda (m-k)|V|}\,.
    \end{align*}
\end{corollary}

\subsection{The contour representation}\label{subsec:contour-representation-defs}
An important tool in the analysis of the SOS model is the \emph{contour representation}, essentially a bijection between configurations and a family of non-crossing loops representing level curves of the SOS surface. The contours will be formed out of \emph{dual edges}; for that, let $V^*$ denote the planar dual graph of $V$ generated by considering $(\mathbb Z^2)^* =\mathbb Z^2 + (\frac{1}{2},\frac{1}{2})$, then taking $E(V^*)$ to consist of all edges of $E((\mathbb Z^2)^*)$ bisecting the edge set $\{e = vw: v\in V\}$. The vertex set $V^*$ is then naturally induced.

We will use the contour representation of SOS configurations, referring the reader to \cite[Section 3]{CLMST14} for more background on the below. 

\begin{definition}
        A geometric contour $\gamma$ is a connected set of dual-edges in $E(V^*)$ 
    whose elements can be sequenced $e_0,...,e_{|\gamma|}$ in such a way that 
    \begin{itemize}
        \item $e_{i} \ne e_j$ if $i\ne j$ except if $i=0$ and $j=n$;
        \item $e_{i}$ and $e_{i+1}$ are incident for every $i$;
        \item if $e_i,e_{i+1},e_j,e_{j+1}$ all share a single dual vertex, then $e_i,e_{i+1}$ are either the south and west pair of dual edges, or the north and east pair. 
    \end{itemize}
\end{definition}

The last item above is a canonical splitting rule, where the pairings SW and NE are simply a convention. With this convention, we will call sets $V\subset \mathbb Z^2$ simply-connected if there is a geometric contour, whose edge set is exactly $(\partial_e V)^*$, and we will use the notation $\gamma_V$ to denote this contour.  
\begin{definition}\label{def:interior-exterior-contour}
    A vertex is exterior to a contour $\gamma$ if it lies in the infinite connected component of $\mathbb R^2\setminus \gamma$. When there is an underlying domain $V\subset \mathbb Z^2$, we use $\Ext(\gamma)$ to denote those vertices of $V$ exterior to $\gamma$. 
    A vertex is interior to a contour $\gamma$ if it lies in a finite connected component of $\mathbb R^2 \setminus \gamma$, and $\Int(\gamma)$ denotes the set of such vertices. 
    We say $\gamma$ nests $\gamma'$ if $\Int(\gamma')\subset \Int(\gamma)$, and two contours $\gamma,\gamma'$ are mutually external if neither nests the other. 
\end{definition}

    We are now able to define up and down contours corresponding to an SOS configuration $\varphi$.  

\begin{definition}\label{def:up-down-contour}
In an SOS configuration $\varphi$, a geometric contour $\gamma$ is an up-$h$-contour if $\varphi_x \ge h$ for all $h\in \partial_i \Int(\gamma)$ and $\varphi_x <h$ for all $h\in \partial_o \Int(\gamma)$. Analogously, a geometric contour $\gamma$ is a down-$h$-contour if $\varphi_x \le h$ for all $h\in \partial_i \Int(\gamma)$ and $\varphi_x >h$ for all $h\in \partial_o \Int(\gamma)$. In these cases, use the notation $S(\gamma)=\pm$ for whether it is up or down, and $h(\gamma)=h$. 
\end{definition}

If we consider SOS configurations with constant boundary $\phi \equiv k$, we can associate to each configuration a contour representation, where the contours it consists of are all signed contours at specific heights. Namely, for a configuration $\varphi$, with boundary conditions $k$, on a domain $V$, first include the outermost up-$(k+1)$-contours and down-$(k-1)$-contours. Interior to each of the outermost up-$(k+1)$-contours, we can consider the outermost up-$(k+2)$ and down-$k$-contours, and proceed iteratively in this manner to get a family of signed contours, each having an associated height. Per this construction, if multiple contours are the same geometric contour, just at different heights, the one whose height is smallest if they are up-contours, or largest if they are down-contours, is called the outermost one.

\begin{definition}\label{def:compatible-contours}
    A family of contours $\Gamma$ is said to be admissible if there exists a $\varphi$ whose contour collection is exactly $\Gamma$. (This effectively imposes that contours may not cross, and if they overlap on an edge, their signs must be the same if one nests the other, while their signs must be different if they are mutually external.)
\end{definition}

With this, we can express the SOS measure in terms of weights associated to contours. 

\begin{definition}\label{def:contour-weight}
    The weight of a contour $\gamma$ is 
    \begin{align*}
        W(\gamma) = \begin{cases}
            e^{ - \beta |\gamma| - \lambda |\Int(\gamma)|} & \quad \gamma \text{ is an up-contour} \\
            e^{ - \beta |\gamma| + \lambda |\Int(\gamma)|} & \quad \gamma \text{ is a down-contour}
        \end{cases}\,.
    \end{align*}
\end{definition}

There is a bijection from height functions that have a constant boundary condition $\phi\equiv h$ outside a set $V\subset \mathbb Z^2$ (and no floor/ceiling) and admissible contour collections in $V^*$. In fact, the information of the heights of the contours is not needed to reconstruct $\varphi$, only their signs are needed (the heights can be reconstructed from the boundary inwards). 
In our contexts, we require contour collections to be compatible with the floor/ceilings. In the default case we are interested in a floor at height zero, and a boundary signing $\eta$ per \cref{def:boundary-signing}.

Denote by $\mathscr{G}_{\eta,h,V}$ the set of all admissible contours collections whose corresponding configuration is in $\Omega_{\eta,h,V}$. Notice that the boundary signing is imposing a constraint on the contours that no dual edge bisecting $\partial_e V \cap \partial_e \{v\}$
is part of an up contour if $\eta_v\in\{-,\emptyset\}$ (respectively is part of a down contour if $\eta_v \in \{+,\emptyset\}$). 

With that, we can rewrite the corresponding partition function in terms of contour weights as:    
\begin{align}\label{eq:contour-rep-partition-function}
    Z_{\eta,h,V} := e^{ - \lambda h|V|} \sum_{\Gamma \in \mathscr G_{\eta, h,V}} \prod_{\gamma \in \Gamma} W(\gamma)\,.
\end{align}

The specific choices where $\eta \equiv +$ and $\eta \equiv -$ will be especially pertinent since, e.g.,  $Z_{+,h(\gamma),\Int(\gamma)}$ describes the possible contour collections found in $\Int(\gamma)$ if $\gamma$ is an up-$h$-contour.

\section{Renormalized weights, elementary contours and free energies}\label{sec:elementary}
We study a family of contours which are exponentially unlikely (in their length) no matter what height they are based at. These are small enough that neither the relative benefit of the external field (for down-contours) nor the entropic repulsion effect (for up-contours) are  enough to overcome the exponential in $\beta$ cost in the length of the contour. Following the language of~\cite{DiMa94,CeMa1}, these contours are called \emph{elementary}. We note that the overarching goal of this section, \cref{thm:main-elementary-general-bc}, tracks~\cite{DiMa94} but for contours instead of the \emph{cylinders}---essentially maximal stacks of identical contours---that were considered in that paper. This simplifies some aspects of the proof, while complicating others; in particular, our proof requires stitching together Peierls-type arguments in the standard contour representation of the SOS model, with arguments for using the ``renormalized" contour representation defined as follows. 

\subsection{Renormalized contour weights}
The first step we take is to rewrite SOS partition functions in terms of renormalized contour weights according to a specified boundary height $h$. The role of this renormalization is to move to a model over contour collections with all the contours based at height $h$, rather than each having their own associated height; the outermost contours of this model will still correspond to the outermost contours of the original model.  

\begin{definition}\label{def:contour-renormalized-weight}
    For a contour $\gamma$, the $h$-renormalized weight is defined as 
    \begin{align*}
         W_h^\rn(\gamma) = \begin{cases} e^{ - \beta |\gamma|} \frac{Z_{+,h+1,\Int(\gamma)}}{Z_{+,h,\Int(\gamma)}} & \quad \gamma \text{ is an up-contour} \\ e^{ - \beta |\gamma|}\frac{Z_{-,h-1,\Int(\gamma)}}{Z_{-,h,\Int(\gamma)}} & \quad \gamma \text{ is a down-contour}
         \end{cases}
         \,.
    \end{align*}
\end{definition}

Given these renormalized weights, we can obtain the following equivalence between the partition function on a simply connected domain $V$ and without tracking the heights of the corresponding contours, allowing arbitrary compatible families. Towards that purpose, define $\mathscr{G}_{\eta,h,V}^\rn$ as admissible collections of signed contours compatible with $\eta$, but with no constraint on the height function they generate respecting a floor at height $0$ (unless $h=0$, in which case no down contours are allowed). With that, we can define 
\begin{align}\label{eq:renormalized-partition-function}
    Z^\rn_{\eta,h,V} = e^{ - \lambda h|V|} \sum_{\Gamma \in \mathscr{G}^\rn_{\eta,k,V}} \prod_{\gamma \in \Gamma} W^\rn_h(\gamma)\,.
\end{align}
In an element $\Gamma\in \mathscr{G}^\rn_{\eta,h,V}$, we say a contour $\gamma\in \Gamma$ is outermost if it is not strictly nested by any other contour in $\Gamma$. Since contours may overlay fully and in the renormalized collection are all based at the same height (i.e., multiple copies of the same contour can be in $\Gamma$), we arbitrarily call one of them the outermost one.  

\begin{lemma}\label{lem:part-function-renormalized-weights}
    The standard partition function and the renormalized one are equal, namely for any simply-connected set $V$, 
    \begin{align*}
        Z_{\eta,h,V}= Z^\rn_{\eta,h,V}\,.
    \end{align*}
    Moreover, the same equality holds if we specify some outermost contours which must belong to a configuration and only sum over $\Gamma$ having those as outermost contours. 
\end{lemma}

\begin{proof}
    If we let $\Gamma_\out$ be collections of outermost contours, then by \cref{eq:contour-rep-partition-function} and \cref{def:contour-renormalized-weight},
    \begin{align*}
        e^{ \lambda h|V|} Z_{\eta,h,V} = \sum_{\Gamma_\out} \prod_{\gamma \in \Gamma_\out} W(\gamma) e^{ \lambda h |\Int(\gamma)|} Z_{S(\gamma),h \pm 1, \Int(\gamma)} = \sum_{\Gamma_\out} \prod_{\gamma\in \Gamma_\out} W^\rn_h(\gamma) e^{ \lambda h |\Int(\gamma)|} Z_{S(\gamma),h,\Int(\gamma)}\,,
    \end{align*}
    where in the middle display, $h+1$ is used when $S(\gamma) =+$ and $h-1$ when $S(\gamma) = -$. Performing the same expansion on  $e^{\lambda h |\Int(\gamma)|} Z_{S(\gamma),h,\Int(\gamma)}$ and repeating, this gives the claimed equality to $e^{ \lambda h |V|} Z_{\eta,h,V}^\rn$ concluding the proof. It is evident that the same argument could be performed restricted to sums over $\Gamma_\out$ that contain a fixed collection of outermost contours. 
\end{proof}

The benefit of the above renormalized contour weight representation is that the deletion of a contour from $\Gamma$ leaves an element of the same $\mathscr{G}_{\eta,h,V}^\rn$, whereas in $\mathscr{G}_{\eta,h,V}$ such a deletion could cause conflicts with the floor at height $0$. This enables the use of simple Peierls maps on the renormalized collection so long as the renormalized weights are exponentially small in $|\gamma|$.  

\subsection{Elementary contours}
We next consider \emph{elementary contours}, which are those we can show to have exponentially small renormalized weight regardless of the height $h$ with respect to which they are renormalized. 
For a subset $A$ in $\mathbb R^2$, we use $\diam(A)$ to denote $\max\{\|x-y\|_\infty: x,y\in A\}$.

\begin{definition}\label{def:h-elem}
    A contour $\gamma$ is called $h$-elementary if it has 
    \begin{align*}
        \diam(\gamma) \le \lambda^{-1}\wedge e^{3\beta (h+1)}\,.
    \end{align*}
\end{definition}

\begin{remark}\label{rem:properties-of-elementary}
Note the following two important properties: 
\begin{itemize}
\item if $\gamma$ is $h$-elementary, then it is also $(h+1)$-elementary.
\item if $\gamma'$ is nested in $\gamma$ and $\gamma$ is $h$-elementary, then $\gamma'$ is also $h$-elementary. 
\end{itemize}
\end{remark}

We now introduce the notation $\mathscr{G}^\elem_{\eta,h,V}$ and the notation $\mathscr{G}^\rnelem_{\eta,h,V}$, for compatible (renormalized) collections of contours all of which are $h$-elementary. With these notations, we can define 
\begin{align}\label{eq:elem-rnelem-partition-functions}
    Z^\elem_{\eta,h,V} = e^{ - \lambda h|V|} \sum_{\Gamma \in \mathscr{G}^\elem_{\eta, h,V}} \prod_{\gamma \in \Gamma}  W(\gamma)\,, \qquad\text{and} \qquad Z^{\rnelem}_{\eta,h,V} = e^{ - \lambda h|V|} \sum_{\Gamma \in \mathscr{G}^\rnelem_{\eta, h,V}} \prod_{\gamma \in \Gamma}  W^\rn_h(\gamma)\,.
\end{align}

By the same reasoning as in the proof of \cref{lem:part-function-renormalized-weights}, we obtain the following equivalence. 

\begin{lemma}\label{lem:elem-part-function-renormalized-weights}
The elementary partition function and the renormalized elementary one are equal, namely for a simply-connected set $V$ 
\begin{align*}
    Z^\elem_{\eta,h,V} = Z^\rnelem_{\eta,h,V}\,.
\end{align*}
Moreover, the same equality holds if we specify some outermost contours which must belong to a configuration, and only sum over $\Gamma$ having those as outermost contours. 
\end{lemma}
In what follows, we use the notation $\mu^\elem_{\eta,h,V}$ to denote the distribution over $\mathscr{G}_{\eta,h,V}^{\elem}$ which assigns each $\Gamma$, the weight $\prod_{\gamma\in \Gamma}W(\gamma)$. Equivalently, this is the SOS measure $\mu_{\eta,h,V}$ conditioned on its contour representation belonging to $\mathscr{G}^\elem_{\eta,h,V}$, i.e., conditioned on all its outermost contours being $h$-elementary.

The following theorem is our main result in this section, and establishes that the $h$-elementary contours always have exponentially small $h$-renormalized weight. It also gives a first bound, that is sharp up to the constant in the exponent, for the ratio of elementary partition functions at different heights. This is the analogue, in the context of our contours, of~\cite[Lemma 2.5]{DiMa94} and~\cite[Lemma 2.7]{CeMa1}.

    \begin{theorem}\label{thm:main-elementary-general-bc}
    There exists $\beta_0$ such that for every simply-connected $V\subset \mathbb Z^2$ and every $\beta>\beta_0$, 
    \begin{enumerate}
        \item \label{it:Z-elem-ratio-general-bc} For every $0\le k<m$ and every boundary signing $\eta$, 
        \begin{align}\label{eq:ratio-of-Z}
            e^{ - 2|V| e^{ - (4\beta - \lambda)(k+1)}} \le e^{-\lambda(m-k)|V|}\frac{Z^\elem_{\eta,k,V}}{Z^\elem_{\eta, m,V}} \le 2e^{ - \frac{1}{2}|\mathring{V}| e^{ - (4\beta - \lambda)(k+1)}}\,.
        \end{align}
        \item \label{it:contour-weights} For every $h\ge 0$, if $\gamma$ is $h$-elementary, 
        \begin{align*}
             W_h^\rn(\gamma) \le e^{ - (\beta -2)|\gamma|}\,.
        \end{align*}
    \end{enumerate}
\end{theorem}

\begin{proof}
        We prove this by induction, showing that \cref{it:Z-elem-ratio-general-bc} implies \cref{it:contour-weights} holds for that same $V$, and then by showing that if \cref{it:contour-weights} holds for $V$ for all connected sets $V$ having $|V|\le l$, then that implies \cref{it:Z-elem-ratio-general-bc} for all sets $V'$ having $|V'|\le l+1$. This second step is the core of the argument and entails maps on elementary contour families to estimate the entropic repulsion effect amongst elementary contours, at a given height $k$. 

\bigskip
\noindent {\em Base case for~\cref{it:Z-elem-ratio-general-bc}.} Consider $V = \{x\}$, whence $\mathring V = \emptyset$. 
   In this case, the only contours in $E(V^*)$ are the loops consisting of the four dual-edges surrounding $x$, i.e., $\gamma_{\{x\}}$ (notice that the these contours are always elementary, no matter its height or sign, since $\diam(\gamma_{\{x\}})=1$ is less than $\lambda^{-1} \wedge e^{3\beta}$).

    If $\eta = \emptyset$ then $Z^\elem_{\emptyset,k,V} = e^{ - \lambda k}$ and the middle quantity in \cref{eq:ratio-of-Z} is exactly $1$, which satisfies the two bounds since $|\mathring{V}| = 0$. 
  If $\eta= +$ then
   \begin{align*}
        Z_{+, k,V}^\elem = e^{ - \lambda k} \sum_{l\ge 0} e^{ - 4\beta l  - \lambda l}\,, \qquad \text{so that}\qquad 
        e^{-\lambda (m-k)} \frac{Z_{+, k,V}^\elem}{Z_{+, m,V}^\elem} = \frac{\sum_{l \ge 0}e^{ - 4\beta l - \lambda l}}{\sum_{l\ge 0}e^{ - 4\beta l - \lambda l}} =1\,,
   \end{align*}
   which again satisfies the two bounds in \cref{eq:ratio-of-Z}. 
      If $\eta= -$, then 
\begin{align*}
        Z_{-, k,V}^\elem = e^{ - \lambda k} \sum_{l\ge 0} e^{ - 4\beta l  + \lambda l}\,, \qquad \text{so that} \qquad
        e^{-\lambda (m-k)} \frac{Z_{-, k,V}^\elem}{Z_{-, m,V}^\elem} = \frac{\sum_{0\le l \le k}e^{ - 4\beta l + \lambda l}}{\sum_{0\le l \le m}e^{ - 4\beta l + \lambda l}}\,.
   \end{align*}
   Since $\lambda \le 1$ and $\beta>\beta_0$, we have $e^{ - 4\beta + \lambda}<1$ and therefore the middle quantity in \cref{eq:ratio-of-Z} is equal to 
   \begin{align*}
        \frac{1-e^{ (- 4\beta + \lambda) (k+1)}}{1-e^{ (- 4\beta +\lambda) (m+1)}}\,.
    \end{align*}
    This is at most $1$ since $m>k$. 
   As for the lower bound, the denominator is at most $1$ (as $k<m$ and $\lambda<4\beta$), and for the numerator we use $1-x\geq e^{-2x}$ for $x<\frac12$, so it is at least $e^{ - 2 e^{ (-4\beta + \lambda)(k+1)}}$ as claimed. Finally if $\eta = f$, then 
   \begin{align*}
       e^{ - \lambda(m-k)}\frac{Z^\elem_{f,k,V}}{Z^\elem_{f,m,V}} = \frac{\sum_{l \ge -k} e^{ - 4\beta |l| - \lambda l}}{\sum_{l\ge -m} e^{- 4\beta |l| - \lambda l}} = \frac{\sum_{l>0} e^{ - 4\beta l - \lambda l} + \sum_{0\le l\le k} e^{ - 4\beta l + \lambda l}}{\sum_{l>0} e^{ - 4\beta l - \lambda l} + \sum_{0\le l\le m} e^{ - 4\beta l + \lambda l}}\,.
   \end{align*}
   The last fraction evidently lies between the ratio of the first sums and the ratio of the second sums. The first of these ratios is exactly $1$ and the second of these ratios is exactly ${Z_{-, k,V}^\elem}/{Z_{-, m,V}^\elem}$; in particular, both satisfy the upper and lower bounds of \cref{eq:ratio-of-Z}, and thus so does the above. 

       \bigskip
    \noindent \emph{\cref{it:Z-elem-ratio-general-bc} for $|V| \le l$ implies \cref{it:contour-weights} for $\gamma: |\Int(\gamma)| = l$.} Suppose \cref{it:Z-elem-ratio-general-bc} holds for all sets $V$ having $|V|\le l$. 
    Fix any $h$ and consider an up-contour $\gamma$ with $\Int(\gamma) = l$. Recalling \cref{def:contour-renormalized-weight}, we have  
    \begin{align*}
        W^\rn_h(\gamma) = e^{ - \beta |\gamma|} \frac{Z_{+, h+1,\Int(\gamma)}}{Z_{+,h,\Int(\gamma)}} 
        e^{ - \beta |\gamma|} \frac{Z^\elem_{+, h+1,\Int(\gamma)}}{Z^\elem_{+, h,\Int(\gamma)}}\,,
    \end{align*}
    where we used that since $\gamma$ is $h$-elementary (and by \cref{rem:properties-of-elementary} also $h+1$-elementary), every $\gamma'$ in $\Int(\gamma)$ must also be both $h$ and $h+1$-elementary, so both partition functions are equal to their $\rnelem$-versions, and those in turn are equal to the elementary ones since $\Int(\gamma)$ is simply connected per \cref{lem:elem-part-function-renormalized-weights}. 
    Thus, by the lower bound in \cref{eq:ratio-of-Z} from \cref{it:Z-elem-ratio-general-bc} for the domain $\Int(\gamma)$, with $k=h$ and $m=h+1$,  
    \begin{align*}
          W^\rn_h(\gamma) \le e^{ - \beta |\gamma| - \lambda |\Int(\gamma)| + 2|\Int(\gamma)|e^{ - (4\beta - \lambda)(h+1)}}\,. 
    \end{align*}
    Since $\gamma$ is $h$-elementary, $|\Int(\gamma)|\le |\gamma| \diam(\gamma)  \le  |\gamma| e^{ 3\beta(h+1)}$, and thus we get the claimed bound on $ W^\rn_h(\gamma)$ at $\beta>\beta_0$.  
    For a down-contour $\gamma$, by \cref{cor:preliminary-partition-function-bound} with $k=h-1$ and $m=h$, 
    \begin{align*}
         W^\rn_h(\gamma) = e^{ - \beta|\gamma|} \frac{Z_{-, h-1,\Int(\gamma)}}{Z_{-, h,\Int(\gamma)}} \le e^{ - \beta |\gamma| + \lambda|\Int(\gamma)|}\,.
    \end{align*} 
    Since $|\Int(\gamma)| \le |\gamma|\diam(\gamma)\le  \lambda^{-1}|\gamma|$ and $|\gamma| \ge 1$, this gives the claimed bound. 

    \bigskip    
    \noindent \emph{\cref{it:contour-weights} for all $|V|< l$ implies \cref{it:Z-elem-ratio-general-bc} for $|V|\le l$.} In what follows, suppose \cref{it:contour-weights} holds for all $\gamma$ having $|\Int(\gamma)|\le l-1$. Let us begin with a few  consequences of the inductive assumption of \cref{it:contour-weights} for $V:|V|\le l-1$ that will be useful in both the upper and lower bound of \cref{it:contour-weights}. 

    The first is the easy observation that a bound on the renormalized weights implies a bound on the probability that under the renormalized contour model, a certain family of compatible $k$-contours is present. 
    \begin{lemma}\label{lem:Peierls-renormalized-elementaries}
        Consider a fixed family $G$ of compatible $k$-elementary contours, all of which have $|\Int(\gamma)|\le l-1$. Let $Z^\rnelem_{\eta,k,V}(G)$ be the analogue of $Z^\rnelem_{\eta,k,V}$ with the sum only running over elements of $\mathscr{G}^{\rnelem}_{\eta,k,V}$ containing all of $G$. 
        Then for any $V$, 
        \begin{align*}
        \frac{Z^\rnelem_{\eta,k,V}(G)}{Z^\rnelem_{\eta,k,V}} \le e^{ - (\beta -2) \sum_{\gamma \in G}|\gamma|}\,.
        \end{align*}
        If the contours in $G$ are mutually external and we only sum over $\mathscr{G}^{\rnelem}_{\eta,k,V}$  in which they are outermost, the same bound applies to $Z^\elem_{\eta,k,V}(G)/Z^\elem_{\eta,k,V}$. 
    \end{lemma}
    \begin{proof}
        Consider the map that removes $G$ from the numerator to send it into an element of the denominator. This is an injection, and the relative weight of each pre-image to its image is $\prod_{\gamma \in G} W^\rn_k(\gamma)$. By the inductive assumption and the fact that every $\gamma\in G$ has interior smaller than $l-1$, we get the claim. 
    \end{proof}

    Since our contours may overlay, when $|V| = l$, it may admit the contour $\gamma_V= (\partial_e V)^*$, which will have $|\Int(\gamma_V)| = l$ and to which the inductive assumption will not apply. We use a 1-site modification to show therefore that if $V$ has size $l$, then it is exponentially unlikely to have $\gamma_V$ as a contour. For $j\in \mathbb Z$,  
    \begin{align}\label{eq:E^j}
    \cE_{\eta,k,V}^j = \{\Gamma\in \mathscr{G}_{\eta,k,V}: \text{exactly $|j|$ copies of $\gamma_V$ with sign matching $\sgn(j)$ in $\Gamma$}\}\,.\end{align}

    \begin{lemma}\label{lem:no-fully-enclosing-contour}
        For every $\eta,k$ and $V$ such that $|V|\le l$, we have for every $j$, 
        \begin{align*}
            \mu_{\eta,k,V}^\elem(\cE^j_{\eta,k,V}) \le \exp( - (\beta - 2) |j| |\gamma_V| + 5(1\vee j))\,.
        \end{align*}
    \end{lemma}
    \begin{proof}
        In order for the probability not to be zero, it must be that $\gamma_V$ is $k$-elementary, and either $\eta_v = \{-,f\}$ for all $v\in \partial_i V$ if $j$ negative, or $\eta_v \in \{+,f\}$ for all $v\in \partial_i V$ if $j$ positive. 

        Consider first the probability that under $\mu^\elem_{\eta,k,V}$ there exists at least one contour in $\Gamma$ that is all of $\gamma_V$. In the language of \cref{lem:part-function-renormalized-weights}, this is equivalent to the fraction $\frac{Z^\elem_{\eta,k,V}(\{\gamma_V\})}{Z^\elem_{\eta,k,V}}=\frac{Z^\rnelem_{\eta,k,V}(\{\gamma_V\})}{Z^\rnelem_{\eta,k,V}}$; however, we cannot yet apply \cref{lem:Peierls-renormalized-elementaries} because $|\Int(\gamma_V)| = |V| = l$, not $l-1$. Therefore, we first consider the modification $M_x \Gamma$ of $\mathscr{G}^\rnelem_{\eta,k,V}(G)$ that performs an $\oplus$ operation with a single copy of $\gamma_{\{x\}}$ satisfying $V\setminus \{x\}$ is s.c. (such an $x$ must exist). For each $\Gamma \in \mathscr{G}^{\rnelem}_{\eta,k,v}(\{\gamma_V\})$, let $d_x$ ($|d_x|\le 2$) be the change in the total number of edges in $\Gamma$ vs.\ the $M_x\Gamma$. The resulting configuration evidently belongs to $\mathscr{G}^\rnelem_{\eta,k,V}$. Its relative weight satisfies 
        \begin{align*}
            \frac{\prod_{\gamma \in M_{x}\Gamma} W^\rn_k(\gamma)}{\prod_{\gamma \in \Gamma} W^\rn_k(\gamma)} \le e^{ \beta d_x+ \lambda}\,.
        \end{align*}
        Furthermore, $M_x$ is an injection since the original configuration can be recovered by the $\oplus$ operation with the singleton loop around $\{x\}$. Let 
        $$\widetilde Z = \sum_{\Gamma\in \mathscr{G}^\rnelem_{\eta,k,V}(\{\gamma_V\})} \prod_{\gamma \in M_x \Gamma} W^\rn_k(\gamma) = \sum_{M_x\Gamma\in M_x\mathscr{G}^\rnelem_{\eta,k,V}(\{\gamma_V\})} \prod_{\gamma \in M_x\Gamma} W^\rn_k(\gamma)\,.$$ 
        Each $M_x\gamma$ appearing in this sum can in turn be mapped to one in which the contour $\tilde \gamma$ in $M_x\Gamma$ resulting from the modification to $\gamma_V$ after $\oplus \gamma_{\{x\}}$, is deleted from the collection. That yields a further weight change $W^\rn_k(\tilde \gamma)$. Since $\tilde \gamma$ necessarily has $|\Int(\tilde \gamma)| \le l-1$, this is at most $e^{ - (\beta -2) |\tilde \gamma|}$. Summing over the possible choices for the contour $\tilde \gamma$ (namely summing over possible contours $\gamma'$ adjacent to $x$ such that $\tilde \gamma = \gamma_V \oplus \gamma_{\{x\}}\oplus \gamma'$, the number of such choices being at most $4^{r}$ for $|\gamma'|=r$, say), we get 
        \begin{align}\label{eq:at-least-one-gammaV}
            \frac{Z^\elem_{\eta,k,V}(\{\gamma_V\})}{Z^\elem_{\eta,k,V}} \le \sum_{r\ge 0} 4^{r} e^{ - (\beta -2)(|\gamma_V| + d_x + r) + \beta d_x + \lambda} \le  e^{ - (\beta -2)|\gamma_V| + \lambda + 4} \,.
        \end{align}
        where the $4$ comes from the worst-case value of $d_x$. 

        We now wish to reduce the event $\cE^j_{\eta,k,V}$ to the event of having at least $1$ such contour. If $j$ is negative, consider the deletion of $|j|-1$ such contours from the (standard) contour representation. That gives a change of $e^{ \beta(j-1)|\gamma_V| -\lambda (j-1)|V|}$ to the weight of the configurations, and is an injection from $\mathcal E^j_{\eta,k,V}$ to $\mathcal E^{-1}_{\eta,k,V}$. Since we assumed $\gamma_V$ to be $k$-elementary, $\lambda |V|\le \lambda \diam(V) |\gamma_V|\le  |\gamma_V|$, and thus we have 
        \begin{align*}
            \mu(\cE^{j}_{\eta,k,V}) \le e^{- (\beta-1) (j-1)|\gamma_V|} \mu(\cE^{-1}_{\eta,k,V}) \le e^{ - (\beta -1)(j-1)|\gamma_V|} e^{ - (\beta -2)|\gamma_V| + \lambda +  4}\,.
        \end{align*}
        which easily the claimed bound using $\lambda \le 1$. 

        If $j$ is positive, notice that on the event of having $\ge j-1$ many up-contours $\gamma_V$, the probability of having at least $j$ is at most the probability of having at least $j-1$, times the conditional probability of having at least one more given those $j-1$. But the conditional measure given at least $j-1$ up-contours $\gamma_V$ is exactly $\mu_{+,k+j-1,V} = \mu_{+,k+j-1,V}^\elem$ because all possible contours in $V$ are enclosed in $\gamma_V$, which is $k$-elementary and therefore also $k+j-1$ elementary for $j\ge 1$. For this conditional probability, the bound of \cref{eq:at-least-one-gammaV} applies (with $k+j-1$ instead of $k$), and iterating that down, this gives  
        \begin{align*}
            \mu^\elem_{\eta,k,V}(\cE^j_{\eta,k,V}) \le \prod_{i=1}^j e^{ - (\beta -2) |\gamma_V| + \lambda + 4}\,,
        \end{align*}
        yielding the claim after $\lambda \le 1$. 
    \end{proof}

    Together, \cref{lem:Peierls-renormalized-elementaries,lem:no-fully-enclosing-contour} imply that if $V$ has $|V|= l$, the probability under $\mu^\elem_{\eta,k,V}$ of any fixed collection of outermost contours is exponentially small. 
    \begin{corollary}\label{cor:outermost-elem-contour-exp-tail}
        Let $V$ have $|V|\le l$. The probability of having a specific contour $\gamma$ as an outermost contour under $\mu^\elem_{\eta,k,V}$, is at most $e^{ - (\beta -3)|\gamma|}$. 
    \end{corollary}
    \begin{proof}
        If $\gamma = \gamma_V$, this is ruled out by \cref{lem:no-fully-enclosing-contour}, a sum over $j$, and $|\gamma_V|>4$. Otherwise, the probability of $\gamma$ being an outermost contour is at most $e^{ - (\beta-2) |\gamma|}$ per \cref{lem:Peierls-renormalized-elementaries}. 
    \end{proof}
    
    With the above in hand, we move to proving \cref{it:Z-elem-ratio-general-bc}; we split up the proofs of the upper and lower bound.  

    \medskip
    \noindent   \emph{Upper bound in \cref{eq:ratio-of-Z}}. Intuitively, the proof of the upper bound goes by mapping configurations in $Z_{\eta,k,V}^\elem$ to $Z_{\eta,m,V}^\elem$ in such a way to inject entropy via singleton spikes $\gamma_{\{x\}}$ that go down by height $m$, which would have been prohibited under $Z_{\eta,k,V}^\elem$ due to the floor constraint. We first construct such a map, using the contour representation of the elementary SOS configurations. We then use its alternative representation with renormalized weights together with the inductive hypothesis to argue that under $Z_{\eta,k,V}^\elem$, most sites are at height $k$, and thus there is typically $(1-\epsilon_\beta)|V|$ many sites to insert these downward spikes. 
    
    We begin by working with $\mathscr{G}^\elem_{\eta,k,V}$, which we recall is the set of contour representations of SOS configurations compatible with the floor at height zero, and all of whose outermost contours are all $k$-elementary. This in fact implies that all its contours are $k$-elementary (though the inner contours may not be elementary at the height they are based at).     
    For $\Gamma \in \mathscr{G}^\elem_{\eta,k,V}$, let 
    \begin{align*}
        V_\Gamma = \big\{x\in \mathring{V}: \min_{\gamma \in \Gamma} d(x,\Int(\gamma)) >1\big\}\,.
    \end{align*}
    Given any $\Gamma \in \mathscr{G}^\elem_{\eta,k,V}$ and a subset $S\subset V_\Gamma$, define the map $T_s$ as the one that adds to $\Gamma$, collections of $k+1$ many copies of the down-contour $\gamma_{\{v\}}$.  
    
    \begin{claim}
        For every $k<m$, every $\Gamma \in \mathscr{G}_{\eta,k,V}^\elem$, and every $S\subset V_\Gamma$, we have $T_S \Gamma \in \mathscr{G}_{\eta,m,V}^\elem$. 
    \end{claim}
    \begin{proof}
    This resulting collection $T_S \Gamma$ is in $\mathscr{G}_{\eta, m,V}$ because the addition of these contours (1) is compatible with the remainder of the contour collection since $S\subset V_\Gamma$ and $V_\Gamma$ ensures that any newly added contours will not be incident to any contours in $\Gamma$; and (2) does not violate the floor constraint since $m\ge k+1$, and in $\Gamma$, the height of sites in $V_\Gamma$ is $k$. Finally, it is compatible with the signing on the boundary condition since $S\subset V_\Gamma \subset \mathring{V}$. To see that $T_S \Gamma\in \mathscr{G}_{\eta, m,V}^\elem$, we further claim that in $T_S \Gamma$, every outermost contour is $m$-elementary. This is because the singleton contours are trivially $m$-elementary having diameter $1$, and all the outermost contours in $\Gamma$ were $k$-elementary so they are also $m$-elementary per \cref{rem:properties-of-elementary}.
    \end{proof}

    Each configuration in the set $(T_S \Gamma)_{S\subset V_\Gamma, \Gamma \in \mathscr{G}_{\eta,k,V}^\elem}$ has a unique pre-image in $\mathscr{G}_{\eta,k,V}^\elem$, since the preimage can be reconstructed by taking the corresponding SOS configuration, shifting all sites at heights $m-(k+1)$ up to $m$, then shifting the whole configuration down by $m-k$. Thus, for any $\theta>0$,  
    \begin{align*}
        Z_{\eta, m,V}^\elem \ge e^{ - \lambda m|V|} \sum_{\substack{\Gamma\in \mathscr{G}_{\eta, k,V}^\elem \\ |V_\Gamma|\ge \theta|\mathring{V}|}} \sum_{S\subset V_\Gamma} \prod_{\gamma \in T_S\Gamma}W(\gamma)\,.
    \end{align*}
    Also, by definition of the map $T_S$, we have 
    \begin{align*}
        \frac{\prod_{\gamma \in T_S\Gamma} W(\gamma)}{\prod_{\gamma \in \Gamma} W(\gamma)} \ge e^{ - (4\beta - \lambda)(k+1)|S|}\,.
    \end{align*}
    (This is an inequality rather than equality because some sites in $S$ could be adjacent.)
    Thus, for every $\theta>0$,  
    \begin{align*}
        \frac{Z_{\eta, m,V}^\elem}{Z_{\eta, k,V}^\elem} & \ge e^{ - \lambda m|V|} \sum_{\substack{\Gamma\in \mathscr{G}_{\eta, k,V}^\elem \\ |V_\Gamma|\ge \theta|\mathring{V}|}} \sum_{S\subset V_\Gamma} \frac{ \prod_{\gamma \in T_S(\Gamma)} W(\gamma)}{Z_{\eta, k,V}^\elem}  \ge e^{ - \lambda (m-k)|V|}\! \sum_{\substack{\Gamma\in \mathscr{G}_{\eta, k,V}^\elem \\ |V_\Gamma|\ge \theta|\mathring{V}|}} \sum_{S\subset V_\Gamma} e^{ - (4\beta - \lambda)(k+1)|S|} \mu_{\eta, k,V}^\elem(\Gamma) \\ 
        & = e^{ - \lambda(m-k)|V|} \sum_{\substack{\Gamma\in \mathscr{G}_{\eta, k,V}^\elem \\ |V_\Gamma|\ge \theta|\mathring{V}|}} \mu_{\eta, k,V}^\elem(\Gamma) \Big(1+e^{ - 4\beta-  \lambda)(k+1)}\Big)^{|V_\Gamma|} \\ 
        & \ge \exp\Big( - \lambda(m-k)|V| + \tfrac{1}{2} e^{ - (4\beta - \lambda)(k+1)} \theta |\mathring{V}|\Big) \mu_{\eta, k,V}^\elem(|V_\Gamma| \ge \theta|\mathring{V}|)\,.
    \end{align*}
    
    Taking $\theta = 3/4$, proving the following display would conclude the proof of the upper bound in \cref{eq:ratio-of-Z}: there is an absolute constant $C$ such that for $\beta>\beta_0$, and every $\eta$, we have 
        \begin{align}\label{eq:whp-VGamma-large}
            \mu_{\eta, k,V}^{\elem}(|V_\Gamma| \ge 3|\mathring{V}|/4) \ge 1-Ce^{ - (\beta - C)|\mathring{V}|^{1/3}}\,.
        \end{align}
        as it will in particular be greater than $0.9$ for every $V$. 
    The event $|V_\Gamma| \ge 3|\mathring{V}|/4$ is a subset of the union of 
    \begin{itemize}
        \item $\cB_1$: there exists a contour $\gamma$ confining a vertex in $\mathring{V}$ and having $|\gamma|\ge |\mathring{V}|^{1/2}$
        \item $\cB_2$: the subset of $\mathring{V}$ interior to outermost contours $\gamma$ with $|\gamma|\le |V|^{1/2}$ is at least $|\mathring{V}|/4$. 
    \end{itemize}
    We first bound the probability of $\cB_1$.   
    The number of choices for $\gamma$ having  $|\gamma|= r \ge |\mathring{V}|^{1/2}$ is at most  $|\mathring{V}| 4^r$, first for picking a site in $\mathring{V}$ to root the contour at, then for picking the contour $\gamma$. The probability of any fixed such $\gamma$ being an outermost contour is at most $\exp( - (\beta -3)r)$ per \cref{cor:outermost-elem-contour-exp-tail}. Thus, 
    \begin{align*}
        \mu_{\eta,k,V}^\elem(\cB_1) \le \sum_{r \ge |\mathring{V}|^{1/2}} |V| 4^r e^{ - (\beta -3)r} \le e^{ - (\beta - C)|\mathring{V}|^{1/2}}\,.
    \end{align*}
    Next, we consider the probability of $\cB_2$. For $\cB_2$ to happen, there must be a collection $G$ of outermost elementary contours each of which have  $|\gamma|\le |\mathring{V}|^{1/2}$, and for which 
\begin{align*}
    \sum_{\gamma \in G} |\Int(\gamma)\cap \mathring{V}| > \frac{1}{4}|\mathring{V}|\,.
\end{align*}
In particular, there must exist $\ell = 1,...,\frac{1}{2}\log_2 |\mathring{V}|$ such that if $\mathfrak U_\ell$ is the set of $\gamma\in G$ having $|\gamma|\in [2^{\ell-1},2^{\ell})$, 
\begin{align*}
    \sum_{\gamma \in \mathfrak U_\ell} |\Int(\gamma)\cap \mathring{V}| \ge \frac{|\mathring{V}|}{8 \ell^2}\,.
\end{align*}
Consider the probability of this for fixed $\ell$, and note that the above requires 
\begin{align}\label{eq:mathfrakU-needs-to-have}
    \sum_{\gamma\in \mathfrak U_\ell} |\gamma| \ge 4 \sum_{\gamma\in \mathfrak U_\ell} 2^{-\ell/2}|\Int(\gamma)| \ge \frac{|\mathring{V}|}{2\ell^2} 2^{- \ell/2}\,.
\end{align}
Counting over the possible choices for such $\mathfrak{U}_\ell$, and applying \cref{lem:Peierls-renormalized-elementaries} to bound the probability that $\mathfrak{U}_\ell$ are outermost contours, the probability that there exists $\mathfrak{U}_\ell$ having \cref{eq:mathfrakU-needs-to-have} is at most 
\begin{align*}
    \sum_{r\ge \frac{|\mathring{V}|}{2^{(\ell+1)/2} \ell^2}} 2^r \binom{|\mathring{V}|}{2^{(-\ell-1)}|\mathring{V}|} 4^r e^{ - (\beta -2)r} \le  e^{|\mathring{V}|\ell 2^{-(\ell-1)}} e^{- (\beta - 3) \frac{1}{2\ell^2} 2^{ - \ell/2} |\mathring{V}|}\,.
\end{align*}
If $\beta$ is large, for every $\ell$, this is at most $\exp(- \frac{\beta}{3\ell^2 2^{\ell/2}}|\mathring{V}|)$. 
The coefficient of $|\mathring{V}|$ is decreasing in $\ell$ and therefore, this is at most its value for $\ell=\frac{1}{2} \log_2 |\mathring{V}|$ which gives $e^{ - 4\beta |\mathring{V}|^{3/4}/(3 (\log_2 |\mathring{V}|)^{2})}$. 
This is at least $\exp( - \beta |\mathring{V}|^{1/2})$ as long as $|\mathring{V}|$ is larger than an absolute constant. The sum of the above probabilities thus yields the bound of \cref{eq:whp-VGamma-large}.

\medskip
\noindent \emph{Lower bound in \cref{eq:ratio-of-Z}}. 
In the other direction, define the identity map $T$ that acts on the subset of configurations in $\mathscr{G}_{\eta, k+1,V}^\elem$ that have $\{\varphi_V\ge 1\}$ and sends them to $\mathscr{G}_{\eta, k,V}$ by preserving the same contour collection, just viewing them all as being based at $1$ height lower. This is equivalent to shifting the height configuration down by $1$ everywhere, and thus $\{\varphi_V \ge 1\}$ implies that $T\varphi$ will respect the floor constraint. The core of the argument is to inductively estimate the fraction of $\mathscr{G}^\elem_{\eta,k+1,V}$ that has $\{\varphi_V \ge 1\}$ (for which this map is valid), as a way to bound the relative weight of $Z^\elem_{\eta,k+1,V}/Z^\elem_{\eta,k,V}$. 

A potential problem is that the image under this map may not be in $\mathscr{G}_{\eta, k,V}^\elem$ because contours can be $(k+1)$-elementary but not $k$-elementary. To this end, define the good event $\mathcal G$ on a contour collection (with a boundary height so that the contours have associated heights) as follows:  
\begin{align}\label{eq:cG-event}
\mathcal G: = \Big\{\Gamma\, : \, \forall \gamma \in \Gamma\,,\, |\gamma| \le e^{ 3\beta h_b(\gamma)}\Big\}\,.
\end{align}
where $h_b(\gamma)$ is its base height, i.e., $h_b(\gamma) = h(\gamma) + 1$ if $S(\gamma) = -$ and $h_b(\gamma) = h(\gamma) -1$ if $S(\gamma)=+$. 
Because outermost contours in $\mathscr{G}^\elem_{\eta,k+1,V}$ have base height $k+1$, and $\diam(\gamma)\le |\gamma|$, we have by definition of $k$-elementary, that 
\begin{align*}
    T(\mathcal G \cap \{\varphi_V \ge 1\}) \subset \mathscr{G}_{\eta, k,V}^{\elem}\,.
\end{align*}
Furthermore, $T$ is an injection because the pre-image can be recovered uniquely by shifting all heights up by $1$. As a consequence of this inclusion, and the simple calculation of the weight change $e^{ \lambda |V|}$ under application of the map $T$, for every $h$, 
\begin{align}\label{eq:part-function-change-T}
    \frac{Z_{\eta, h,V}^{\elem}}{Z_{\eta, h+1,V}^{\elem}} \ge e^{\lambda |V|} \mu_{\eta, h+1,V}^{\elem}(T^{-1}(\mathscr G_{\eta, h,V}^{\elem})) \ge e^{ \lambda |V|} \mu_{\eta, h+1,V}^{\elem} (\mathcal G, \varphi_V \ge 1)\,.
\end{align}
Our goal is now to establish for every $h$,
\begin{align}\label{eq:wts-positivity-probability}
    \mu_{\eta, h+1,V}^\elem(\mathcal G, \varphi_V \ge 1) \ge \exp\Big( - \tfrac{3}{2} |V|e^{ - (4\beta - \lambda)(h+1)}\Big)\,,
\end{align}
since plugging this bound into \cref{eq:part-function-change-T} and iterating over $h=k,...,m-1$ yields the claimed lower bound.

Henceforth, fix any $h$. The core content of \cref{eq:wts-positivity-probability}, is that the probability of $\mathcal G^c \cup \{\min_v \varphi_v =0\}$ is dominated by the event that at some single site $v\in V$ a singleton spike reaches down from height $h+1$ to height $0$, as that is what provides the rate $e^{ - (4\beta -\lambda)(h+1)}$. Towards this, for a simply-connected subdomain $U\subset V$ and point $y\in U$, we can define the event in $\mathscr{G}^\elem_{\zeta, h+1,U}$, 
$$\tilde G_y = \{\varphi_y \ge 1\} \cap \bigcap_{\gamma: y\in \Int(\gamma)} \{|\gamma|\le e^{ 3\beta h_b(\gamma)}\}\,.
$$
\begin{claim}\label{clm:one-point-probability}
Let $U \subsetneq V$ be simply connected. For any $\ell\ge 1$, $\zeta$ and every $y\in U$, we have 
    \begin{align*} \mu_{\zeta,\ell,U}^\elem(\tilde G_y^c) \le (1+e^{ -\beta}) e^{-(4\beta -\lambda)\ell}\,.
    \end{align*}
\end{claim}
\begin{proof}
Let us prove the claim inductively over $U$. The case $|U|=1$ is governed entirely by the event of $\{\varphi_y =0\}$ via singleton contours, which has probability at most $e^{ - (4\beta - \lambda)\ell}$. 
    Generally, let us decompose 
    \begin{align}\label{eq:want-to-bound-tildeGy}
        \mu_{\zeta,\ell,U}^\elem(\tilde G_y^c) \le \mu_{\zeta,\ell,U}^\elem(\tilde G_y^c, \cE^0_{\zeta,\ell,U})  + \sum_{j\ne 0} \mu_{\zeta,\ell,U}^\elem(\tilde G_y^c, \cE^j_{\zeta,\ell,U})\,.
    \end{align}
    where we recall the events $\cE^j_{\zeta,\ell,U}$ from \cref{eq:E^j} and use the shorthand $\cE^j$. 
    We begin with comparing the second term in \cref{eq:want-to-bound-tildeGy} back to the first, so that the dominant contribution comes on the event $\cE^0$. This will be helpful again in the future, so we isolate this reduction with a general event $A$ in place of $\tilde G_y^c$. 
    
    Let $j_0$ be the largest number such that $|\gamma_U|\le e^{ 3\beta (\ell-j_0 +1)}$ (this is the largest number of down $\gamma_U$ contours we can reveal while still having that conditional on those contours, the interior measure is the elementary one). 
    For $j > -j_0$, the probability of $A\cap \cE^j$ is bounded by the probability that there exist $|j|$ contours that are all $\gamma_U$, times the probability that under $\mu^\elem_{\sgn(j), \ell+ j, U}(A,\cE^0)$. For the $j\le -j_0$ contributions, their total probability is at most the probability of first having $j_0-1$ down contours, times the probability of having a contour of size $|\gamma_U| \ge e^{ 3\beta (\ell - j_0 +1)}$ (as well as $|\gamma_U|\ge 6$ since $|U|\ge 2$), whose probability by \cref{cor:outermost-elem-contour-exp-tail} is at most $e^{ - (\beta -2) e^{ 3\beta (\ell - j_0+1)}}$. 
    Together, these imply the bound 
    \begin{align*}
        \sum_{j\ne 0} \mu^\elem_{\zeta,\ell,U}(A,\cE^j) \le   \sum_{j\ne 0} e^{ - (\beta - 2) |j| |\gamma_U|} \mu_{\sgn(j),\ell +j,U}^\elem(A,\cE^0) +  \max_{j\ne 0} e^{ - (\beta -2)((|j|-1)|\gamma_U|+\max\{6,e^{ 3\beta (\ell-j+1)}\})}\,.
    \end{align*}
    For the first term, we can use the bound $|\gamma_U|\ge 6$. For the second term, we notice that $6(|j|-1) + \max\{6, e^{ 3\beta (\ell-j+1)}\}$ is at least $6\ell$ for all $j$. Together, these imply for a general event $A$ and $U: |U|\ge 2$, 
    \begin{align}\label{eq:reduction-to-cE^0}
        \sum_{j\ne 0} \mu^\elem_{\zeta,\ell,U}(A,\cE^j) \le \sum_{j\ne 0} e^{ - 6(\beta -2) |j|} \mu_{\sgn(j),\ell+j,U}^\elem(A,\cE^0) + e^{ - 6(\beta-2)\ell}\,.
    \end{align}
    The last term being at most $\frac{1}{4}e^{ - \beta}e^{ - (4\beta - \lambda)\ell}$, we have reduced \cref{eq:want-to-bound-tildeGy} to bounding $\mu_{\eta,\ell +j,U}^\elem(A,\cE^0)$. 
    
    In what follows, let $\gamma_y^\out$ be the outermost contour having $y\in \Int(\gamma_y^\out)$. 
    We have for general $\ell$,
    \begin{align}\label{eq:tilde-g-E^0-bound}
        \mu_{\zeta,\ell,U}^\elem(\tilde G_y^c, \cE^0) \le  \mu_{\zeta,\ell,U}^\elem( \varphi_y =0, |\gamma_y^\out| = 4) +\! \sum_{6\le r \le e^{ 3\beta \ell}} \mu_{\zeta,\ell,U}^\elem( \tilde G_y^c, |\gamma_y^\out| = r, \cE^0) + \mu_{\zeta,\ell,U}^\elem( |\gamma_y^\out| \ge e^{3\beta \ell})\,.
    \end{align}
    For the first term in \cref{eq:tilde-g-E^0-bound}, observe that if $|\gamma_y^\out|=4$ and $\varphi_y =0$, the only contours surrounding $y$ are exactly $\ell$ down contours that are each the singleton contour $\gamma_{\{y\}}$. By explicit calculation, the probability of such configurations is at most $e^{ - (4\beta - \lambda)\ell}$. 
    For the third term in \cref{eq:tilde-g-E^0-bound}, we can
    sum the probability of $\gamma_y^\out = \gamma$ over all possible $\gamma$ having $|\gamma|\ge e^{3\beta \ell}$ using \cref{cor:outermost-elem-contour-exp-tail}, to find that 
    \begin{align*}
        \sum_{r\ge e^{ 3\beta \ell}} \mu_{\zeta,\ell,U}^\elem( |\gamma_y^\out| = r) \le \sum_{r\ge e^{ 3\beta \ell}} 4^r e^{ - (\beta -3)r} \le e^{ - (\beta -5) \max\{6, e^{ 3\beta \ell}\}} \le e^{ - 6(\beta -5) \ell}\,.
    \end{align*}
    For the second term in \cref{eq:tilde-g-E^0-bound}, bounding the probability of $|\gamma_y^\out| = r$ similarly, and then conditioning on that outermost contour surrounding $y$, (using that for $r\le e^{ 3\beta \ell}$ its interior will necessarily only consist of $h(\gamma_y^\out)$-elementary contours), 
    \begin{align*}
        \sum_{r\ge 6} \mu_{\zeta,\ell,U}^\elem(\tilde G_y^c, |\gamma_y^\out| =r, \cE_{\zeta,\ell,U}^0) \le \sum_{r\ge 6} 4^r e^{ - (\beta -3)r} \max_{|\gamma_y^\out|\ge 6\,;\, \Int(\gamma_y^\out)\ne U} \max_{\zeta} \mu_{\zeta,\ell \pm 1,\Int(\gamma_y^\out)}^\elem(\tilde G_y^c)\,.
    \end{align*}
    Since $\Int(\gamma_y^\out)$ is strictly contained in $U$, we can use the inductive hypothesis on this last quantity, to get 
    \begin{align*}
        \sum_{r\ge 6} \mu_{\zeta,\ell,U}^\elem(\tilde G_y^c, |\gamma_y^\out| =r, \cE_{\zeta,\ell,U}^0) \le e^{ - 6(\beta -5)} (1+e^{-\beta})e^{ - (4\beta - \lambda)(\ell -1)} \le \frac{1}{4} e^{ -\beta}e^{ - (4\beta - \lambda) \ell}\,.
    \end{align*}
    Combining these three bounds, we get for general $\ell\ge 1$, and $\zeta$,  
    \begin{align}\label{eq:cE^0-bound}
        \mu_{\zeta,\ell,U}^\elem(\tilde G_y^c, \cE^0) \le (1+ \frac{1}{2}e^{ - \beta})e^{ - (4\beta - \lambda)\ell}\,.
    \end{align}
    Plugging this into \cref{eq:reduction-to-cE^0} with $\ell$ replaced by $\ell + j$ and $A$ replaced by $\tilde G_y^c$ gives 
    \begin{align*}
        \sum_{j\ne 0} \mu^\elem_{\zeta,\ell,U}(\tilde G_y^c,\cE^j) & \le \sum_{j\ne 0} e^{ - 6(\beta -2)|j|}(1+e^{- \beta}) e^{ - (4\beta - \lambda)(\ell+j)}  + \frac{1}{4}e^{ - \beta} e^{ - (4\beta - \lambda)\ell} \\ 
        & \le (1+e^{- \beta}) e^{ - (4\beta - \lambda)\ell} \sum_{j\ne 0} e^{ - 2(\beta -2)|j|} + \frac{1}{4} e^{ - \beta} e^{ - (4\beta - \lambda)\ell}\,,
    \end{align*}
    which is at most $\frac{1}{2} e^{ - \beta} e^{ - (4\beta - \lambda)}\ell$. Combined with \cref{eq:cE^0-bound} and plugged into \cref{eq:want-to-bound-tildeGy} we conclude.
 \end{proof}

In order to go from the one-point estimate to the probability that $\tilde G_y$ holds at all $y\in V$, we argue as follows. 
 (A boosting like this in the setting of the SOS model without $\elem$ could be done by means of FKG, but in the absence of FKG---$\mu^\elem$ does not satisfy this and further the event $\mathcal G$ is non-monotone---a revealing procedure is required.) 
For a configuration in $\mathscr{G}_{\eta, h+1,V}^\elem$, let 
\begin{align*}
    G_x = \{\varphi: \varphi(\Int(\gamma_x^\out)) \in \mathcal G\} \cap \{\varphi: \varphi(\Int(\gamma_x^\out)) \ge 1\}\,.
\end{align*}
The role of this event is that $$\mathcal G \cap \{\varphi_V \ge 1\} = \bigcap_{x\in V} G_x \qquad \text{and}\qquad G_x^c \subset \bigcup_{y\in \Int(\gamma_x^\out)} \tilde G_y^c\,.$$
Our goal is to establish that if we order the vertices of $V$ in such a manner that $V\setminus \bigcup\{y:y<x\}$ is always simply connected (e.g., repeatedly assign $x$ to be an arbitrary boundary vertex of $V\setminus \bigcup\{y:y<x\}$), then 
\begin{align}\label{eq:positivity-lower-bound}
    \mu^\elem_{\eta, h+1,V} (G_{x} \mid (G_{y})_{y<x}) \ge 1- (1+\epsilon_\beta) e^{- (4\beta -\lambda)(h+1)}\,.
\end{align}
from which the claim would follow since 
\begin{align*}
    \mu_{\eta, h+1,V}^\elem\Big(\bigcap_{x\in V} G_x \Big) = \prod_{x\in V} \mu_{\eta, h+1,V}^\elem\big(G_x \mid (G_y)_{y<x}\big) \ge \Big(1-(1+\epsilon_\beta)e^{ - (4\beta - \lambda)(h+1)}\Big)^{|V|}\,,
\end{align*}
where $\eta_i$ are the induced boundary signings by the exposure of the outermost contours confining $y<x$. 
This gives the claimed bound of \cref{eq:wts-positivity-probability} since $1-x\ge e^{-x}$ for $0<x<1/2$. 

In order to establish \cref{eq:positivity-lower-bound}, for ease of notation, let $$\nu_{h+1} = \mu_{\eta_x, h+1,V}^\elem (\,\cdot \mid (G_y)_{y<x},(\cF_y)_{y<x})\,,$$
where, by $\cF_y$ we mean the $\sigma$-algebra generated by the configuration interior to $\gamma_y^\out$. 
Notice that $\nu_{h+1}$ is exactly the measure $\mu_{\zeta,h+1,U}^\elem$ where $U = V\setminus \bigcup_{y<x} \Int(\gamma_{y}^\out)$ (which is $(\cF_y)_{y<x}$-measurable and is simply-connected by the ordering we imposed over $V$) and $\zeta$ agrees with $\eta$ on $\partial_i V\setminus \partial_o U$ and takes the opposite sign of $\gamma_{y}^\out$ on every vertex in $\partial_o U$. If $U$ is a singleton, the bound of \cref{eq:positivity-lower-bound} is immediate as the probability of $G_x^c$ would be $e^{ - (4\beta - \lambda)(h+1)}$ if a single-column spike down to height zero is permissible and $0$ otherwise. 

Suppose now that $|U|\ge 2$. Then we can naturally decompose 
\begin{align}\label{eq:split-G-x^c}
   \nu_{h+1}(G_x^c)\le \nu_{h+1}(G_x^c, \cE^0)  + \sum_{j\ne 0} \nu_{h+1}(G_x^c, \cE^j)\,,
\end{align}
and apply \cref{eq:reduction-to-cE^0} to get the bound  
\begin{align}\label{eq:j-bound-G-x^c}
    \nu_{h+1}(G_x^c) \le \sum_{j} e^{ - (\beta -2)|j|} \nu_{h+1+j}(G_x^c,\cE^0) + e^{ - 6(\beta -2)(h+1)}\,,
\end{align}
(where we use the shorthand $\nu_{h+1+j}$ for $\mu_{\sgn(j),h+1+j,U}$ when $j\ne 0$). 
We now examine $\nu_{h+1+j}(G_x^c,\cE^0)$. For general $\ell$, just as in \cref{eq:tilde-g-E^0-bound}, we can split
\begin{align*}
     \nu_\ell (G_x^c, \cE^0) \le \nu_\ell(\varphi_x =0, |\gamma_x^\out| = 4) + \sum_{6\le r\le e^{ 3\beta \ell}} \nu_\ell(G_x^c,|\gamma_x^\out|=r,\cE^0) + \nu_{\ell}(|\gamma_x^\out|\ge e^{3\beta \ell})\,.
\end{align*}
As argued following \cref{eq:tilde-g-E^0-bound}, the first term will contain the dominant contribution of at most $e^{ - (4\beta - \lambda)\ell}$, and the last term will be at most $e^{ - 6(\beta -5)\ell} \le \frac{1}{4} e^{-\beta} e^{ - (4\beta - \lambda)\ell}$. The middle sum is similarly at most 
\begin{align*}
    \sum_{r\ge 6} 4^r e^{ - (\beta-3)r}\max_{|\gamma_x^\out|\ge 6\,;\, \Int(\gamma_x^\out)\ne U} \max_{\zeta} \mu_{\zeta,\ell\pm 1,\Int(\gamma_x^\out)}^\elem(G_x^c)\,.
\end{align*}
Recalling that $G_x^c$ is a subset of $\bigcup_{y\in \Int(\gamma_x^\out)} \tilde G_y^c$, by a union bound and \cref{clm:one-point-probability},
\begin{align*}
     \sum_{6\le r\le e^{ 3\beta \ell}} \nu_\ell(G_x^c,|\gamma_x^\out|=r,\cE^0) \le \sum_{r\ge 6} e^{ - (\beta -4)r} r^2 (1+e^{- \beta}) e^{ - (4\beta - \lambda)(\ell-1)}\,.
\end{align*}
This sum is evidently at most $\frac{1}{4} e^{ - \beta} e^{ - (4\beta - \lambda)\ell}$. Altogether, we get 
\begin{align*}
    \nu_\ell (G_x^c, \cE^0) \le (1+ \tfrac{1}{2} e^{ - \beta}) e^{ - (4\beta - \lambda) \ell}\,. 
\end{align*}
Plugging this bound into \cref{eq:j-bound-G-x^c} with $\ell = h+1+j$, we get 
\begin{align*}
    \nu_{h+1}(G_x^c) \le \sum_j e^{ -(\beta -2)|j|} (1+\tfrac{1}{2} e^{- \beta}) e^{ - (4\beta - \lambda) (h+1+j)} + e^{ - 6(\beta -2) (h+1)} \le (1+e^{ - \beta}) e^{ - (4\beta - \lambda) (h+1)}\,,
\end{align*}
yielding the claimed \cref{eq:positivity-lower-bound} and concluding the proof. 
    \end{proof}

    \subsection{The elementary cluster expansion}
We can use \cref{thm:main-elementary-general-bc} to deduce that the elementary partition functions admit convergent cluster expansions. For a thorough background on cluster expansion, see~\cite[Chapter 5]{FreidliVelenik}, though we will only use fairly elementary consequences. We will work with the renormalized elementary partition function $Z^\rnelem_{\eta,k,V}$, though we recall that for simply-connected domains, this is the same as $Z^\elem_{\eta,k,V}$ per \cref{lem:elem-part-function-renormalized-weights}. For non-simply connected domains, we define $Z^\rnelem_{\eta,k,U}$ as in \cref{eq:elem-rnelem-partition-functions} (contours encircling holes of $U$ being admissible).

\begin{lemma}\label{lem:elementary-free-energies}
    For each $i$, the limit,
    \begin{align*}
        f^\elem_i = \lim_{n\to\infty} \frac{1}{|\Lambda_n|} \log Z_{\eta,i,\Lambda_n}^\elem\,,
    \end{align*} 
    exists. Furthermore, for all (not necessarily simply connected) $U$, all $i$ and all $\eta$, 
    \begin{align*}
        \big| \log Z_{\eta,i,U}^\rnelem  - |U| f^\elem_i \big| \le e^{-\beta/2} |\partial_e U|\,.
    \end{align*}
\end{lemma}

\begin{proof}
    The partition function $Z^\rnelem$ fits into the standard cluster expansion framework with hardcore interactions between contours (which serve as the polymers). Furthermore, per \cref{it:contour-weights} of \cref{thm:main-elementary-general-bc}, they satisfy the requisite exponential decay property on their weights for the Kotecky--Preiss condition~\cite{Kotecky-Preiss} to hold at large $\beta$. The lemma is then a direct corollary of convergence of the cluster expansion. We have used the bound $e^{ - \beta/2}$ on the sum of weights of  all possible elementary contours confining a vertex $v$ in their interior, which holds for $\beta$ larger than some absolute constant. 
\end{proof}

Using the bounds we obtained in \cref{it:Z-elem-ratio-general-bc} from \cref{thm:main-elementary-general-bc},  we can provide uniform lower bounds on $f^\elem_i-f^\elem_j$ in certain windows of the parameter $\lambda$. Namely, we get an infinite sequence of disjoint intervals $I_i$ in which the boundary condition $i$ elementary partition function dominates. Define 
\begin{align}\label{eq:I-i}
    I_i =  [\lambda_{-}^{(i)},\lambda_+^{(i)}] \quad\mbox{for}\quad
    \lambda_-^{(i)}:=e^{ - 4\beta i - 2\beta}\,,\quad \text{and}\quad \lambda_+^{(i)}:=e^{ - 4\beta i- \beta} \,.
\end{align}

\begin{lemma}\label{lem:free-energy-lower-bound-in-windows}
    If $\lambda \in I_i$, then for every $j\ne i$ we have 
    \begin{align*}
        f^\elem_i - f^\elem_j \ge e^{ - 4\beta( i\wedge (j+1)) - 3 \beta}\,.
    \end{align*}
\end{lemma}
\begin{proof}
    Applying \cref{it:Z-elem-ratio-general-bc} of \cref{thm:main-elementary-general-bc} with $\lambda \in I_i$, we find that for every $k<m$, 
    \begin{align}\label{eq:elem-FE-bounds}
        \lambda (m-k) - 2e^{ - (4\beta - \lambda)(k+1)} \le \frac{1}{|V|} \log \frac{Z_{\eta,k,V}^\elem}{Z_{\eta,m,V}^\elem} \le \lambda(m-k) - \frac{1}{2}\frac{|\mathring{V}|}{|V|} e^{ - (4\beta -\lambda)(k+1)} + \frac{\log 2}{|V|}\,.
    \end{align}
    Taking $k=i$ and $m=j>i$, we see that 
    \begin{align*}
    \frac{1}{|V|} \log \frac{Z_{\eta,i,V}^\elem}{Z_{\eta,j,V}^\elem} \ge \lambda (j-i)(1-2e^{ - (4\beta -\lambda)(i+1) + 4\beta i + 2\beta}) \ge \lambda(j-i)(1-2e^{ - 4\beta +e^{ - 4\beta i}(i+1) + 2\beta})\,.
    \end{align*}
    Since $e^{ - 4\beta i}(i+1) \le 1$, we see that for $\beta$ large, this satisfies   
    \begin{align*}
         \frac{1}{|V|} \log \frac{Z_{\eta,i,V}^\elem}{Z_{\eta,j,V}^\elem} \ge \lambda(j-i)(1-\epsilon_\beta)\,,
    \end{align*}
    and $\epsilon_\beta$ can be taken to be $e^{ - \beta}$. Since this bound is uniform over $V$, we can take $V$ to be sequences of boxes going to $\mathbb Z^2$ to get the lower bound in the limit. 

    If $j<i$, the upper bound of \cref{eq:elem-FE-bounds} with $k=j$ and $m=i$ implies 
    \begin{align*}
        \frac{1}{|V|}\log \frac{Z_{\eta,j,V}^\elem}{Z_{\eta,i,V}^\elem} \le e^{ - 4\beta i - \beta} (i-j) - \tfrac{1}{2} e^{- 4\beta (j+1)}+ O(\tfrac{|\partial V|}{|V|})\,.
    \end{align*}
    Since $j<i$, we can write $e^{ - 4\beta i - \beta} = e^{ - 4\beta (j+1)} e^{ - 4\beta (i-j-1) - \beta}$ and use the fact that $e^{ - 4\beta(i-j-1)- \beta}(i-j) - \frac{1}{2}$ is at most $- \frac{1}{2} + e^{ - \beta} \le -e^{ - \beta}$ uniformly over $i-j\ge 1$, yielding 
    \begin{align*}
        \frac{1}{|V|}\log \frac{Z_{\eta,j,V}^\elem}{Z_{\eta,i,V}^\elem} \le - e^{- 4\beta (j+1) -\beta} + O(\tfrac{|\partial V|}{|V|})\,.
    \end{align*}
    Taking $V$ to be sequences of boxes going to $\mathbb Z^2$, and negating the above, we get the claimed bound. 
\end{proof}

\section{An infinite sequence of critical external field vales}\label{sec:infinite-sequence-lambda}
The aim in this section is to establish an infinite sequence of discontinuous phase transitions as one varies the external field $\lambda$, with critical points $(\lambda_c^{(k)})_{k\ge 0}$ ordered $\lambda_c^{(0)}>\lambda_c^{(1)}>\ldots$ and accumulating at zero, such that for all $\lambda \in [\lambda_c^{(k)},\lambda_c^{(k-1)}]$ the SOS measure with boundary conditions $k$ is thermodynamically stable. In particular, if one starts with boundary conditions $k$ on any domain, then the interface will remain rigid about height $k$ in the bulk of that domain whenever $\lambda \in [\lambda_c^{(k)},\lambda_c^{(k-1)}]$. 

Moreover, and crucially for the purposes of this paper, we probe the precise behavior as $\lambda$ approaches the critical points, demonstrating that both of the heights $k$ and $k-1$ will be stable as long as the domain is smaller than $|\lambda-\lambda_c^{(k-1)}|^{-1}$. It is in this and the following section that the true benefits of the contour representation and our more geometric arguments arise, allowing us to use much shorter arguments to get sharp results analogous to~\cite{CeMa1} for \emph{all} $\lambda \ge 0$ rather than only for $\lambda\ge c_\beta>0$ fixed and bounded away from zero.

\subsection{Rigidity of height \texorpdfstring{$h$}{h} when \texorpdfstring{$\lambda \in I_h$}{lambda is in I\_h}}
For each $h$, the full window $[\lambda_c^{(h)},\lambda_c^{(h-1)}]$ will contain the interval $I_h$ from \cref{eq:I-i}, so we begin by establishing the rigidity in the smaller interval away from the critical points. 
A general contour collection in $\mathscr{G}_{\eta, h,V}^\rn$ consists both of $h$-elementary contours, and ones that are not $h$-elementary. The above estimates give us the ability to prove that even the contours that are not $h$-elementary have exponentially small $h$-renormalized weight when $\lambda \in I_h$. 

\begin{lemma}\label{lem:non-elem-contour-bound}
    If $\lambda \in I_h$, then for every (not necessarily elementary) $\gamma$, we have 
    \begin{align*}
         W^\rn_h(\gamma)\le e^{- (\beta -4)|\gamma|}\,.
    \end{align*}
\end{lemma}

The proof of \cref{lem:non-elem-contour-bound} goes by treating nested collections of non $h$-elementary contours together, so that the partition functions on their complements are exactly the elementary ones, for which we know that being at height $h$ is the most preferred when $\lambda \in I_h$, per \cref{lem:elementary-free-energies}. The key step in the argument is to bound the total weight contribution over all choices of non-elementaries that can be nested in $\gamma$, noticing that their being non-elementary is enough to beat out the entropy over the possible locations they can be placed (also using that if they are far from the boundary of their nesting contour, there is a cost due to the height-$h$ elementary free energy being the dominant one). This argument was performed in \cite[Lemmas 2.13--2.14]{DiMa94} for their cylinders rather than our contours, using an encoding of \emph{clusters of non-elementary cylinders} into weighted trees. With only minor modifications, this argument can be imported into our setting of contours, and we therefore defer the details to \cref{sec:exp-tails-nonelem-renormalized-weights}.

    Given \cref{lem:non-elem-contour-bound}, we can deduce convergence of the cluster expansion for $Z_{h,V}$ (see the next subsection) as well as the following rigidity of the (full) SOS measure with $h$-boundary conditions when $\lambda \in I_h$. 
    \begin{lemma}\label{lem:rigidity-in-I-windows}
        Suppose $\lambda \in I_h$. For every $V$, every $\eta$, and every $x\in V$,
        \begin{align*}
            \mu_{\eta,h,V}(|\varphi_x - h| \ge r) \le e^{ - 4(\beta-5)r}\,.
        \end{align*}
    \end{lemma}
    \begin{proof}
        It suffices to bound the probability that there is a set of all contours nesting $x$ have total size at least $4r$. The bound on this goes via a standard Peierls estimate, with the set of all such configurations being mapped, via deletion of their outermost contour, to the set of configurations whose set of all contours nesting $x$ has total size at least $4(r-1)$ (since $|\varphi_x - h|$ would change by $1$ under this operation). If in the contour representation of $\varphi$, $\gamma_x^\out$ is the outermost contour confining $x$, we can enumerate over all such choices of the contour having size $r\ge 4$, there being at most $4^r$ many such choices, and recall from \cref{lem:part-function-renormalized-weights} that the probability of a specific contour $\gamma$ being $\gamma_x^\out$ is exactly $W^\rn_{h}(\gamma)$. Together with \cref{lem:non-elem-contour-bound}, we get 
        \begin{align*}
        \mu_{\eta,h,V}\Big(\sum_{\gamma: x\in \Int(\gamma)}|\gamma| \ge 4r\Big) \le \sum_{r\ge 4} 4^r e^{ - (\beta -4)r} \mu_{\eta,h,V}\Big(\sum_{\gamma: x\in \Int(\gamma)}|\gamma| \ge 4(r-1)\Big)\,,
        \end{align*}
        which, when iterated implies the claimed bound for large $\beta$. 
    \end{proof}

\subsection*{Truncated free energies and cluster expansion}
Having established exponential tails not only on elementary $\gamma$, but also when $\lambda \in I_h$, non-elementary $\gamma$ based at height $h$, we can deduce that the full partition function $Z_{\eta,h,V}$ admits a convergent cluster expansion when $\lambda \in I_h$. In order to compare this with other partition functions (which may not admit convergent cluster expansions), we introduce ``truncated" partition functions which modify the renormalized weights so that they admit convergent cluster expansions at all reference heights and all $\lambda$. These will coincide with the usual partition function whenever the usual partition function has exponential tails on its renormalized weights.

\begin{definition}
    For an contour $\gamma$, define its $h$-truncated weight 
    \begin{align}\label{eq:W-tr}
         W^{\tr}_h(\gamma) = \min\big\{  W^\rn_h(\gamma), e^{ - (\beta - 5)|\gamma|} \big\}\,.
    \end{align}
\end{definition}

As a consequence of \cref{it:contour-weights} of \cref{thm:main-elementary-general-bc}, we have for every elementary $\gamma$, $W^\tr_h(\gamma) =  W^\rn_h(\gamma)$. 
Also, by \cref{lem:non-elem-contour-bound}, if $\lambda \in I_h$, then every non-elementary contour $\gamma$ also has $ W_h^\tr(\gamma) =  W^\rn_h(\gamma)$. We can now use the truncated weights to construct truncated partition functions. 

\begin{definition}
    For general $\eta,h,V$, define the truncated partition function 
    \begin{align*}
        Z_{\eta,h,V}^\tr = e^{ - \lambda h |V|} \sum_{\Gamma\in \mathscr G_{\eta,h,V}^\rn} \prod_{\gamma\in \Gamma}  W_h^\tr(\gamma)\,.
    \end{align*}    
\end{definition}

By definition of $W^\tr_h$ and $Z_{\eta,h,V}^\tr$, we always have 
\begin{align}\label{eq:Z-tr-leq-Z-rn}
    Z_{\eta,h,V}^\tr \le Z_{\eta,h,V}^\rn\,,
\end{align}
By the earlier observation that $W^\tr_h(\gamma) = W^\rn_h(\gamma)$ for all $\gamma$ when $\lambda\in I_h$, we have 
\begin{align}\label{eq:Z-tr-eq-Z-rn}
    Z_{\eta,h,V}^\tr = Z_{\eta,h,V}^\rn \qquad \text{for all $\lambda \in I_h$}\,.
\end{align}
By \cref{lem:part-function-renormalized-weights}, if $V$ is simply connected both the right-hand sides above can be replaced by $Z_{\eta,h,V}$. 

Below, we collect some standard consequences of convergent cluster expansions that we get for the truncated partition functions. Towards this, define the following finite-volume free energies: 
\begin{align}\label{eq:finite-vol-free-energies}
            f^\tr_{\eta,h,V} = \frac{1}{|V|} \log Z^\tr_{\eta,h,V}\,, \qquad\text{and}\qquad f_{\eta,h,V}= \frac{1}{|V|}\log Z_{\eta,h,V}\,.
\end{align}

As argued in \cref{lem:elementary-free-energies}, the uniform $e^{ - (\beta -5)|\gamma|}$ bound on truncated contour weights implies that the truncated partition functions admit convergent cluster expansions, yielding the following. 

\begin{lemma}\label{lem:truncated-partition-function-cluster-expansion}
    For each $h$, the following infinite-volume free energy
    \begin{align*}
        f^\tr_h:= \lim_{n\to\infty} f^\tr_{\eta,h,\Lambda_n}\,,
    \end{align*}
    exists. Furthermore, for all (not necessarily simply connected) $U$, all $h$, and all $\eta$, 
    \begin{align*}
        |\log Z_{\eta,h,U}^\tr - |U|f^\tr_h| \le e^{ - \beta/2} |\partial_e U|\,.
    \end{align*}
\end{lemma}

At the same time, due to \cref{eq:Z-tr-eq-Z-rn} and \cref{lem:part-function-renormalized-weights}, when $\lambda \in I_h$,  
\begin{align*}
    f_h := \lim_{n\to\infty} \frac{1}{|\Lambda_n|} \log Z_{\eta,h,\Lambda_n}\,,
\end{align*}
exists and is equal to $f^\tr_h$ because the partition functions are equal. Since 
\begin{align*}
   \Big| \frac{1}{|V|}\log \frac{Z_{\eta, j,V}}{Z_{\zeta,h,V}} \Big| \le  (\beta + \lambda)\cdot|h-j|\cdot\frac{|\partial_e V|}{|V|}\,,
\end{align*}
(by crudely forcing the heights along the boundary to be at $h$), we find that when $\lambda \in I_h$, 
\begin{align}\label{eq:f_i-eq-f_j}
    f_j = \lim_{n\to\infty} f_{\eta,j,\Lambda_n} \quad \text{exists, and} \quad f_j = f_h= f_h^\tr \text{ for all $j\ne h$}\,.
\end{align}

We now show using the truncated free energies, that there exist an infinite sequence of critical values of $\lambda$ dictating transitions between which height maximizes the infinite-volume truncated free energy. 

\begin{proposition}\label{prop:lambda_c}
    There exists $(\lambda_c^{(i)})_{i=0}^\infty>0$ such that for every $i$, it has $I_i \subset [\lambda_c^{(i)}, \lambda_c^{(i-1)}]$ and
    \begin{enumerate}
        \item \label{it:between_lambda_c} If $\lambda_c^{(i)}<\lambda<\lambda_c^{(i-1)}$, then $f^\tr_i>f^\tr_j$ for all $j\ne i$\,.
        \item \label{it:existence_lambda_c} If $\lambda= \lambda_c^{(i)}$, then $f^\tr_{i+1} = f^\tr_i >f^\tr_j$ for $j\notin \{i,i+1\}$. 
    \end{enumerate}
\end{proposition}

\begin{proof}
    We first use  \cref{eq:f_i-eq-f_j}, then \cref{eq:Z-tr-leq-Z-rn} to deduce that for every $\lambda \in I_i$, we have for all $j$, 
    \begin{align}\label{eq:f-tr-i-ge-f-tr-j}
        f^\tr_i = f_i = f_j \ge f^\tr_j\,.
    \end{align}
    We now lower bound the $\lambda$-derivative of $f_i^\tr - f_j^\tr$ to deduce that there is a single $\lambda$ at which the former becomes strictly larger than the latter. 
    We start by considering the derivative for finite volumes. (The boundary signing can be arbitrary here, so let us take it as free and drop it from the notation.) Explicit differentiation gives
    \begin{align*}
        \Big|\frac{d}{d\lambda} f_{i,V}^\tr + i\Big| \le \frac{1}{|V|}\mu_{i,V}^\tr\big[ \sum_{\gamma \in \Gamma} |\Int(\gamma)| \big]\,,
    \end{align*}
    (where when the derivative does not exist due to the $\min$ in $W^\tr(\gamma)$, this bound holds for both the right and left derivatives). 
    By standard cluster expansion reasoning, the total volume confined by $\gamma\in \Gamma$ has an exponential tail beyond $e^{ - \beta} |V|$, and therefore the expectation on the right above is at most $e^{ - \beta/2}|V|$, uniformly over $\lambda$. This implies that for any simply connected $V$, for all $i,j$
    \begin{align}\label{eq:dlambda-ftr}
        \Big|\frac{d}{d\lambda}(f^\tr_{i,V}- f^\tr_{j,V})  - (j-i)\Big| \le e^{ - \beta/2}\,.
    \end{align}
    Taking $V= \Lambda_n$ and $n\to \infty$, this implies that $f^\tr_i,f^\tr_j$ are (Lipschitz-)continuous in $\lambda$, and for all $\lambda>0$, we have for all $i,j$
    \begin{align*}
        \frac{d}{d\lambda} (f^\tr_i - f^\tr_j) \ge (1-\epsilon_\beta) (j-i)\,.
    \end{align*}
    (Again, when there are issues of non-differentiability, this lower bound holds both for the right and left derivatives.) 
    When this lower bound is combined with \cref{eq:f-tr-i-ge-f-tr-j}, the desired sequence of critical values $\lambda_c^{(i)}$ and the claimed strict inequalities on truncated free energies follows by continuity in $\lambda$. 
\end{proof}

    \cref{eq:dlambda-ftr} also gives the following relation between $f_k^\tr - f_{k\pm 1}^\tr$ and $\dl^\pm(\lambda)$ from \cref{eq:dl}. 

\begin{corollary}\label{cor:delta-f-delta-l-relation}
    We have for every $k$ that for $\lambda <\lambda_c^{(k-1)}$, 
    \begin{align*}
       (1-\epsilon_\beta) (\lambda_c^{(k-1)}-\lambda) \le f_k^\tr - f_{k-1}^\tr \le (1+\epsilon_\beta) (\lambda_c^{(k-1)}-\lambda)\,, 
    \end{align*}
    and similarly, for $\lambda>\lambda_c^{(k)}$, we have $(1-\epsilon_\beta) (\lambda -\lambda_c^{(k)}) \le f_k^\tr - f_{k+1}^\tr \le (1+\epsilon_\beta) (\lambda - \lambda_c^{(k)})$.
\end{corollary}

In what follows, to simultaneously consider the $i=0$ case, by convention set $\lambda_c^{(-1)} = \infty$. 

\subsection{Extending control of non-elementary contours to 
\texorpdfstring{$\lambda\in [\lambda_c^{(h)},\lambda_c^{(h-1)}]$}
{lambda is in [lambda\_c(h),lambda\_c(h-1)]}.}
Our aim is now to prove that the rigidity at height $h$ (e.g., exponential tails on $W^\tr_h(\gamma)$ even for non-elementary $\gamma$) holds throughout the full window $[\lambda_c^{(h)}, \lambda_c^{(h-1)}]$ where \cref{prop:lambda_c} showed the $h$-free energy dominates.  

From this point forth, our arguments deviate much more significantly, even at the high level, from the strategy of~\cite{CeMa1} as their arguments became more constrained by the usage of cylinders instead of contours, and begin to necessitate $\lambda$ bounded away from zero by a $\beta$-dependent constant, or equivalently necessitate the boundary level bounded by a $\beta$-dependent constant. By contrast, our arguments are significantly shorter than those in~\cite{CeMa1} and allow us to interpolate all the way down to $\lambda=0$.   

\begin{proposition}\label{prop:non-elem-control-entire-window}
    Let $\lambda \in [\lambda_c^{(k)},\lambda_c^{(k-1)}]$. For every (including non-elementary) contour $\gamma$, 
    \begin{align}\label{eq:W-tr-eq-W-rn}
        W_k^\tr (\gamma) =  W^\rn_k(\gamma)\,.
    \end{align}
    In particular, if $\lambda \in [\lambda_c^{(k)},\lambda_c^{(k-1)}]$, then 
        \begin{align}\label{eq:Z-rn-eq-Z-tr-full-window}
            Z_{\eta,k,V}^\rn = Z_{\eta,k,V}^\tr\,.
        \end{align}
    and if $V$ is simply connected, $Z_{\eta,k,V} = Z_{\eta,k,V}^\tr$. 
\end{proposition}

\begin{proof}
    We prove the proposition by induction. \cref{eq:W-tr-eq-W-rn} is already proved for elementary contours in \cref{thm:main-elementary-general-bc} so it suffices to show it for $\gamma$ that is not $k$-elementary. 
    Suppose \cref{eq:W-tr-eq-W-rn} has been shown for every contour $\gamma'$ having $|\Int(\gamma')|\le m$ and show it for $\gamma$ such that $|\Int(\gamma)| = m+1$. For ease of notation, let $V\cup\{x\} = \Int(\gamma)$ where $x$ is a prescribed point on the boundary $\partial_i(V\cup \{x\})$ such that $V$ is simply connected. Recalling the definition \cref{eq:W-tr}, if $\gamma$ is an up-contour, we wish to show
        \begin{align}\label{eq:wts-nonelem-induction}
        \frac{Z_{+,k+1,V\cup x}}{Z_{+,k,V\cup x}} \le e^{|\gamma|}\,.
    \end{align}
    If we let $\tilde +$ be the boundary signing on $\partial_i V$ that is $+$ on all vertices in $\partial_i (V\cup x)$, and free on $\partial_i V\setminus \partial_i (V\cup x)$, the first bound we make in this direction is     
    \begin{align*}
          \frac{Z_{+,k+1,V\cup x}}{Z_{+,k,V\cup x}} \le e^{\lambda k} \frac{Z_{+,k+1,V\cup x}}{Z_{\tilde +,k,V}} = e^{ \lambda k} \frac{Z_{\tilde +,k+1,V}}{Z_{\tilde +,k,V}} \frac{Z_{+,k+1,V\cup x}}{Z_{\tilde +,k+1,V}}\,,
    \end{align*}
    using in the first inequality the fact that $Z_{+,k,V\cup x}\ge e^{ - \lambda k} Z_{\tilde +,k,V}$. 
    The second ratio on the right is exactly 
    $$\big(\mu_{+,k+1,V\cup x}(\varphi_{x} = k+1)\big)^{-1} = \big(\mu_{+,k+1,V\cup x}(\varphi_{x} \le k+1)\big)^{-1}\,.$$ 
     By monotonicity, the probability of $\{\varphi\le k+1\}$ is only increasing as we increase $\lambda$, and so the right-hand side above increases as we decrease $\lambda$. Therefore, we upper bound the right-hand side by decreasing $\lambda$ to $\lambda'\in I_{k+1}$. At that point, the rigidity ~\cref{lem:rigidity-in-I-windows} implies the above display is at most $1+\epsilon_\beta$. 

    Plugging this bound in, we are left with 
    \begin{align}\label{eq:Z-ratio-intermediate-step}
        \frac{Z_{ +,k+1,V\cup x}}{Z_{ +,k,V\cup x}} \le (1+\epsilon_\beta) e^{\lambda k} \frac{Z_{\tilde +,k+1,V}}{Z_{\tilde +,k,V}}\,.
    \end{align}
    For the ratio on the smaller domain $V$, we first use \cref{lem:monotonicity-part-function-ratio} (saying that the ratio is decreasing in $\lambda$) to increase the ratio by decreasing $\lambda$ to $\lambda_c^{(k)}$, at which, by the inductive hypothesis, both partition functions are exactly their truncated versions. 
   By definition of $\lambda_c^{(k)}$ and \cref{prop:lambda_c}, we have $f^\tr_{k+1} = f^\tr_{k}$ at $\lambda_c^{(k)}$, whence 
    by \cref{lem:truncated-partition-function-cluster-expansion}, 
    \begin{align*}
        \frac{Z^\tr_{\lambda_c^{(k)},\tilde +,k+1,V}}{Z^\tr_{\lambda_c^{(k)},\tilde +,k,V}} \le \frac{Z^\tr_{\lambda_c^{(k)},\tilde +,k+1,V}}{Z^\tr_{\lambda_c^{(k)},\tilde +,k,V}} \le e^{ 2e^{ - \beta/2}|\partial_e V|}\,.
    \end{align*}
    Plugging this bound into \cref{eq:Z-ratio-intermediate-step}, we are left with 
    \begin{align*}
         \frac{Z_{+,k+1,V}}{Z_{\tilde +,k,V}}  \le (1+\epsilon_\beta)e^{ \lambda k} e^{ 2e^{ - \beta/2}|\partial_e V|}\,.
    \end{align*}
    Since $\lambda\le e^{- 4\beta k+2\beta}\wedge 1$, we have $(1+\epsilon_\beta) e^{\lambda k}\le e^2$ say.  
    Since we are assuming the confining contour $\gamma$ bounding the region $V\cup x$ is not $k$-elementary, it has $|\gamma|\ge e^{ 3\beta}$, so for $\beta$ large,
    \begin{align*}
        2e^{ - \beta/2} |\partial_e V| + 2 \le |\partial_e (V\cup x)|\,,
    \end{align*}
    which together with the above implies the claimed bound of \cref{eq:wts-nonelem-induction} since $|\partial_e (V \cup x)| = |\gamma|$. 

    If $\gamma$ is instead a negative contour, the argument proceeds analogously. Here, $$(\mu_{-,k-1,V\cup x}(\varphi_x \ge k-1))^{-1}\,,$$ will  be increasing as we increase $\lambda$, and taking $\lambda$ to $\lambda'\in I_{k-1}$ is an increase of $\lambda$. Similarly, the ratio $Z_{\tilde -,k-1,V}/Z_{\tilde -,k,V}$ will be increasing in $\lambda$, so we can increase $\lambda$ to $\lambda_c^{(k-1)}$ as needed. 
\end{proof}

    We can also deduce the following, showing exponential tails on all contours based at heights $k\pm 1$ if they are small as compared to $\dl(\lambda_n)$ from \cref{eq:dl} so the free energy difference does not overcome the cost in the boundary. Later, this will be crucial to understanding bottlenecks and the formation of critical droplets when $\lambda$ is microscopically close to a critical value. 
    
\begin{corollary}\label{cor:non-elem-control-smaller-than-delta-tr}
    Suppose $\lambda \in [\lambda_c^{(k)},\lambda_c^{(k-1)}]$. If $\gamma$ is a down contour and $\diam(\gamma)\le (f^\tr_k - f^\tr_{k+1})^{-1}$, then $ W^\rn_{k+1}(\gamma) =  W_{k+1}^\tr(\gamma)$, and if $\gamma$ is an up contour with $\diam(\gamma)\le (f^\tr_k - f^\tr_{k-1})^{-1}$ then $ W^\rn_{k-1}(\gamma) =  W_{k-1}^\tr(\gamma)$. 
\end{corollary}

\begin{proof}
        Consider $\gamma$ a down contour having $\diam(\gamma) \le (f^\tr_k - f^\tr_{k+1})^{-1}$. For these, we have 
    \begin{align*}
        \frac{Z_{-,k,\Int(\gamma)}}{Z_{-,k+1,\Int(\gamma)}} \le \frac{Z^\tr_{-,k,\Int(\gamma)}}{Z_{-,k+1,\Int(\gamma)}^\tr}\,,
    \end{align*}
    where we used \cref{eq:Z-rn-eq-Z-tr-full-window} on the numerator, and \cref{eq:Z-tr-leq-Z-rn}. By \cref{lem:truncated-partition-function-cluster-expansion} applied to both terms on the right, 
\begin{align*}
     W^\rn_{k+1}(\gamma) \le \exp \big( - \beta|\gamma| + (f^\tr_k - f^\tr_{k+1})^{-1} |\Int(\gamma)| + 2e^{ - \beta/2}|\gamma|\big)\,. 
    \end{align*}
    Using the assumption that $\diam(\gamma)\le (f^\tr_k - f^\tr_{k+1})^{-1}$, and $|\Int(\gamma)| \le \diam(\gamma) \cdot |\gamma|$, we get that this is at most $\exp( - (\beta - 2)|\gamma|)$ as claimed. The reasoning for a up contour from $k-1$ is analogous.  
\end{proof}

We will also use the following more technical corollary deducing for the full, as opposed to truncated, partition functions that up to an error of $\epsilon_\beta |\partial_e U|$, the full partition function on $U$ with boundary $k$ is dominant when $\lambda \in [\lambda_c^{(k)},\lambda_c^{(k-1)}]$ (cf.\ the crude error size $\beta |\partial_e U|$ one gets by forcing the configuration values along the boundary).

\begin{corollary}\label{cor:part-function-comparison-full-window}
    Consider an arbitrary connected domain $U$ (possibly not simply connected) and any boundary signings $\eta$ and $\zeta$. If $\lambda \in [\lambda_c^{(k)},\lambda_c^{(k-1)}]$, then for all $j$, 
    \begin{align*}
        \frac{Z_{\eta,j,U}}{Z_{\zeta,k,U}}\le e^{ e^{-\beta/3}|\partial_e U| \, |j-k|}\,.
    \end{align*}
\end{corollary}

\begin{proof}
    Suppose first that $j = k+1$. By \cref{lem:monotonicity-part-function-ratio}, we can decrease $\lambda$ to $\lambda_c^{(k)}$, only increasing the ratio between the partition functions. Since \cref{lem:part-function-renormalized-weights}  only applied to simply connected domains, we cannot move to $Z^\tr$ immediately. Instead, if we call $H_i$ the holes of $U$ 
    on the one hand, at $\lambda_c^{(k)}$,
    \begin{align*}
        Z_{\eta,k+1,U} \le \frac{Z_{\eta,k+1, U \cup \bigcup_i H_i}}{\prod_i Z_{\eta,k+1,H_i}} = \frac{Z_{\eta,k+1, U \cup \bigcup_i H_i}^\tr}{\prod_i Z_{\eta,k+1,H_i}^\tr}\,,
    \end{align*}
    where the inequality is by inclusion, and the equality uses \cref{prop:non-elem-control-entire-window} and simply-connectedness of $U \cup \bigcup_i H_i$, and of $H_i$. By \cref{lem:truncated-partition-function-cluster-expansion}, therefore, 
    \begin{align}\label{eq:Z-k+1-not-sc}
         Z_{\eta,k+1,U} \le \exp\Big( f^\tr_{k+1} |U| + e^{ - \beta/2} (|\partial_e (U\cup \bigcup_i H_i)| + \sum_i |\partial_e H_i|)\Big)\,.
    \end{align}
    
    On the other hand, if we define $\bar Z_{\zeta,k,U}$ as the restricted partition function where no contours confining any hole $H_i$ in their interior are allowed, then 
    \begin{align*}
        Z_{\zeta,k,U} \ge \bar Z_{\zeta,k,U}  = \bar Z^\rn_{\zeta,k,U} = \bar Z^\tr_{\zeta,k,U}\,.
    \end{align*}
    From the convergence of the cluster expansion for $Z_{\zeta,k,U}^\tr$, it is not hard to deduce that 
    \begin{align*}
        \frac{\bar{Z}_{\zeta,k,U}^\tr}{Z_{\zeta,k,U}^\tr} \ge \exp \big( - e^{ - \beta/2} \sum_i |\partial_e H_i|\big)\,.
    \end{align*}
    Combining the above two, and applying \cref{lem:truncated-partition-function-cluster-expansion} to $Z_{\zeta,k,U}^\tr$, we get 
    \begin{align*}
        Z_{\zeta,k,U} \ge \exp\Big( f^\tr_{k} |U| - e^{ - \beta/2}(|\partial_e U| + \sum_i |\partial_e H_i|)\Big)\,.
    \end{align*}
    Combined with \cref{eq:Z-k+1-not-sc} and the fact from \cref{prop:lambda_c} that the limiting truncated free energies $f^\tr_{k+1}$ and $f^\tr_k$ are equal at $\lambda_c^{(k)}$, we get 
    \begin{align*}
        \frac{Z_{\eta,k+1,U}}{Z_{\zeta,k,U}} \le \exp\Big( e^{- \beta/2} \big(|\partial_e U| + 2 \sum_i |\partial_e H_i| + |\partial_e (U \cup \bigcup_i H_i)|\big)\Big)\,.
    \end{align*}
    That sum of the various boundary terms counts each edge in $\partial_e U$ at most three times, leaving that the above is at most $e^{ 3 e^{ - \beta/2}|\partial_e U|} \le e^{ e^{ - \beta/3} |\partial_e U|}$ at large $\beta$. 
    If $j= k+\ell$ more generally, we first express 
    \begin{align*}
        \frac{Z_{\eta,j,U}}{Z_{\zeta,k,U}} = \frac{Z_{\eta,k+\ell,U}}{Z_{\eta,k+\ell-1,U}} \cdots \frac{Z_{\eta,k+1,U}}{Z_{\zeta,k,U}}\,.
    \end{align*}
    We can then evaluate the last term at $\lambda_c^{(k+1)}$, the second-to-last one at $\lambda_c^{(k+2)}$, etc, each of these only increasing the ratio of partition functions per \cref{lem:monotonicity-part-function-ratio}, and then being calculated as above. 
    The other direction, where $j<k$ is analogous, with the increase of $\lambda$ to $\lambda_c^{(k-1)}$ increasing the partition function ratio.  
\end{proof}

\subsection{A geometric rigidity statement about height 
\texorpdfstring{$k$}{k} when 
\texorpdfstring{$\lambda \in [\lambda_c^{(k)},\lambda_c^{(k-1)}]$}
{lambda is in [lambda\_c(k),lambda\_c(k-1)]}}\label{subsec:rigidity-height-k} 
As a consequence of \cref{prop:non-elem-control-entire-window}, we can deduce rigidity of the SOS distribution with height $k$ boundary conditions just like \cref{lem:rigidity-in-I-windows} throughout the entire regime $\lambda \in [\lambda_c^{(k)},\lambda_c^{(k-1)}]$: for all $\eta$ and all $x\in V$
\begin{align}\label{eq:height-deviation-k-full-window}
     \mu_{\eta,k,V}(|\varphi_x - h| \ge r)\le e^{ - 4(\beta - 6)r }\,, \quad \text{for all $\lambda\in [\lambda_c^{(k)},\lambda_c^{(k-1)}]$}\,.
\end{align}
We go beyond that, showing a more geometric rigidity statement. Namely, if $\lambda \in [\lambda_c^{(k)},\lambda_c^{(k-1)}]$, we show that the set of height $k$ sites qualitatively behave like a highly supercritical percolation. 

For this, we will need a few geometric definitions. 
\begin{definition}\label{def:height-loop}
    We say a set of sites $L\subset \mathbb Z^2$ forms a \emph{loop} if the vertices of $L$ can be ordered $v_0,...,v_{|L|}$ such that $v_0=v_{|L|}$, for every other $i\ne j$, $v_i\ne v_j$, and $v_i \sim v_{i+1}$ for all $i$. They form a loop surrounding a set $A$ if $A$ is in a finite connected component of $\mathbb R^2 \setminus L$. 
\end{definition}

Given a configuration $\varphi$ on a domain $U$ we say $\cL\subset U$ is a height-$k$ loop if it is a loop and $\varphi_x =k$ for all $x\in \cL$. For a fixed set $A$, the \emph{outermost} height-$k$ loop in $U$ surrounding $A$ is measurable with respect to the sites exterior to it. In particular, this outermost loop can be revealed by starting from $\partial_i U$ and revealing its $*$-connected components (connected components with respect to $*$-adjacency, meaning adjacent or diagonally adjacent to) of non-height-$k$ sites. 

\begin{lemma}\label{lem:rigidity-height-k-contour}
    Consider a domain $V$ containing $\Lambda_{m+r}$, with boundary conditions $k$ and suppose $\lambda \in [\lambda_c^{(k)},\lambda_c^{(k-1)}]$. For every $\delta>0$, the outermost loop $\cL$ of height-$k$ sites in $\Lambda_{m+r}$ surrounding the origin satisfies  
	\begin{align*}
		|\cL|\le 4(1+\delta) (m+r)\,, \qquad \text{and} \qquad \Lambda_m \subset \Int(\cL)\,,
	\end{align*}
	except with $\mu_{\eta,k,V}$-probability $m e^{ - (\beta - C)r}+ e^{ - \delta (\beta - C)(m+r)}$.
\end{lemma}

\begin{proof}
	Consider the outermost contours under $\mu_{\eta,k,V}$ whose interiors intersect $\partial_i \Lambda_{m+r}$, as well as any mutually external contours which are incident to those, and so forth. This leaves a collection of contours, call it $G$ the union of whose interiors confines the $*$-connected components of non-height-$k$ sites of $\partial_i U$. 
 
    The probability that $|\cL|>4(1+\delta)(m+r)$ is bounded by the probability that $\sum_{\gamma\in G} |\gamma| >\delta (m+r)$. This is governed by a Peierls map, since all contours in $G$ are mutually external, each contribute a weight of $W_{k}^\rn(\gamma)$ per \cref{lem:part-function-renormalized-weights}, which is at most $e^{ - (\beta -5)|\gamma|}$ using \cref{prop:non-elem-control-entire-window}. The number of choices for $G$ having $\sum_{\gamma \in G} |\gamma|= l$ is at most $2^{\partial_i \Lambda_{m+r}} C^l$, while the probability of any fixed $G$ is at most $\prod_{\gamma\in \Gamma}W_{k}^\rn(\gamma)\le e^{ - (\beta -5)l}$, so the sum over $l\ge \delta(m+r)$ comes out to at most $e^{ - \delta (\beta -C)(m+r)}$ for some universal $C$. 

    The probability that $\Lambda_m \not \subset \Int(\cL)$ is bounded by the probability that some incident sequence of contours of $G$ intersect both $\partial_i \Lambda_{m+r}$ and $\partial_o \Lambda_m$, which is seen to be at most $me^{-(\beta - C) r}$ by a union bound over the $4(m+r)$ sites where the sequence of incident contours could start, and then a Peierls map like the above to bound the probability of such a sequence of incident contours having size at least $r$. 
\end{proof}

\section{Spatial mixing properties in finite volumes}\label{sec:spatial-mixing}
In the previous section, we derived a sequence of critical external field values $(\lambda_c^{(k)})_k$ dictating transitions in which truncated free energy is largest, and showed that when $\lambda \in [\lambda_c^{(k)},\lambda_c^{(k-1)}]$, we have rigidity  for the height $k$ boundary SOS model. 
Our aim is now to deduce spatial mixing properties when the boundary conditions start from a different height. Namely, we show that if we start with some arbitrary boundary condition, the distance between the boundary and the nearest loop of height-$k$ sites has an exponential tail. Inside of the outermost such $k$-loop the surface will be rigid, and this will in particular be used to couple SOS measures with different boundary conditions away from their boundary. 

This section is sensitive to the isoperimetry properties of the underlying domain and we henceforth restrict attention to boxes $\Lambda_m$. Also, we now need to allow consideration of non-constant boundary conditions, so we let $\phi\in \mathbb Z_+^{\partial_o V}$ denote general boundary conditions on $V$, unless otherwise specified the floors and ceilings will be at $0$ and $\infty$ respectively everywhere, and we will use the shorthand notation $\mu_{\phi,\Lambda_m}$ for this measure on $\Lambda_m$. 
The main theorem in this section is the following.   

\begin{theorem}\label{lem:wsm-ball}
    Let $\beta>\beta_0$, $k_m < a_m \le \frac{m}{\log m}$ and $\lambda_m \in (\lambda_c^{(k_m)}, \lambda_c^{(k_m-1)})$. Consider the concentric boxes $\Lambda_{m/2}\subset \Lambda_m$. There is an absolute constant $C>0$ such that for all $m\ge C\beta  a_m  /\dl(\lambda)$ (where $\dl(\lambda)$ is defined as in \cref{eq:dl}) and all boundary conditions $\phi,\phi'$ having $\|\phi\|_\infty \vee \|\phi'\|_\infty < a_m$, we have 
    \begin{align*}
      \big\|\mu_{\phi,\Lambda_m}\big(\varphi(\Lambda_{m/2})\in \cdot\big) - \mu_{\phi',\Lambda_m}\big(\varphi(\Lambda_{m/2})\in \cdot\big)\big\|_\tv \le e^{ - \beta m/C}\,.
    \end{align*}
\end{theorem}

 When $\lambda$ is kept fixed independent of the domain size, we desire 
 a stronger spatial mixing estimate, and towards that purpose need to allow for unbounded boundary conditions.

    \begin{corollary}\label{cor:fixed-lambda-wsm}
    If for some $\epsilon>0$, $\lambda>\epsilon$, then the maximum in \cref{lem:wsm-ball} can be replaced by a supremum over \emph{all possible} boundary conditions $\phi,\phi'$, so long as $m\ge C(\beta,\epsilon)$ for $C(\beta,\epsilon)>0$ that is finite for all $\lambda>0$ (it blows up as $\epsilon \downarrow 0$).  
    \end{corollary}

We will also prove a similar weak spatial mixing statement to \cref{lem:wsm-ball} for annuli of side-length $n$ and thickness $m$, with boundary conditions differing on the inner boundary: see \cref{thm:wsm-annulus}. 

In this section, we mix arguments based on the height function representation of the SOS model, and loops of height $k$, $\{\ge k\}$ and $\{\le k\}$, with those based on the contour representation. This enables us to combine more percolation-theoretic arguments based on coarse-graining, together with the cluster expansion-based estimates of the previous sections, and is key to our results holding arbitrarily close to $\lambda = 0$.

\subsection{\texorpdfstring{$k$}{k}-contours in boxes with height \texorpdfstring{$k\pm 1$}{k plus/minus 1}}
We first aim to establish that started from height $k\pm 1$, there will be a $k$-contour surrounding the concentric box of half the side-length in any box of side-length larger than $\dl^{\mp}(\lambda)$, which we recall are defined as 
\begin{align*}
    \dlp(\lambda) = \min_k \{(\lambda_c^{(k)} - \lambda): \lambda_c^{(k)}\ge \lambda\} \qquad \text{and}\qquad \dlm(\lambda) = \min_k \{(\lambda - \lambda_c^{(k)}): \lambda_c^{(k)}\le \lambda\}\,.
\end{align*}
(Recall also that $\dl(\lambda) = \min\{\dlp(\lambda),\dlm(\lambda)\}$.) Notice that if the boundary conditions on $\Lambda_m$ are $k+1$ and there is a down $k$-contour $\gamma$ surrounding $\Lambda_{m'}$, then necessarily $\partial_i \Int(\gamma)$ forms a loop (per \cref{def:height-loop}) of height-$\{\le k\}$ sites. Similarly for $k-1$ boundary conditions, an up $k$-contour, and height-$\{\ge k\}$ loops. 

The following lemma will be essential to our spatial mixing arguments, and will get boosted into exponential tails to get to height $k$ from any height, rather than just $k\pm 1$ using monotonicity and coarse-graining.  

\begin{lemma}\label{lem:height-pm1-wsm}
	For any $\epsilon>0$, there is $\beta_0(\epsilon)$, $\delta_0(\epsilon)$, and absolute constant $C_0$, such that the following holds. Suppose $V$ is a simply-connected subset of $\Lambda_{m}$ having $|\partial_e V|\le 4(1+\delta_0)m$, and $\lambda \in (\lambda_c^{(k)},\lambda_c^{(k-1)})$, with 
	$$ m \ge C_0 \beta /\dlm(\lambda) \,;$$
    Then, for $\eta \in \{-,f\}$, 
    \begin{align*}
        \mu_{\eta,k+1,V}\big(\exists \gamma\in \Gamma: h(\gamma) = k\,;\, \Int(\gamma) \supset \Lambda_{(1-\epsilon)m}\big) \ge 1-e^{ - \beta m/C}\,.
    \end{align*}
	The same  holds under $\mu_{\eta,k-1,V}$ for $\eta \in \{+,f\}$ if $m \ge C_0 \beta/ \dlp(\lambda)$ instead.  
\end{lemma} 

\begin{proof}
We prove the lemma with $k+1$-boundary conditions, the $k-1$ case being symmetrical. For readability, we drop the boundary signing subscript $\eta$ from the notation. 

We consider the geometric properties of the family of large (at least $(\log m)^2$ length) $k$-contours, to show that with high probability one of them must confine the box $\Lambda_{(1-\epsilon)m}$. Throughout this proof, we call a contour $\gamma$ \emph{macroscopic} if it is outermost and has $|\gamma|\ge (\log m)^2$. For a (standard) contour collection $\Gamma$, let $\Gamma_k^\mac$ be the collection of outermost macroscopic $k$-contours, and let $\Gamma_k^\mic$ be the set of outermost $k$-contours that are not macroscopic.  We use the following basic isoperimetric bound from~\cite[Lemma 2.6]{CLMST16}. 

\begin{fact}\label{fact:isoperimetric-bound}
    For every $\epsilon>0$, there exists $\delta(\epsilon)>0$ such that if $(\gamma_i)_i$ are any collection of mutually external geometric contours in $\Lambda_m$ satisfying 
    \begin{align}\label{eq:isoperimetric-bound}
        \sum_{i} |\gamma_i| \le 4(1+\delta)m\,, \qquad \text{and}\qquad \sum_{i} |\Int(\gamma_i)|\ge (1-\delta)m^2\,.
    \end{align}
    Then, if $\gamma_1$ is the one with the largest interior, $\Int(\gamma_1)$ contains $\Lambda_{(1-\epsilon)m}$. 
\end{fact}

Given \cref{fact:isoperimetric-bound}, it suffices to show that in a typical sample from $\mu_{\eta,k+1,V}$, the collection $\Gamma_k^\mac$ satisfies the two conditions of \cref{eq:isoperimetric-bound} for $\delta = \delta(\epsilon)$. Towards this, we will prove \cref{lem:height-pm1-wsm} with $\delta_0 = \frac{1}{2}\delta(\epsilon)$ and cover the complement of the event in \cref{eq:isoperimetric-bound} by the following bad events: 
\begin{align}
    \mathcal B_1 & := \Big\{\sum_{\gamma \in \Gamma_k^\mac} |\gamma| >4(1+\delta) m \Big\}\,, \label{eq:B1} \\
    \mathcal B_2 &  := \bigcup_{1\le i\le 2\log_2 \log m} \cB_{2,i}\,, \quad \text{where} \quad \cB_{2,i} = \Big\{ |\{\gamma\in \Gamma_k^\mic: |\gamma|\in [2^{i},2^{i+1}]\}| \ge \tfrac{1}{\sqrt{\beta}} e^{ - (1.1)^{i}} m^2\Big\}\,,  \label{eq:B2} \\
    \mathcal B_3 & := \cB_1^c \cap \cB_2^c \cap \Big\{\sum_{\gamma \in \Gamma_{k}^\mac} |\Int(\gamma)| \le (1-\delta) m^2\Big\}\,. \label{eq:B3}
\end{align}
Observe that on the intersection of $(\cB_{2,i}^c)_{1\le i\le 2\log_2 \log m}$, we have 
\begin{align*}
    \sum_{\gamma \in \Gamma_k^\mic} |\Int(\gamma)| \le \sum_{i} 2^{2(i+1)} \frac{1}{\sqrt{\beta}}e^{ - (1.1)^{i}} m^2 \le \frac{C}{\sqrt{\beta}} m^2\,,
\end{align*}
which for $\beta$ large, is at most $\delta m^2/2$ area confined by contours in $\Gamma_k^\mic$. 
We show that each of $\cB_1,\cB_2,\cB_3$ have small probability in \cref{eq:B1,eq:B2,eq:B3}, which when summed yield \cref{lem:height-pm1-wsm}. 

\medskip
\noindent \emph{Bound on $\cB_1$: too much length in macroscopic loops.}
We first show that for $\beta$ large enough, $\delta_0 \le \delta/2$, 
\begin{align}\label{eq:isoperimetry-B1-bound}
    \mu_{k+1,V}(\cB_1) \le e^{ - (\beta-C) \delta m}\,.
\end{align}
Consider a fixed collection of mutually external macroscopic $k$-contours $\Gamma_k^\mac$ in $\cB_1$. The weight of configurations with this specific realization of $\Gamma_k^\mac$ can be written as 
\begin{align*}
    Z_{k+1,V}(\Gamma_k^\mac) := \overline{Z}_{k+1,\Ext(\Gamma_{k}^\mac)} \prod_{\gamma\in \Gamma_k^\mac} e^{ - \beta |\gamma|} Z_{-,k,\Int(\gamma)}\,,
\end{align*}
where $\overline{Z}$ indicates here that the sum is restricted to having no outermost macroscopic $k$-contours, and no contours confining any of the $\gamma \in \Gamma_{k}^{\mac}$ (ensuring compatibility with $\Gamma_k^\mac$ indeed being the collection of outermost macroscopic $k$ contours), and the boundary signing is that induced by $-$ on $\partial_i V$ and $+$ on $\partial_0 \Int(\gamma)$ for $\gamma \in \Gamma_k^\mac$.  
By inclusion, the first term is at most $Z_{k+1,\Ext(\Gamma_k^\mac)}$. 
At the same time, the denominator in $\mu_{k+1,V}(\Gamma_k^\mac) = Z_{k+1,V}(\Gamma_k^\mac)/Z_{k+1,V}$ evidently satisfies
\begin{align}\label{eq:Z_{k+1}-lb}
    Z_{k+1,V} \ge e^{ - \beta|\partial_e V|} Z_{-,k,V}  \ge e^{ - \beta |\partial_e V|} Z_{k,\Ext(\Gamma_{k}^\mac)} \prod_{\gamma \in \Gamma_k^\mac} Z_{-,k,\Int(\gamma)}\,.
\end{align}
Dividing through, and using $|\partial_e V| \le 4(1+\delta_0) m$ we get 
\begin{align*}
    \mu_{k+1,V}(\Gamma_k^\mac) \le e^{\beta(4(1+\delta_0)m -\sum_{\gamma \in \Gamma_k^\mac}|\gamma|)} \frac{Z_{k+1,\Ext(\Gamma_k^\mac)}}{Z_{k,\Ext(\Gamma_k^\mac)}}\,.
\end{align*}
By \cref{cor:part-function-comparison-full-window} applied to $U = \Ext(\Gamma_k^\mac)$, whence $|\partial_e \Ext(\Gamma_k^\mac)| \le 4(1+\delta_0)m + \sum_{\gamma\in \Gamma_k^\mac} |\gamma|$, 
\begin{align*}
    \mu_{k+1,V}(\Gamma_k^\mac) \le e^{ \beta(4(1+\delta_0)m - \sum_{\gamma} |\gamma|) + e^{ - \beta/3}(4(1+\delta_0)m + \sum_{\gamma\in \Gamma_k^\mac}|\gamma|)}\,.
\end{align*}
At this point, we can bound $\mu_{k+1,V}(\cB_1)$ by summing over the possible choices of $\Gamma_k^\mac \in \cB_1$. 
Let $K$ denote the total number of contours in $\Gamma_k^\mac$, let $M$ denote $\sum_{\gamma\in \Gamma_k^\mac} |\gamma|$, and note that $K \le M/(\log m)^2$ since these are all macroscopic contours. Then bounding $\binom{m^2}{K} \le e^{ 2 K\log m}$, we have  
\begin{align*}
    \mu_{k+1,V}(\cB_1) 
    & \le \sum_{M\ge 4(1+\delta)m} \sum_{K\le M/(\log m)^2} e^{ 2 K \log m} 4^{M} e^{ - (\beta -1)(M-4(1+\delta_0)m)}\,.
\end{align*}
Since $K\le M/(\log m)^2$, we have $e^{2K \log m}\le e^{M}$, (for $m \ge 10$ say) and this is at most 
\begin{align*}
    \sum_{M\ge 4(1+\delta) m} e^{M} e^{ -(\beta -1)(M-4(1+\delta_0)m)} \le e^{4(1+\delta_0)m} \sum_{L \ge 4\delta_0 m} e^{ - (\beta -C) L}\,,
\end{align*}
where we used that $\delta  \ge 2\delta_0$. 
This implies the desired \cref{eq:isoperimetry-B1-bound}. 

\medskip
\noindent \emph{Bound on $\cB_2$: too much area in microscopic loops.}
Our next aim is to show that for every $i\le 2\log_2 \log m$,  
    \begin{align}\label{eq:isoperimetry-B2-bound}
        \mu_{k+ 1,V} (\cB_{2,i}) \le e^{ -(\beta - C)m^{3/2}}\,.
    \end{align}
    which when summed over $i\le 2\log_2 \log m$ bounds the probability of $\cB_{2}$ by $e^{ - (\beta - C)m^{4/3}}$, say.

For ease of notation, let $\Gamma_{k,i}^\mic = \{\gamma \in \Gamma_k^\mic: |\gamma| \in [2^i, 2^{i+1}]\}$
so that $\cB_{2,i}= \{|\Gamma_{k,i}^{\mic}| \ge \beta^{-1/2}e^{ - (1.1)^i}m^2\}$. 
Consider a fixed collection of outermost non-elementary $k$-contours $\Gamma_k= (\Gamma_{k}^\mic,\Gamma_k^\mac)$ belonging to $\mathcal B_{2,i}$. 
The total weight of configurations with a specific realization of $\Gamma_{k,i}^\mic$, denoted $Z_{k+1,V}(\Gamma_{k,i})$ is 
\begin{align*}
    Z_{k+1,V}(\Gamma_{k,i}^\mic) = \overline{Z}_{k+1,\Ext(\Gamma_{k,i}^\mic)}\prod_{\gamma \in \Gamma_{k,i}^\mic}e^{ - \beta |\gamma|}Z_{-,k,\Int(\gamma)}\,.
\end{align*}
where $\overline{Z}$ indicates that it has no non-elementary outermost down contours of size between $[2^i,2^{i+1}]$ and also no non-elementary contours nesting any $\gamma \in \Gamma_{k,i}^\mic$. By inclusion, the $\overline{Z}$ term is at most $Z_{k+1,\Ext(\Gamma_{k,i}^\mic)}$.   
Dividing out by $Z_{k+1,V}$ for which we have the lower bound analogous to \cref{eq:Z_{k+1}-lb},
\begin{align*}
    \frac{Z_{k+1,V}(\Gamma_{k,i}^\mic)}{Z_{k+1,V}} 
    \le e^{ 4\beta (1+\delta_0)m - \beta \sum_{\gamma \in \Gamma_{k,i}^\mic}|\gamma|} \frac{Z_{k+1,\Ext(\Gamma_{k,i}^\mic)}}{Z_{k,\Ext(\Gamma_{k,i}^\mic)}}\,.
\end{align*}
By \cref{cor:part-function-comparison-full-window}, this is at most 
\begin{align*}
    e^{\beta (4 (1+\delta_0)m - \sum_{\gamma\in \Gamma_{k,i}}|\gamma|) + e^{ - \beta/3} |\partial_e \Ext(\Gamma_{k,i})|}\,.
\end{align*}
Now observe that on the event $\cB_{2,i}$,  
\begin{align*}
    \sum_{\gamma \in \Gamma_{k,i}^\mic}|\gamma| \ge 2^i \beta^{-1/2} e^{ - (1.1)^{i}} m^2 = \beta^{-1/2} e^{i \log 2 - (1.1)^{i}} m^2\,.
\end{align*}
For all $i \le 2 \log_2 \log m$, this is at least $m^{3/2}$, say, so long as $m\ge C \beta^{1/4}$. 
Now to sum over the possible choices of $\Gamma_{k,i}^{\mic}$, let  $\sum_{\gamma \in \Gamma_{k,i}^\mic}|\gamma| = M$, let 
$K$ count the number of contours in $\Gamma_{k,i}^\mic$, and note the bound $|
\partial_e \Ext(\Gamma_{k,i})| \le 4(1+\delta_0)m + M$.  
We then obtain    
\begin{align*}
    \mu_{k+1,V}(\cB_{2,i})  \le  \sum_{M \ge \beta^{-1/2} e^{i\log 2 - (1.1)^i}m^2} \sum_{K \le 2^{-i}M}\binom{m^2}{K} C^{M} e^{ - (\beta -1) (M- 4(1+\delta_0)m)}\,.
\end{align*}
Letting $2^{-i} M = \rho m^2$, and using the fact that $M \ge m^{3/2}$, this is at most 
\begin{align*}
     \sum_{ \beta^{-1/2} e^{- (1.1)^i}\le \rho \le 1} (\rho m^2) e^{ \rho m^2 (1+\log\frac{2}{\rho})} e^{- (\beta - C)M}\,.
\end{align*}
Since $\rho \ge \beta^{-1/2} e^{-(1.1)^{i}}$ we can upper bound $\log(2/\rho)$ and use  $2^i \beta \rho m^2  \ge  4 \rho m^2 (1.1)^i\log \beta$ to see that the second exponential dominates the first, whence for all $m\ge C \beta^{1/4}$, we get 
\begin{align*}
    \mu_{k+1,V} (\cB_{2,i}) \le e^{ - (\beta -C) e^{-(1.1)^i}m^2}\,.
\end{align*}
Using that $e^{ - (1.1)^i}m^2 \ge m^{3/2}$ for all $i\le 2\log_2 \log m$ then yields the claimed \cref{eq:isoperimetry-B2-bound}.

\medskip
\noindent \emph{Bound on $\cB_3$: too little area in macroscopic loops.}
    For the last bound, let 
\begin{align}\label{eq:Delta-tr-def}
    \Delta_\tr = f^\tr_k - f^\tr_{k+1}\,.
\end{align}
which is positive and at least $(1-\epsilon_\beta)\dlm(\lambda)$ per \cref{cor:delta-f-delta-l-relation} for $\lambda \in (\lambda_c^{(k)},\lambda_c^{(k-1)}]$. 
The goal of this last part is to establish the following bound: 
    \begin{align}\label{eq:isoperimetry-B3-bound}
        \mu_{k+ 1,V} (\cB_3) \le e^{ - \beta m/C}\,.
    \end{align}
On $\cB_2^c$, the total boundary length confined by outer contours of size at least $\Delta_\tr^{-1}$ is at most 
\begin{align*}
    \sum_{i\ge -\log_2\Delta_\tr} 2^{i+1} \beta^{-1/2} e^{ - (1.1)^i } m^2\,. 
\end{align*}
We can absorb the $2^{i+1}$ prefactor by changing this to $\beta^{-1/2} e^{-(1.05)^i}$, and summing that out, we get that $\sum_{i\ge -\log_2 \Delta_\tr} \sum_{\gamma \in \Gamma_{k,i}^\mic} |\gamma| \le \beta^{-1/2} m^2 e^{ - \Delta_\tr^{-\alpha}}$, 
for some universal constant $\alpha>0$. 

Now consider any realization of $G_k = ((\Gamma_{k,i}^\mic)_{i\ge -\log_2 \Delta_\tr},\Gamma_k^\mac)$ belonging to $\cB_3$.  The partition function associated to these contour families is 
\begin{align*}
    Z_{k+1,V}(G_k) = \overline{Z}_{k+1,\Ext(G_k)} \prod_{\gamma \in G_k} e^{ - \beta |\gamma|} Z_{-,k, \Int(\gamma)}\,,
\end{align*}
where $\overline{Z}$ here indicates that there are no outermost $k$-contours of size larger than $\Delta_\tr^{-1}$, and none of the ones in $G_k$ are surrounded by another contour. As such, for every contour $\gamma$ appearing in the sum corresponding to $\overline{Z}$, it has $\diam(\gamma)\le |\gamma|/2 \le \Delta_\tr^{-1}$. By \cref{cor:non-elem-control-smaller-than-delta-tr}, this implies that for all contours appearing in $\overline{Z}$, they have $ W^\rn_{k+1}(\gamma) = {W}^\tr_{k+1}(\gamma)$. By inclusion, then, 
\begin{align*}
    Z_{k+1,V}(G_k) \le Z^\tr_{k+1,\Ext(G_k)} \prod_{\gamma \in G_k} e^{ - \beta |\gamma|} Z_{-,k,\Int(\gamma)}\,.
\end{align*}
Dividing by $Z_{k+1,V}$, on which we have the lower bound analogous to \cref{eq:Z_{k+1}-lb}, we get 
\begin{align*}
    \mu_{k+1,V}(G_k) \le e^{ 4\beta(1+\delta_0) m - \beta \sum_{\gamma \in G_k} |\gamma|} \frac{Z_{k+1,\Ext(G_k)}^\tr}{Z_{k,\Ext(G_k)}}\,.
\end{align*}
By \cref{prop:non-elem-control-entire-window}, $Z_{k,\Ext(G_k)} = Z_{k,\Ext(G_k)}^\tr$ since $\lambda \in [\lambda_c^{(k)},\lambda_c^{(k-1)}]$, so by \cref{lem:truncated-partition-function-cluster-expansion}, 
\begin{align*}
    \mu_{k+1,V}(G_k) \le e^{ \beta (4(1+\delta_0)m - \sum_{\gamma \in G_k} |\gamma|)} e^{- \Delta_\tr |\Ext(G_k)| + 2e^{ - \beta/2} |\partial_e \Ext(G_k)|}\,.
\end{align*}
On the event $\cB_3$, we must have $|\Ext(G_k)|\ge \delta m^2/2$ as under $\cB_2^c$ at most $\delta m^2/2$ area is confined by $\bigcup_{i} \Int(\Gamma_{k,i}^\mic)$ and on $\cB_3$ at most $(1-\delta)m^2$ is confined in $\Gamma_k^\mac$. Thus,  
\begin{align*}
    \mu_{k+1,V}(G_k) \le e^{ 4(\beta+1)(1+\delta_0)m - (\beta -1)\sum_{\gamma \in G_k}|\gamma|} e^{ - \Delta_\tr \delta m^2/2} \le e^{(4\beta + C)m - \Delta_\tr \delta m^2/2}\,. 
\end{align*}

Now consider the entropy over the choices of $G_k$ that belong to $\cB_3$. Since $\cB_3 \subset \cB_1^c$, for the macroscopic contours, the total number of choices of $\Gamma_k^\mac$ in $\cB_1^c$ is evidently at most 
\begin{align*}
    \binom{m^2}{4(1+\delta_0)m/(\log m)^2} 4^{4(1+\delta_0)m} \le \exp(10 m)\,.
\end{align*}
For the microscopic contours, since $\cB_3 \subset \cB_2^c$, for each $i$, there are at most $\beta^{-1/2} e^{ - (1.1)^i} m^2$ many contours in $\Gamma_{k,i}^\mic$, each of size at most $2^{i+1}$; thus, the number of choices is at most 
\begin{align*}
    \prod_{i=-\log\Delta_\tr}^{2\log_2 \log n} e^{\beta^{-1/2} e^{ - (1.1)^i} m^2 (1+ (1.1)^i + \log \beta)} 4^{ \beta^{-1/2} 2^{i+1} e^{ - (1.1)^i} m^2} = \exp\Big(\beta^{-1/3} m^2 \sum_{i=-\log \Delta_\tr}^{2\log_2\log n}  2^i e^{ - (1.1)^i}\Big)
\end{align*}
Since the sum is dominated by the smallest $i= -\log \Delta_\tr$, the count above is at most $\exp( \beta^{-1/3} e^{ - \Delta_\tr^{-\alpha}} m^2)$
for some universal constant $\alpha>0$. 
Therefore, summing up over the choices of $G_k \in \cB_3$, we get 
\begin{align*}
    \mu_{k+1,V}(\cB_3) \le \exp\big( - ( \tfrac{1}{2}\Delta_\tr \delta - \beta^{-1/3} e^{ - \Delta_\tr^{-\alpha}}) m^2  + (4\beta + C) m\big)\,.
\end{align*}
We have the naive upper bounds of $\Delta_\tr \le \max\{\lambda, (1+\epsilon_\beta)e^{ - 4\beta}\}$ on $\Delta_\tr$, we are really only interested in the behavior when $\Delta_\tr \le 1$, say. Since $\alpha$ is universal (independent of $\beta,\lambda$), it is easy to check that for large $\beta$ (depending only on $\alpha$) we have $\beta^{-1/3} e^{ - \Delta_\tr^{-\alpha}} \le \Delta_\tr/4$. As such, we deduce that 
\begin{align*}
    \mu_{k+1,V}(\cB_3)\le \exp( - \tfrac{1}{4} \Delta_\tr \delta m^2 + (4\beta + C) m)\,.
\end{align*}
This gives \cref{eq:isoperimetry-B3-bound} so long as $m\ge 20 \beta \delta^{-1}/\dlm(\lambda)$. 
\end{proof}

\subsection{The outermost \texorpdfstring{$k$}{k}-loop in general domains with \texorpdfstring{$k\pm 1$}{k plus/minus 1} boundary conditions}
We next use \cref{lem:height-pm1-wsm} to show that in general domains (not necessarily boxes) with $k\pm 1$ boundary conditions, the distance of the outermost loop of height-$k$ sites to the boundary will have an exponential tail. 
For any domain $V$, define the shrinkage of $V$ by $r$ as $$S_r(V)= \{v\in V: d(v,\partial_i V)\ge r\}\,.$$
For $V: S_r(V)\supset \Lambda_{m/2}$, let $\cL_{\le k}$ be the outermost loop of $\{\le k\}$ sites surrounding the origin. 

\begin{lemma}\label{lem:kpm1-to-k-geenral-domain}
    Let $V$ be any simply connected set having $\partial \Lambda_{m/2} \subset V \subset \partial \Lambda_m$. For every $\lambda \in (\lambda_c^{(k)},\lambda_c^{(k-1)})$, if $r\ge C_0 \beta /\dlm(\lambda)$ and $m\ge 2r$, then for $\eta \in \{-,f\}$ 
    $$\mu_{\eta,k+1,V}(\cL_{\le k} \subset V\setminus S_r(V)) \ge  1- |\partial_i S_r(V)| e^{ - \beta r/C}\,.$$
    The same holds under $\mu_{\eta,k-1,V}$ for $\eta \in \{+,f\}$ if $r\ge C_0 \beta/ \dlp(\lambda)$. 
\end{lemma}

\begin{proof}
We prove one direction, the other direction being symmetrical. 
For every $x$ at distance $r$ from $\partial V$, let $B_x =B_{r}(x)$ be the ball of radius $r$ about $x$, and let $A_x= B_r(x)\setminus B_{r/2}(x)$ be the annulus between $B_x$ and its concentric box of half the radius. By construction, $B_x,A_x$ are interior to $V$. Let $E_x$ be the event that there is a loop of $\{\le k\}$-sites in the annulus $A_x$.  If $E_x$ occurs for all $x\in \partial_i S_{r}(V)$, then $\cL_{\le k} \subset V\setminus S_r(V)$.  

Therefore, the complement of that is a subset of the event $\bigcup_{x\in \partial_i S_r(V)} E_x^c$, and it suffices to upper bound the probability of this union. We use a union bound for this and consider any fixed $x$. For fixed $x\in \partial_i S_{r}(V)$, of which there are at most $m^2$, consider the probability of $E_x^c$. Let $A_x'$ be the annulus $B_{r}(x) \setminus B_{3r/4}(x)$ and first consider the event $E_{x,0}$ that there is an outermost loop of $\{\le k+1\}$-height sites $\cC_x$ in $A_x'$ having total length $|\cC_x|\le 4(1+\delta_0) r$. We use the bound 
\begin{align}\label{eq:E_x^c-split}
    \mu_{k+1,V}(E_x^c) \le \mu_{k+1,V}(E_{x,0}^c) + \mu_{k+1,V}(E_x \mid E_{x,0})\,.
\end{align}
For the first probability in \cref{eq:E_x^c-split}, note that the event $E_{x,0}^c$ is an increasing event, so its probability is only larger by decreasing $\lambda$ to $\lambda_c^{(k)}$. Upon doing so, the probability of $E_{x,0}^c$ is upper bounded by \cref{lem:rigidity-height-k-contour} whereby it is at most $e^{ - \delta_0 \beta r/4}$ for $r$ larger than an absolute constant.

For the second probability in \cref{eq:E_x^c-split}, we can condition on the outermost loop of $\{\le k+1\}$ sites, $\cC_x$, and note that it being measurable with respect to its exterior, we get 
\begin{align*}
    \mu_{k+1,V}(E_x^c \mid E_{x,0}) \le \max_{\cC_x} \max_{\phi \le k+1} \mu_{\phi,\Int(\cC_x)}(E_x^c)\,,
\end{align*}
where the maximum over $\cC_x$ runs over those confining $B_{3r/4}$ in their interior and having length at most $4(1+\delta_0)r$. Since $E_x^c$ is an increasing event, we can take the maximal boundary conditions $k+1$ on $\Int(\cC_x)$, at which point \cref{lem:height-pm1-wsm} with $\epsilon = 1/2$ (so that $\delta_0$ above is $\delta_0(\epsilon=1/2)$) provides an upper bound of $e^{ - \beta r/C}$ so long as $r \ge C_0 \beta /\dlm(\lambda)$ and $\beta$ large. Combining the above two estimates and taking a union bound over the $|\partial_i S_r(V)|$ many possible choices of $x$ gives the claimed bound with a different absolute $C$. 
\end{proof}

\subsection{The case of general boundary heights}
Our aim is now to show that if the boundary conditions are at any height between $0$ and $m/\log m$, the interface will, with high probability, contain a loop of height-$k$ sites surrounding the origin when $\lambda \in (\lambda_c^{(k)},\lambda_c^{(k-1)})$. The argument essentially goes by repeatedly using monotonicity to apply \cref{lem:kpm1-to-k-geenral-domain} to each level from the boundary condition height to $k$ one at a time. 

\begin{lemma}\label{lem:reaching-height-k}
Suppose $\lambda\in (\lambda_c^{(k)},\lambda_c^{(k-1)})$, $k+1\le j\le \frac{m}{\log m}$, and 
$ m \ge 2 C_0 j \beta / \dlm(\lambda)$. 
 For $\eta \in \{-,f\}$,  
\begin{align*}
    \mu_{\eta, j,\Lambda_m}\big(\cL_{\le k} \subset \Lambda_{m}\setminus \Lambda_{3m/4} \big) \ge  1- e^{ - \beta m/C}\,.
\end{align*}
The analogue holds for $\eta \in \{+,f\}$, $0\le j\le k-1$, and $\cL_{\ge k}$ replacing $\cL_{\le k}$ if $m\ge 2C_0 k \beta /\dlp(\lambda)$.  
\end{lemma}

\begin{proof}
We show the bound supposing that $j\ge k+1$, the other direction being analogous. Let $\cC_i$ denote $\Int(\cL_{\le i})$ for all $i$. Let $r_i$ be the minimal $r$ such that $\cL_{\le i-1}$ is a subset of $\cC_{i}\setminus S_r(\cC_i)$, in other words the maximal distance from $\cL_{\le i}$ that $\cL_{\le i- 1}$ gets. In order for $\cL_{\le k}$ to not be a subset of $\Lambda_{m}\setminus \Lambda_{3m/4}$, it must be the case that $\sum_{i=k}^{j-1} r_i$ is at least $m/4$. We fix any possible sequence $(r_i)_i$ with $r_i \ge 0$ for all $i$ and such that $\sum_{i=k}^{j-1} r_i = m/4$. Since $j\le m$, there are at most $2^{2m}$ many such choices. Since $m\ge 2 C_0 j \beta / \dlm(\lambda)$, there must be some subsequence of these $(r_i)$, call them $(r_{i_l})_l$, such that they are each larger than $C_0 \beta /\dlm(\lambda)$ and whose sum exceeds $m/8$. 

We now upper bound the probability that for all $l$, we have $\cL_{\le i_l}$ intersects $S_{r_{i_l}}(\cC_{i_l +1})$. We do so iteratively from the boundary inwards. Namely, for fixed such $(r_{i_l})_l$, the probability of interest is at most
\begin{align*}
    \prod_{l} \mu_{\eta,j,\Lambda_m}(\cL_{\le i_l} \cap S_{r_{i_l}}(\cC_{i_l + 1})\ne \emptyset \mid \varphi(\Lambda_m \setminus \cC_{i_l+1}))\,. 
\end{align*}
(Here we used the fact that $\cC_{i_l+1}$ is measurable with respect to the height values on its exterior.) The conditional measure is then exactly $\mu_{-,i_l+1,\cC_{i_l+1}}$. By monotonicity, since the event in question is an increasing event, the probability is only larger if we lift the floor up to $i_l - k$ whence the measure will be an additive translation of $\mu_{-k+1,\cC_{i_l+1}}$ with $\lambda \in (\lambda_c^{(k)},\lambda_c^{(k-1)})$. At that point, the probability we are interested in is exactly that bounded by \cref{lem:kpm1-to-k-geenral-domain}, whence it is at most $|\partial_i S_r(\cC_{i_l+1})| e^{ - r_{i_l}/C}$.  Performing the sum over all possible choices of $r_{i_l}$, and plugging this in, we get 
\begin{align*}
    \mu_{j,\Lambda_m}(\cC_k \cap \Lambda_{3m/4} \ne \emptyset) \le 4^m (m^2)^{j} e^{ - \beta m/C}\,.
\end{align*}
So long as $j\le m/(\log m)$, this will be at most $e^{ -\beta m/C}$ for a different $C$. 
\end{proof}

\begin{remark}
    At this point, one can see \cref{mainthm:equilibrium} by applying (the second part of) \cref{lem:reaching-height-k} with $j=0$, revealing the outermost loop $\cL_{\ge k}$, then using the rigidity at height $k$ interior to $\cL_{\ge k}$ per the estimates of \cref{subsec:rigidity-height-k} to see that interior to $\cL_{\ge k}$, most sites are exactly at height $k$ when $\lambda \in (\lambda_c^{(k)},\lambda_c^{(k-1)})$. 
\end{remark}

In the case where $\lambda>0$ is kept independent of the box size, the external field is strong enough to allow us to consider all possible boundary conditions (unbounded) and still obtain spatial mixing. For this purpose, the following lemma shows that the surface drops to some finite height quickly even started arbitrarily high.

\begin{lemma}\label{lem:getting-to-finite-heights}
    For every $\lambda>0$, if $m\ge C \lambda^{-1}$, the outermost loop of height $\{\le C\lambda^{-2}\}$ in $\Lambda_m$ confines $\Lambda_{(1-\epsilon)m}$ in its interior except with probability $e^{ - c m}$. 
\end{lemma}

\begin{proof}
It is fairly straightforward to check that as long as a box has side-length $L$ at least $L_0 :=\lceil 8/\lambda\rceil$ the values of $\varphi$ will have finite moments uniformly over all (unbounded) boundary conditions: see~\cite[Proposition 3.2, item (i)]{CeMa1} for a short self-contained argument. In particular, in a domain $V$, for any $x$ at distance at least $L_0$ away from $\partial V$, for every $s$ there exists a constant $C(\beta,s)>0$ such that for all $\lambda>0$,
\begin{align*}
    \sup_{\phi}\mu_{\phi, V}[\varphi_x^s] \le \lambda^{-s} C(\beta,s)\,.
\end{align*}
Now for simplicity, consider the domain $\Lambda_m$, and tile it by overlapping boxes of side-length $2L_0 \times 2L_0$ that overlap on half of their area (i.e., centered at the vertices of $\Lambda_m \cap L_0 \mathbb Z^2$.

Consider any such box $B_x$. Uniformly over the boundary conditions on that box, by Markov's inequality, and a union bound, the probability that there is no $\{\le L_0\lambda^{-1}\}$ loop around the concentric $L_0 \times L_0$ box is at most $C(\beta,s)L_0^{-s+1}$. Call such a box bad, and call a box good if there is such a loop. The set of bad blocks form a $1$-dependent percolation in the graph induced by the boxes, with two boxes called adjacent if they intersect. By classical reasoning (see e.g.,~\cite{LSS}), this will be stochastically dominated by an independent percolation with a parameter that can be made arbitrarily small by taking $L_0$ large, from which we deduce that the probability of a path of bad blocks from $\partial \Lambda_m$ to $\partial \Lambda_{(1-\epsilon)m}$ is exponentially unlikely in $\epsilon m$. 
\end{proof}

\subsection{From \texorpdfstring{$\{\ge k\}$}{\{>= k\}} and \texorpdfstring{$\{\le k\}$}{\{<= k\}} loops to spatial mixing}
We now wish to use \cref{lem:reaching-height-k} to deduce weak spatial mixing amongst the set of boundary conditions ranging from $0$ to some upper bound $a_m$ (and in the case where $\lambda$ is bounded away from zero, any arbitrary boundary conditions). The following lemma boosts the probability of two configurations individually having $k$-loops to the joint probability of having a common $k$-loop (a loop of sites whose height is $k$ in both configurations).

\begin{lemma}\label{lem:coupling-shared-k-loop}
    Suppose $\lambda \in (\lambda_c^{(k)}, \lambda_c^{(k-1)})$. For any two boundary conditions $\phi,\phi'$ and any coupling of $\varphi \sim \mu_{\phi,\Lambda_m}$ and $\varphi'\sim \mu_{\phi',\Lambda_m}$, the probability that $(\varphi,\varphi')$ don't have a common $k$-loop surrounding $\Lambda_{(1-2\epsilon)m}$ is at most the sum of the marginal probabilities of not having a $k$-loop surrounding $\Lambda_{(1-\epsilon)m}$, plus $e^{ - \beta \epsilon m /C}$. 
\end{lemma}

\begin{proof}
    Let $\cE_k$ and $\cE_k'$ be the events that there is a loop of  height-$k$ sites surrounding $\Lambda_{(1-\epsilon)m}$ in $\varphi$ and $\varphi'$ respectively. Let $\cS_k$ be the event that there is a common $k$-loop in the annulus $\Lambda_{(1-\epsilon)m}\setminus \Lambda_{(1-2\epsilon)m}$. The desired probability is at most 
    \begin{align*}
        \mu_{\phi,\Lambda_m}(\cE_k^c) + \mu_{\phi',\Lambda_m}((\cE_k')^c) + \mathbb P(\cE_k,\cE_k', \cS_k^c)\,.
    \end{align*}
    It suffices to bound the third probability.    
    The event $\cS_k^c$ implies the existence of a path $\cP$ from $\partial_i \Lambda_{(1-\epsilon)m}$ to $\partial_o \Lambda_{(1-2\epsilon)m}$ such that for all $x\in \cP$ either $x$ is interior to a contour inside the outermost $k$-loop of $\varphi$, or interior to a contour inside the outermost $k$-loop of $\varphi'$. In particular, if we consider the collection of all outermost contours inside the outermost $k$-loop of $\varphi$, and color them $\red$ and do the same for $\varphi'$ and color them $\blue$, there must be a collection of $\red$ and $\blue$ contours whose union is a connected subset $\cP$ of $\Lambda_m^*$ connecting $\partial_i \Lambda_{(1-\epsilon)m}$ to $\partial_o \Lambda_{(1-2\epsilon)m}$.  Let $M\ge \epsilon m$ denote the total number of edges in $\cP$, there being at most $4m 3^M$ many choices for it, and another $3^M$ many choices of for choosing whether each edge in $\cP$ was colored $\red$, $\blue$, or both. There must then be a subset of $\cP$, call it $\tilde \cP$ which has at least $M/2$ many edges, and all of whose edges were colored $\red$, or all of whose edges were colored $\blue$. This gives us the following bound: 
    \begin{align*}
        \mathbb P(\cE_k,\cE_k',\cS_k^c) \le \sum_{M\ge \epsilon m}4 m 3^{2M} \max_{\cP: |\cP| = M} 2^M \max_{\tilde\cP \subset \cP: |\tilde \cP|\ge M/2} \big(\mu_{\phi,\Lambda_m}(\cE_k, E_{\tilde \cP}) + \mu_{\phi',\Lambda_m}(\cE_k',E_{\tilde \cP}')\big)\,.
    \end{align*}
    where we are using $E_{\tilde \cP}, E_{\tilde \cP}'$ to denote the event that the contours in $\tilde \cP$ are all outermost contours interior to $\cL_k$ in $\varphi$ or $\varphi'$ respectively. 
    
    Fix $M$, a colored collection of contours $\cP$, a subset $\tilde \cP$ forming mutually external contours, and w.l.o.g.\ consider the first of the two probabilities on the right-hand side above. We can condition on  the outermost $k$-loop in $\mu_{\phi,\Lambda_m}$, call its interior $\cC_k$, this being measurable with respect to its exterior and inducing boundary conditions $k$ on its interior. The bound on the probability of $E_{\cP}$ then becomes a standard consequence of the rigidity at height $k$ in $\cC_k$. Namely, by the Peierls map that deletes the contours making up $\tilde \cP$ from the renormalized contour collection for $\varphi(\cC_k)$, and using \cref{prop:non-elem-control-entire-window} to bound the weight change, uniformly over $\cC_k$, we have 
    \begin{align*}
        \mu_{\emptyset,k,\cC_k}(E_{\tilde \cP}) \le e^{ - (\beta -5)|\tilde \cP|} \le e^{ - (\beta -5) M/2}\,.
    \end{align*}
    Plugging this bound in to the earlier sum, together with its equivalent for $\varphi'$, implies $\mathbb P(\cE_k,\cE_k',\cS_k^c)$ is at most $\exp( - \beta \epsilon m/C)$ as claimed. 
\end{proof}

\begin{corollary}\label{cor:general-bc-k-loop}
    If $\lambda \in (\lambda_c^{(k)},\lambda_c^{(k-1)})$ and $k \vee \|\phi\|_\infty < a_m \le \frac{m}{\log m}$, as long as $m\ge 2 C_0 a_m \beta /\dl(\lambda)$, 
    \begin{align*}
        \mu_{\phi,\Lambda_m}(\cL_k \subset \Lambda_m \setminus \Lambda_{2m/3}) \ge 1-e^{ - \beta m/C}\,,
    \end{align*}
    If $\lambda>\epsilon$ for some $\epsilon>0$ fixed, then we can drop the upper bound on $\|\phi\|_\infty$, at the expense of also assuming $m$ is at least some constant $C(\epsilon)>0$.   
\end{corollary}

\begin{proof}
    Consider $\varphi,\varphi'$ being drawn from the SOS measures with minimal and maximal boundary conditions $\phi \equiv a_m$ and $\phi'\equiv 0$, and couple them using the grand monotone coupling. For each of $\varphi,\varphi'$, the individual probability of not having a $\{\le k\}$ or $\{\ge k\}$-loop surrounding $\Lambda_{3m/4}$ is bounded by $e^{ - \beta m/C}$ per \cref{lem:reaching-height-k} (applicable to both since $\dl(\lambda)= \dlm(\lambda)\wedge \dlp(\lambda)$). Interior to the outermost such loop, the rigidity at height $k$ as argued in \cref{lem:rigidity-height-k-contour} implies that $\cL_k$ will in fact contain $\Lambda_{2m/3}$ in its interior, except with probability $e^{ - \beta m/C}$. Applying \cref{lem:coupling-shared-k-loop}, they will therefore have a \emph{common} $k$-loop except with probability $e^{ - \beta m/C}$ for some other $C$. Since we are using the grand monotone coupling, if $\varphi,\varphi'$ share a common $k$-loop, then so do the coupled draws from $\mu_{\rho,\Lambda_m}$ for all boundary conditions $\rho: \|\rho\|_\infty \le a_m$, as desired.   
    
    When $\lambda>0$ independently of $m$, we use the grand coupling of all possible boundary conditions, noting that if a constant boundary condition at $\phi^+$ is coupled to the one with boundary conditions at $0$, then all boundary conditions in between also share a $k$-loop. For those two, first use \cref{lem:getting-to-finite-heights} to get a $\{\le C\lambda^{-2}\}$-height loop surrounding $\Lambda_{ 4m / 5}$ uniformly over $\phi^+$, then apply the above reasoning interior to that. 
\end{proof}

With the above, we are able to deduce the claimed spatial mixing bound on boxes. 

\begin{proof}[\textbf{\emph{Proof of \cref{lem:wsm-ball}}}]
For any two boundary conditions $\phi,\phi'$ having $\|\phi\|_\infty \vee \|\phi'\|_\infty < a_m$ inducing configurations $\varphi, \varphi'$, we construct a coupling such that they agree on $\Lambda_{m/2}$ except with exponentially small probability. First, expose (under an arbitrary coupling) from the outside in, their outermost common $k$-loop surrounding $\Lambda_{m/2}$ if it exists. By \cref{cor:general-bc-k-loop} together with \cref{lem:coupling-shared-k-loop}, such a shared $k$-loop exists except with probability $e^{ - \beta m/C}$. On the event of existence of such a shared $k$-loop, this being measurable with respect to the randomness exterior to it, we can use the identity coupling of the two configurations on their interior so that they agree with probability $1$ on $\Lambda_{m/2}$. 
\end{proof}

\begin{proof}[\textbf{\emph{Proof of \cref{cor:fixed-lambda-wsm}}}]
    The proof is the same as the above proof of \cref{lem:wsm-ball} except with the application of \cref{cor:general-bc-k-loop} using its second part where we allow arbitrary, unbounded, boundary conditions. 
\end{proof}

\subsection{Spatial mixing in an annulus}
We conclude the section by using the above to also show a type of weak spatial mixing result in annuli, showing that if two boundary conditions differ on the inner boundary of an annulus, they can be coupled with high probability in an annulus of half the thickness close to the outer boundary. 
For ease of notation, let $A_{\bar m,m}$ be the annulus $\Lambda_{\bar m} \setminus \Lambda_{\bar m-m}$. 

\begin{theorem}\label{thm:wsm-annulus}
Suppose that $\frac{\bar m}{2} \ge m \ge 4C_0 \beta a_m /\dl(\lambda)$, and $\log \bar m \le a_m \le \frac{m}{\log m}$. Then for any pair of boundary conditions $\phi,\phi'$ on $\partial_i \Lambda_{\bar m-m}$, having $\|\phi\|_\infty, \|\phi'\|_\infty < a_m$, we have
\begin{align*}
    \| \mu_{(0,\phi),A_{\bar m,m}} (\varphi(A_{\bar m,m/2})\in \cdot)- \mu_{(0,\phi'),A_{\bar m,m}}(\varphi(A_{\bar m,m/2})\in \cdot)\|_\tv \le 4 \bar m (e^{ - \beta m/C} + e^{ - \beta a_m/C})\,.
\end{align*}
where $(0,\phi)$ boundary conditions are $0$ on $\partial_o \Lambda_{\bar m}$ and $\phi$ on $\partial_i \Lambda_{\bar m-m}$. 
\end{theorem}

\begin{proof}
     Let $\varphi,\varphi'$ be independent draws from the two relevant distributions. We wish to show that with high probability, there will be a common loop of height $k$ sites in the annulus $A_{\bar m,3m/4}\setminus A_{\bar m,m}$. If we expose the \emph{innermost} shared $k$-loop in the annulus, it being measurable with respect to its interior, we can use the identity coupling exterior to the common $k$-loop to bound the total variation distance by the probability of non-existence of such a common $k$-loop in $A_{\bar m,3m/4}\setminus A_{\bar m,m}$.   
    
    The existence of such a common height-$k$ loop will follow if around every vertex in $\Lambda_{\bar m-3m/4}$, there is a common $k$-loop in the ball $B_{v,m/4}$ surrounding that vertex. For any $v\in \partial_i \Lambda_{\bar m-3m/4}$, consider the probability of non-existence of a common $k$-loop in $B_{v,m/4}$ surrounding $v$.

    First, expose the two configurations $\varphi$ and $\varphi'$ on the complement of $B_{v,m/4}$. Except with probability $e^{ - \beta a_m/C}$, the maximal height of the configurations revealed on $\partial_o B_{v,m/4}$ will be $2a_m$. In order to see this, observe that the configuration is stochastically below the one that has all its boundary conditions at height $a_m$ on the annulus at $\lambda =0$, for which the tail bound on the maximum height from \cref{eq:easy-SOS-maximum} applies since $a_m \ge \log \bar m$.      
    Given those boundary conditions, we can apply \cref{cor:general-bc-k-loop} and \cref{lem:coupling-shared-k-loop} to deduce that the two configurations will share a $k$-loop surrounding $v$ in $B_{v,m/4}$ except with probability $e^{ - \beta m/C}$. 
     A union bound over the at most $4\bar m$ vertices in $\partial \Lambda_{\bar m-3m/4}$ concludes the proof.   
\end{proof}

\section{Mixing time upper bounds}\label{sec:mixing-time-upper-bounds}
We now turn to the dynamical part of the paper, starting with the proofs of our mixing time upper bounds. Our focus is on the mixing time of the Glauber dynamics on the $n\times n$ box $\Lambda_n$. In order to have a finite state space, so that the mixing time is finite, we introduce a ceiling, and as is convention for such models, we will take the ceiling height to also be $n$. The floor will always be fixed to $0$, but the ceiling will be variable in this section, and therefore it will be useful to use the notation $\Lambda_{n}^\ell$ for the domain $\Lambda_n$ with ceiling at height $\ell$.

The main content of this section is establishing the upper bound on the mixing time when $\dl(\lambda)$ is going to zero with $n$. The proof goes in two stages: 
\begin{enumerate}
    \item Reducing the mixing time on $\Lambda_n$ with ceiling $n$ to the mixing time on $\Lambda_n$ with $\log n$ ceiling.
    \item Reducing the mixing time on $\Lambda_n$ with $\log n$ ceiling to the maximum of the mixing times on boxes of size $n_0$ and an annulus of width $n_0$, with arbitrary boundary conditions, for $n_0 = \tilde \Theta( \dl(\lambda))$.
\end{enumerate}
Finally, we apply a standard bound on the mixing time on boxes/annuli with cut-width  $n_0$, by an exponential in $n_0$ (possibly times a $\log n$ factor). Together, these arguments take up \cref{subsec:reducing-to-logn-ceiling,subsec:reducing-to-n0-domains,subsec:concluding-upper-bound}. 

In \cref{subsec:SSM-upper-bound}, we justify the $O(1)$ inverse gap bound when $\dl$ is fixed independent of $n$; given the weak spatial mixing with arbitrary boundary conditions for fixed $\lambda$ in \cref{cor:fixed-lambda-wsm} this is essentially reproducing the argument of \cite{CeMa2}. In \cref{subsec:torus-upper-bounds}, we explain the minor modifications needed for these upper bounds to apply to the torus, instead of zero boundary conditions.  

\subsection{Mixing time preliminaries}
Let us recall some basics from the study of mixing times for Markov chains with finite state spaces; we refer the reader to \cite{LP} for more on the topic. 

Consider a state space $\Omega$, with a transition kernel $P(x,y)$ describing the rate at which a continuous-time Markov chain 
jumps from $x$ to $y$, reversible with respect to a stationary distribution $\mu$. Let $(X_t^{x_0})_{t\ge 0}$ be the corresponding Markov chain initialized from state $x_0$, and let $L = (I-P)$ be its infinitesimal generator.  

Let $\gap$ denote the gap in the spectrum of $-L$, i.e., the size of its smallest non-zero eigenvalue, and let $\tmix$ be its mixing time, i.e., 
\begin{align*}
    \tmix(\epsilon) = \inf \{t>0: \max_{x_0 \in \Omega} \|\mathbb P(X_t^{x_0}\in \cdot) - \mu\|_\tv <\epsilon\}\,.
\end{align*}
By convention, $\tmix = \tmix(1/4)$. The $\epsilon$-mixing time satisfies a useful sub-multiplicativity property that $\tmix(\epsilon)\le  \tmix \log_2 \frac{2}{\epsilon}$. The inverse of the spectral gap, sometimes called the relaxation time, is closely tied to the mixing time. Namely, 
\begin{align}\label{eq:tmix-gap}
	(\gap^{-1} -1) \log\big(\tfrac{1}{2\epsilon}\big) \le \tmix(\epsilon) \le \gap^{-1} \log\big(\tfrac{1}{\epsilon \mu_{\min}}\big)\,.
\end{align}
The SOS Glauber dynamics for an SOS distribution $\mu$ of the form of \cref{eq:general-SOS-measure} has transition kernel $P(\varphi,\varphi')$ which is non-zero only if $\varphi,\varphi'$ differ at exactly one vertex, say $v$, by exactly one height at that vertex. In that case, if $\varphi^{(v,\iota)}$ denotes the configuration $\varphi$ with the value at $v$ changed to $\varphi_v + \iota$, 
\begin{align*}
	P(\varphi, \varphi') = \frac{\mu(\varphi' \mid (\varphi_w)_{w\ne v})}{\sum_{\iota\in \{-1,0,1\}} \mu(\varphi^{(v,\iota)} \mid (\varphi_w)_{w\ne v})}
\end{align*}
 This is easily checked to be reversible with respect to $\mu$. 

The monotonicity of \cref{lem:SOS-monotonicity} implies monotonicity for the SOS Glauber dynamics, whereby if $x_0 \ge y_0$, then $X_t^{x_0} \succeq X_t^{y_0}$ for all $t\ge 0$. In particular, there is a \emph{grand coupling} of all the SOS dynamics chains from all the possible initializations, such that with probability $1$, $X_t^{x_0}\ge X_t^{y_0}$ for all $t\ge 0$ and all $x_0 \ge y_0$.

\subsection{Reducing the mixing time to \texorpdfstring{$\log n$}{log n} ceiling}\label{subsec:reducing-to-logn-ceiling}
Let us introduce the notation $$T_{\Lambda_n^\ell}:= \tmix(\Lambda_n^\ell)$$ for the mixing time of SOS dynamics at parameters $\beta,\lambda$ on $\Lambda_n$ with floor at height $0$, ceiling at height $\ell$, and height $0$ boundary conditions. Our reduction of the ceiling from $n$ to $\log n$ uses the general censoring inequality of~\cite{PWcensoring}. A similar argument at $\lambda =0$ can be found in~\cite[Section 6.3]{CLMST14}. The censoring inequality says that if started from a maximal initial state in a monotone dynamics, the total-variation distance of the configuration at time $t$ to stationarity can only increase if pre-specified updates are ignored at pre-specified times; furthermore this can only stochastically increase the law of the configuration at time $t$. While this usually refers to ignoring updates at certain sites for the Glauber dynamics, we apply it to censor moves that would take the SOS dynamics below or above certain heights, effectively imposing different floors/ceilings on the Glauber dynamics. This application of censoring was also used in~\cite[Theorem 2.2]{CLMST14}.

\begin{definition}\label{def:SOS-censoring}
    A censoring scheme for an SOS Glauber dynamics chain $X_t$ prescribes a sequence of times $t_0<t_1<...$, subsets $V_i$ of $\Lambda_n$, and heights $a_i<b_i$ such that between times $[t_{i-1},t_i]$, the only updates that are permitted are those that move heights of vertices in $V_i$ and only if they move them between $a_i$ and $b_i$. 
\end{definition}

    We therefore have in our context that if $X_t^\ell$ (resp.\ $X_t^0$) is the SOS Glauber dynamics on $\Lambda_n^\ell$ initialized from the maximal (resp., minimal) height of $\ell$ (resp., $0$) everywhere, and $\bar{X}_t^\ell$ (resp., $\bar{X}_t^0$) is a censoring of it per \cref{def:SOS-censoring}, then for all $t\ge 0$, 
    \begin{align}\label{eq:censoring-ineq}
        \bar{X}_t^\ell \succeq X_t^\ell\,, \qquad \text{and}\qquad \bar{X}_t^0 \preceq X_t^0\,.
    \end{align}

    Our goal in this subsection is to use censoring to show the following reduction. 
\begin{lemma}\label{lem:reducing-to-logn-ceiling}
There is an absolute constant $C$ such that for all $\beta>\beta_0$, for any $\lambda \ge 0$, we have 
\begin{align*}
    T_{\Lambda_n^n} \le C n  T_{\Lambda_n^{\log n}}\,.
\end{align*}
\end{lemma}

\begin{proof}
    For a general initial state $x_0$, use the shorthand $X_t^{x_0}$ to denote the Glauber dynamics chain $X_t$ initialized from $x_0$, and use $X_t^\mu$ to denote the one initialized from the stationary distribution $\mu_{0,\Lambda_n^n}$. Under the grand coupling, by a union bound and monotonicity, 
    \begin{align}\label{eq:monotonicity-mixing-reduction}
                \max_{x_0}\|\mathbb P(X_t^{x_0} \in \cdot) - \mu_{0,\Lambda_n^n}\|_\tv \le \sum_{v\in \Lambda_n} \mathbb P(X_t^{x_0}(v) \ne X_t^\mu(v)) & \le \sum_{v\in \Lambda_n} \mathbb P(X_t^{n}(v) \ne X_t^0(v))\,. 
    \end{align}
    Writing  $\mathbb P(X_t^{n}(v) \ne X_t^0(v))  = \mathbb P(X_t^n(v) - X_t^0 (v) \ge 1)$ and using Markov's inequality, and then \cref{eq:censoring-ineq},
        \begin{align}\label{eq:tv-to-censored-reduction}
         \max_{x_0} \|\mathbb P(X_t^{x_0} \in \cdot) - \mu_{0,\Lambda_{n}^{n}}\|_\tv \le \sum_v \mathbb E[X_t^n(v)] - \mathbb E[X_t^0(v)] \le \sum_v \mathbb E[\bar X_t^n(v)]- \mathbb E[\bar X_t^0(v)]\,.
    \end{align}
    for any censoring scheme of the form of \cref{def:SOS-censoring}. The censoring scheme we use is defined as follows.      Define a sequence of epochs, each of length $\Delta t:= C_1 T_{\Lambda_n^{\log n}} \log n$ for a large absolute constant $C_1$, via 
     \begin{align*}
         t_0 = 0\,, \qquad \text{and} \qquad t_i = t_{i-1} + \Delta t\,.
     \end{align*}
    for $1\le i \le N:=\frac{2n}{\log n}$ (note that $t_N \le Cn T_{\Lambda_n^{\log n}}$). Define corresponding floor and ceiling pairs 
     \begin{align*}
         a_i = n- \tfrac{i}{2} \log n\,, \qquad \text{and} \qquad b_i = a_i + \log n\,.
     \end{align*}
     Let $\bar X_t^n$ be Glauber dynamics initialized from all-$n$, with the following censoring: 
     \begin{itemize}
    \item For each $i$, for times in $[t_{i-1},t_i]$, updates are only allowed in $\Lambda_n \times [a_i,b_i]$. 
\end{itemize}
    Let $\bar X_t^0$ be the Markov chain initialized from all-$0$, with the following censoring scheme: 
    \begin{itemize}
    \item For all $t\ge 0$, updates are only allowed in $\Lambda_n \times [a_N,b_N]= \Lambda_n \times [0,\log n]$. 
\end{itemize}
    By the definition of $T_{\Lambda_n^n}= \tmix(\Lambda_{n}^n)$ it suffices to show that at $t= t_N$, each difference in expected values on the right-hand side of \cref{eq:tv-to-censored-reduction} is within $n^{-4}$ of the stationary expected value $\mu_{0,\Lambda_n^n}[\varphi_v]$. We start with the minimal chain since that is easier; since all heights are bounded above by $n$, it suffices to show that the TV-distance between $\mathbb P(\bar X_t^0(v) \in \cdot)$ and $\mu_{0,\Lambda_n^n}(\varphi_v \in \cdot)$ are within $n^{-5}$, and the bound on the difference of expectations will follow.  Towards that, notice that since $\bar X_0^0 \le \log n$ everywhere, $\bar X_t^0$ is exactly an (uncensored) SOS Glauber dynamics on $\Lambda_n^{\log n}$ initialized from all-$0$. As such, after time $t \ge \Delta t$, by sub-multiplicativity of TV-distance to equilibrium, we have that 
    \begin{align*}
       \| \mathbb P(\bar X_t^0(v) \in \cdot) - \mu_{0,\Lambda_n^{\log n}}\|_\tv \le n^{-6}\,.
    \end{align*}
    At the same time, since the $\mu_{0,\Lambda_n^n}$-probability that $\{\max_v \varphi_v\le \log n\}$ is $1-n^{ - 6}$ per \cref{eq:easy-SOS-maximum}, we get that 
    \begin{align}\label{eq:censoring-lower-chain-bound}
         \| \mathbb P(\bar X_t^0(v) \in \cdot) - \mu_{0,\Lambda_n^n}\|_\tv \le \| \mathbb P(\bar X_t^0(v) \in \cdot) - \mu_{0,\Lambda_n^{\log n}}\|_\tv + \|\mu_{0,\Lambda_n^{\log n}} - \mu_{0,\Lambda_n^n}\|_\tv \le O(n^{-6})\,.
    \end{align}

     We now turn to controlling the chain $\bar X_t^n$. We will show that for every $i<N$, the following event holds with high probability. 
    \begin{align*}
        E_i := \bigcap_{v\in \Lambda_n} \{\bar X_{t_i}^n(v) \le b_{i+1}\}\,.
    \end{align*}
    If this event holds for all $i$, then between times $[t_{i},t_{i+1}]$, the chain $\bar X_t^n$ is exactly an SOS chain with floor $a_{i+1}$ and ceiling $b_{i+1}$ (and boundary condition $0$ which is equivalent to boundary condition $a_{i+1}$ as noted in \cref{eq:boundary-condition-equivalence}) run for a time $\Delta t$. Namely if $Y_s^i$ is a standard Glauber dynamics with floor and boundary condition $a_i$, ceiling $b_i$, initialized from $\bar X_{t_{i-1}}^n$, then on the event $E_{i-1}$, we have $\bar X_{t_{i-1}+s}^n \stackrel{d}= Y_s^{(i)}$ for $s\le \Delta t$. Note that each $Y_s^{(i)}$ is a vertical translate of Glauber dynamics on $\Lambda_n^{\log n}$ with $0$ boundary conditions, and thus has mixing time $T_{\Lambda_n^{\log n}}$. With this, we can write  
     \begin{align}\label{eq:ceiling-reduction-tv-distance}
         \|\mathbb P(\bar X_t^n \in \cdot) - \mu_{0,\Lambda_n^n}\|_\tv \le \sum_{i\le N} \mathbb P(E_{i-1},E_i^c) + \max_{\eta: \|\eta\|_\infty \le \log n} \|\mathbb P(Y_{\Delta t}^{(N)} \in \cdot \mid Y_{0}^{(N)} = \eta) - \mu_{0,\Lambda_n^n}\|_\tv\,.
     \end{align}
     The total variation distance is bounded just like the bound for $\bar X_t^0$. For the probabilities in the sum, since on $E_{i-1}$, we have $\bar X_{t_{i-1}+s}^n \stackrel{d}= Y_s^{(i)}$ for all $s\le \Delta t$, it suffices to bound the probability (maximized over its possible initializations) that $Y_{\Delta t}^{(i)}$ has maximum height $b_{i+1} = a_i + \frac{1}{2}\log n$. 
     
     By definition of $\Delta t$ and the sub-multiplicativity of TV-distance to stationarity, $Y_{\Delta t}^{(i)}$ will be within distance $n^{-6}$ of its equilibrium distribution, which we said is a vertical shift of $a_i$ above a sample from $\mu_{0,\Lambda_n^{\log n}}$. A sample from that stationary distribution, will have maximum height at most $\frac{1}{2}\log n$ except with probability $n^{-6}$ per \cref{eq:easy-SOS-maximum}. Together this implies an $O(n^{-6})$ bound on the probability that $Y_{\Delta t}^i$ will have maximum height larger than $b_{i+1}$, and thus on $\mathbb P(E_{i-1},E_i^c)$. Plugging in to \cref{eq:ceiling-reduction-tv-distance}, using $N= o(n)$, and combining with \cref{eq:censoring-lower-chain-bound} gives an $n^{-5}$ bound on $\mathbb E[\bar X_{t_N}^n(v)]- \mathbb E[\bar X_{t_N}^0(v)]$ implying that $T_{\Lambda_n^n}\le t_N$ by \cref{eq:tv-to-censored-reduction}. 
\end{proof}

\subsection{Reducing the mixing time to domains of cut-width \texorpdfstring{$n_0$}{n\_0}}\label{subsec:reducing-to-n0-domains}
Now that we have reduced the ceiling down to $\log n$, we wish to bound the mixing time on $\Lambda_n^{\log n}$ with zero boundary conditions, by the mixing time on smaller domains of side-length $n_0$, the minimal scale at which spatial mixing kicks in (depending on $\dl(\lambda)$). 

Recall the definition of the annulus $A_{n,m} = \Lambda_n \setminus \Lambda_{n-m}$, and let $A_{n,m}^\ell$ denote the annulus $A_{n,m}$ with ceiling at height $\ell$. Let 
\begin{align*}
    T_{A_{n,m}^\ell} := \tmix(A_{n,m}^\ell)\,,
\end{align*}
i.e., the mixing time on $A_{n,m}^\ell$ with $0$ boundary conditions. We also need to consider mixing times over more general boundary conditions, so we use a superscript $\phi$ to indicate a boundary condition different from $0$. 

\begin{lemma}\label{lem:localizing-dynamics}
    There exists a universal $C_0$ such that for all $n\ge m \ge \max\{C_0 \beta \log n / \dl(\lambda), \log n \log \log n\}$, 
    \begin{align*}
        T_{\Lambda_n^{\log n}} \le C \log n \cdot \max\{\max_{\phi}T^\phi_{\Lambda_m^{\log n}}, \max_{\phi}T^{\phi}_{A_{n,m}^{\log n}}\}\,.
    \end{align*}
    where the maxima are over boundary conditions having $\|\phi\|_\infty \le \log n$. 
\end{lemma}

\begin{proof}
    Similar to \cref{eq:monotonicity-mixing-reduction}, under the grand coupling, 
    \begin{align*}
       \max_{x_0} \mathbb P(X_t^{x_0} \ne X_t^\mu) \le \sum_{v\in \Lambda_n} \mathbb P(X_t^{\log n}(v) \ne X_t^0(v))\,,
    \end{align*}    
    where the maximal initialization is now identically $\log n$ since that is the ceiling height. 
    We aim to localize the dynamics by introducing Markov chains that only make updates in local neighborhoods around $v$. 
    For every $v\in \Lambda_{n-m}$, let $B_{v,m}$ be the ball of radius $m$ around $v$, and let $X_{t,v}^{\log n}$ and $X_{t,v}^0$ be the Glauber dynamics chains that only make updates inside $B_{v,m}$ (still with the same floors and ceilings). In particular, these will be SOS chains on $B_{v,m}$ with boundary conditions $\log n$, and $0$ respectively. 
    For $v\in A_{n,m}$, use the annulus $A_{n,2m}$ as its corresponding block, and (overloading notation slightly), let $X_{t,v}^{\log n},X_{t,v}^{0}$ be the SOS Glauber dynamics chains that only make updates inside $A_{n,2m}$; these will be SOS chains on $A_{n,2m}$ with boundary conditions that are $0$ on the outer boundary, and $\log n$ or $0$ respectively on the inner boundary.  

    Fix any $v\in \Lambda_n$, and notice that by monotonicity, under the grand coupling, 
    \begin{align*}
        \mathbb P(X_t^{\log n}(v) \ne X_t^{0}(v)) \le \mathbb P(X_{t,v}^{\log n}(v) \ne X_{t,v}^{0}(v))\,.
    \end{align*}
    This latter probability is bounded by 
    \begin{align}\label{eq:localization-mixing-time-split}
        \|\mathbb P(X_{t,v}^{\log n}(v)\in \cdot) - \mu_{\log n,B_v}\|_\tv + \|\mathbb P(X_{t,v}^{0}(v)\in \cdot) - \mu_{0,B_v}\|_\tv + \|\mu_{\log n,B_v}(\varphi_v\in \cdot) - \mu_{0,B_v}(\varphi_v\in \cdot)\|_\tv\,,
    \end{align}
    where we used $B_v$ as a generic stand-in for the block of $v$, be it the ball $B_{v,m}$ or the annulus $A_{n,2m}$. If $v\in \Lambda_{n-m}$ so that $B_v = B_{v,m}$, as 
     long as $t$ is at least $C_1 \log n$ (for a large absolute constant $C_1$)  times $T_{\Lambda_m^{\log n}}^\phi$ with boundary conditions $\phi \equiv \log n$, the first of these is at most $n^{-10}$, and as long as it is at least $C_1 \log n$ times the mixing time on $T_{\Lambda_m^{\log n}}^\phi$ with $\phi \equiv 0$ the second of these is at most $n^{-10}$. Similarly when $v\in A_{n,m}$ as long as $t$ is at least $C_1 \log n$ times $T_{A_{n,m}^{\log n}}^{(0,\phi)}$ with $\phi \equiv \log n$ or with $\phi \equiv 0$. 

     Finally, the third term in \cref{eq:localization-mixing-time-split} is governed by the spatial mixing estimates we established in the previous section. Namely, if $v\in \Lambda_{n-m}$ so that $B_v = B_{v,m}$, then by \cref{lem:wsm-ball} with $a_m = \log n$, it is at most $e^{ - \beta m/C}$ which is at most $n^{-10}$, for large $\beta$. Similarly, if $v\in A_{n,m}$ so that $B_v = A_{n,2m}$, then by \cref{thm:wsm-annulus} with the choices $\bar m = n$ and $a_m = \log n$, this is at most $n^{-10}$ as well. 

    Altogether, we get that so long as $t$ is at least $C_1$ times $\log n$ times $\max\{\max_{\phi}T^\phi_{\Lambda_m^{\log n}}, \max_{\phi}T^{\phi}_{A_{n,m}^{\log n}}\}$ for some absolute constant $C_1$, we have for all $v\in \Lambda_n$, 
    \begin{align*}
        \max_{v\in \Lambda_n} \mathbb P(X_t^{\log n}(v)  \ne X_t^{0}(v)) \le O(n^{-10})\,,
    \end{align*}
    which when summed over $v\in \Lambda_n$, yields the claimed bound on the mixing time.
\end{proof}

\subsection{Concluding the upper bound}\label{subsec:concluding-upper-bound}
We are now in position to conclude the upper bounds on the inverse spectral gap in \cref{mainthm:dynamical} when $\lambda$ approaches a critical $\lambda_c^{(k)}$ and/or $0$. 

\begin{proof}[\textbf{\emph{Proof of \cref{it-thm-1-O(1/f)} of \cref{mainthm:dynamical}: upper bound}}]
By \cref{lem:reducing-to-logn-ceiling} and \cref{lem:localizing-dynamics}, with  
\begin{align*}
    m = \max\{C_0 \beta \log n/\dl(\lambda), \log n\log \log n\}\,,
\end{align*}
(which is less than $n$ as required so long as $\dl(\lambda)\ge \frac{C_0 \beta \log n}{n}$ as assumed in \cref{it-thm-1-O(1/f)} of \cref{mainthm:dynamical})  
we obtain for some absolute constant $C$, the bound 
\begin{align}\label{eq:tmix-reduction}
    T_{\Lambda_n^n} \le C n\log n \cdot \big(\max_{\phi}T_{\Lambda_m^{\log n}}^\phi \vee \max_{\phi} T_{A_{n,m}^{\log n}}^\phi\big)\,.
\end{align}
The following is an upper bound on the mixing times on the right-hand side of the above that are exponential in $m \log n$, and essentially conclude the proof.  

\begin{lemma}\label{lem:canonical-path-bound-sos}
    There is an absolute constant $C$, such that for all $\beta>0$, all $\lambda\ge 0$,  
    \begin{align*}
        \max_{\phi: \|\phi\|_\infty\le \ell} T_{\Lambda_{m}^{\ell}}^\phi \le e^{ C\beta m\ell + \lambda}\,, \qquad \text{and}\qquad  \max_{\phi: \|\phi\|_\infty\le \ell} T_{A_{n,m}^\ell}^\phi \le n e^{C \beta  m \ell + \lambda}\,.
    \end{align*}
\end{lemma}

\cref{lem:canonical-path-bound-sos} is a by-now standard estimate using the canonical path method (\cite[Corollary 13.21]{LP}) to bound the mixing time by the cut-width of the underlying graph, a technique first developed in~\cite{JS} and applied in the context of the Ising model~\cite{Martinelli-phase-coexistence}. We could not find a statement at the requisite level of generality for the SOS dynamics with general field, floors, and ceilings, (without a field, a similar bound is in \cite[Proposition 2.3]{CLMST14})
so we have included a proof for completeness in \cref{sec:canonical-path-SOS}. Plugging the bounds of \cref{lem:canonical-path-bound-sos} into \cref{eq:tmix-reduction}, 
\begin{align*}
    T_{\Lambda_n^n} \le \exp\Big(C_\beta (\log n)^2 \max \Big\{\frac{1}{\dl(\lambda)}, \log \log n\Big\}\Big)\,,
\end{align*}
and the analogous bound on the inverse spectral gap per \cref{eq:tmix-gap}. 
Notice that if we simply assume $\lambda \ge 0$, by \cref{lem:reducing-to-logn-ceiling} and \cref{lem:canonical-path-bound-sos}, we also always have an upper bound of $\exp( C_\beta n \log n)$. 
\end{proof}

\begin{remark}\label{rem:getting-rid-of-the-log}
    The upper bound of \cref{lem:canonical-path-bound-sos} can be improved to a bound of $e^{ C \beta m}$ (without the dependence on $\ell$ in the exponent when $\lambda$ is kept away from zero, say $\lambda>\epsilon$ uniformly, since the cost to force a set of sites to all take value $1$ is only exponential in the number of sites: see the argument in \cite[Theorem 4.1]{CeMa2}.  
\end{remark}

\subsection{Sharp upper bounds using strong spatial mixing}\label{subsec:SSM-upper-bound}
When $\dl(\lambda)$ is fixed positive, independent of $n$, the above reasoning gives an upper bound of $\exp(n^{o(1)})$. In this section, we give a better upper bound using the exponential-rate spatial mixing properties of the model, using strong spatial mixing which controls decay of correlations in the presence of nearby boundary conditions. 

\begin{proof}[\textbf{\emph{Proof of \cref{it-thm-1-O(1)} of \cref{mainthm:dynamical}}}]
Let us first define the notion of strong spatial mixing we will use. 

\begin{definition}\label{def:SSM-SOS}
    The SOS model has SSM with constant $C$ above scale $n_0$ if for all $n\ge n_0$, for all $y\in \partial \Lambda_n$, and any subset $A\subset \Lambda_n$
    \begin{align*}
        \max_{\phi,\phi': \phi_x = \phi'_x  \forall x\ne y} \|\mu_{\phi,\Lambda_n}(\varphi(A)\in \cdot) - \mu_{\phi',\Lambda_n}(\varphi(A)\in \cdot)\|_\tv \le e^{ - d(A,y)/C}\,.
    \end{align*}
\end{definition}

It is a classical result of~\cite{MOS-WSM-SSM} that weak spatial mixing implies strong spatial mixing (above a certain scale) for finite-range 2D spin systems with finite state space per spin. As noted in \cite[Section 2.1]{CeMa2}, in our context with infinitely-many possible heights available to each site, the strong spatial mixing still follows from weak spatial mixing so long as $\lambda>0$ uniformly in $n$. This is because as long as $n_0 > \lceil 8\beta / \lambda \rceil$, there is a constant probability that a $n_0 \times n_0$ box entirely takes height $1$, say, uniformly over its boundary conditions (\cite[Eq.~(2.1)]{CeMa2}), whence the rest of the boosting of weak spatial mixing to strong spatial mixing can be carried out, resulting in the following consequence of \cref{cor:fixed-lambda-wsm}. 

\begin{lemma}\label{lem:SSM}
    For every $\epsilon>0$, there exists a $C_0(\epsilon,\beta)>0$ such that for all $\lambda$ having $\dl(\lambda)>\epsilon$, the SOS model has SSM with constant $C_0$ above scale $C_0$. 
\end{lemma}

To go from \cref{lem:SSM} to an $O(1)$ inverse spectral gap is quite standard. We sketch the steps below with relevant references to demonstrate that so long as $\lambda>0$ uniformly in $n$, there is no issue with this step: we refer to~\cite[Section 2.2]{CeMa2} for further details. Let $L_0 = C_1 C_0$ where $C_0$ is given by \cref{lem:SSM} and $C_1$ is a constant depending on $C_0$ to be chosen. 
\begin{enumerate}
    \item Cover the box by overlapping $\ell^\infty$ balls (called blocks) of side-length $L_0$; 
    \item The corresponding block dynamics that assigns blocks independent Poisson clocks and performs a heat-bath update on a block when its clock rings is contractive as long as $L_0$ is sufficiently large as a function of $C_0$ (this follows from path coupling with the Hamming distance, and the amenability of $\mathbb Z^2$); in particular, the block dynamics has an $O(1)$ inverse spectral gap \cite[Eq. (2.3)]{CeMa2};
    \item When $\lambda>0$, each individual block has an $O(1)$ inverse spectral gap (though blowing up as $\lambda \downarrow 0$), since its base side-lengths are $L_0 = O(1)$: see \cite[Proposition 2.2]{CeMa2}. 
\end{enumerate}
The $O(1)$ inverse spectral gap for the SOS Glauber dynamics follows from these by the classical block dynamics bound on the inverse spectral gap \cite[Proposition 3.4]{MartinelliLectureNotes}. 
\end{proof}

\begin{remark}
    The dependence of $n_0$ for which \cref{def:SSM-SOS} holds depends exponentially on the $m_0$ at which the weak spatial mixing \cref{lem:wsm-ball} kicks in, so using the above approach when the weak spatial mixing only holds above a diverging scale, e.g., above $1/\dl(\lambda)$, would give very sub-optimal mixing time upper bounds, namely double exponential in $1/\dl(\lambda)$, rather than simply exponential in $1/\dl(\lambda)$. 
\end{remark}

\subsection{The case of the torus}\label{subsec:torus-upper-bounds}
We describe the modifications to get the upper bounds for the torus. 

\begin{proof}[\textbf{\emph{Proof of \cref{it-torus-O(1/f)} of \cref{mainthm:dynamical-torus}: upper bound}}]
In the bounds of \cref{subsec:reducing-to-logn-ceiling}, the key role played by the boundary conditions was in enabling the a priori bound ensuring that the maximum height of the interface at equilibrium is at most $\frac{1}{2}\log n$ except with probability $n^{-5}$, say. Such a bound evidently does not hold at $\lambda =0$ as there will be no force keeping the interface pinned close to height $0$, but as long as we assume $\lambda \ge n^{-1+o(1)}$ say, the following lemma replaces that estimate. 

\begin{lemma}\label{lem:torus-max-oscillation}
    Suppose $\lambda \ge \frac{e^{8\beta}\log n}{n}$. For $\beta>\beta_0$, the SOS model on $\Lambda_n^{\log n}$ with periodic boundary conditions, which we denote using the shorthand $\phi = p$,  has 
    \begin{align*}
        \mu_{p,\Lambda_{n}^{\log n}}(\max_v \varphi_v \ge \tfrac{1}{2}\log n) \le n^{-5}\,.
    \end{align*}
\end{lemma}
\begin{proof}
     Fix a vertex $v$, the transitivity of the torus implying we can simply fix the origin vertex $o$, and consider $\mu_{p,\Lambda_n^{\log n}}(\varphi_o \ge \frac{1}{2}\log n)$. 
     Due to the ceiling at height $\log n$, we have $\mu_{p,\Lambda_n^{\log n}}$ is stochastically dominated by $\mu_{\log n, \Lambda_n^{\log n}}$, i.e, with maximal boundary conditions, so it suffices to bound $\mu_{\log n, \Lambda_n^{\log n}}(\varphi_o \ge \frac{1}{2}\log n)$.

    Let $k$ be the $k$ for which $\lambda \in (\lambda_c^{(k)}, \lambda_c^{(k-1)})$ and using monotonicity, lower $\lambda$ down to $I_{k+1}$, whence $\dl(\lambda) \ge \frac{e^{4\beta}\log n}{n}$. Moreover, by~\cref{eq:I-i} and the fact that $\lambda \le \frac{n}{\log n}$ we know $k+1$ to be at most $(C/\beta) \log n$.
    By \cref{cor:general-bc-k-loop} with $a_m = \log n$ (applicable since $\dl(\lambda)\ge \frac{n}{e^{4\beta} \log n}$), there will be a height $k$ loop surrounding $\Lambda_{n/2}$ except with probability $e^{ - n}$. Interior to that height $k$-loop, the rigidity results at height $k$ from \cref{eq:height-deviation-k-full-window} imply that the probability that $\varphi_0$ exceeds $k + r$ is at most $e^{ - 4(\beta - C)r}$. Thus, for $\beta$ large, for every vertex $v$ in the torus, 
     $$\mu_{p, \Lambda_n^{\log n}}(\varphi_v \ge \tfrac{1}{2}\log n) \le n^{-7}\,,$$
     whence a union bound over the $n^2$ many such vertices implies the bound on their maximum. 
\end{proof}

With this upper bound in hand, the remainder of the proof of \cref{eq:ceiling-reduction-tv-distance} goes through unchanged when $\dl(\lambda)\ge \frac{c\beta \log n}{n}$. The second reduction of \cref{lem:localizing-dynamics} works as stated on the torus, though it makes more sense in the torus to only use blocks of the form $B_v = B_{v,m}$ ($\ell^\infty$-balls of radius $m$ about each vertex $v$ since the graph is vertex transitive and our use of the annulus was to deal with behavior near the boundary).  On the other hand, when $\dl(\lambda)\le \frac{c\beta \log n}{n}$, the reduction of \cref{lem:reducing-to-logn-ceiling} together with the canonical paths bound of \cref{lem:canonical-path-bound-sos} gives an upper bound of $\exp(C \beta n\log n)$ on the torus. 
\end{proof}

\begin{proof}[\textbf{\emph{Proof of \cref{it-torus-O(1)} of \cref{mainthm:dynamical-torus}}}]
Regarding the $O(1)$ inverse gap estimates when $\dl(\lambda)>\epsilon$ uniformly in $n$, the proof sketched in \cref{subsec:SSM-upper-bound} is easily checked to go through without any modifications. 
\end{proof}

\section{Mixing time lower bounds}\label{sec:mixing-time-lower-bounds}
Our aim in this section is to establish the complementary lower bounds on the mixing times on $\Lambda_n^n$ with zero and periodic boundary conditions. Our lower bounds are based on \emph{bottlenecks} in the state space. Let us recall the basic Cheeger inequality lower bound for Markov chains. The Cheeger constant of a Markov chain with state space $\Omega$ and stationary distribution $\mu$ is 
\begin{align}\label{eq:bottleneck-ratio}
    \Phi_\star = \min_{A\subset \Omega} \frac{Q(A,A^c)}{\mu(A)\mu(A^c)}\,, \quad  \text{where} \quad Q(A,A^c) = \sum_{\omega\in A} \mu(\omega) \sum_{\omega'\in A^c} P(\omega,\omega')\,.
\end{align}
Cheeger's inequality then states that the inverse spectral gap has
\begin{align}\label{eq:Cheeger-inequality}
    \gap^{-1} \ge \frac{1}{2}\Phi_\star^{-1}\,.
\end{align}

One can easily see that $Q(A,A^c)\le \mu(\partial A)$ where $\partial A$ indicates all configurations in $A$ having positive transition rates to $A^c$. It therefore suffices to construct a set $A$ having small conditional probability $\mu(\partial A)/\mu(A) = \mu(\partial A \mid A)$ to lower bound the inverse spectral gap. 

\subsection{Height-zero boundary conditions}
We show a bottleneck between configurations that are predominantly at height $k-1$ vs.\ predominantly at height $k$, when $\lambda$ is close to $\lambda_c^{(k-1)}$, but in $(\lambda_c^{(k)},\lambda_c^{(k-1)})$. This forms a bottleneck because any contour collection whose diameter is too small (depending on $\dlp(\lambda)$) would prefer to remain at height $k-1$ than rise up to $k$, even if the equilibrium measure prefers height $k$. 

\begin{proof}[\textbf{\emph{Proof of \cref{it-thm-1-O(1/f)} of \cref{mainthm:dynamical}: lower bound}}]
For a configuration $\varphi$, let $\cV_{\ge k}\subset \Lambda_n$ denote its collection of all $\{\ge k\}$ sites. Let 
\begin{align*}
    A_{r} = \{\text{all connected components of $\cV_{\ge k}$ have size at most $r$}\}\,.
\end{align*}
Letting $o= (0,0)$ denote the origin, our choice of bottleneck set will be 
\begin{align*}
    A = A_r\,, \qquad \text{for} \qquad r = \min\{\tfrac{n}{8},\tfrac{1}{8\dlp(\lambda)}\}\,.
\end{align*}
We firstly claim that for an absolute constant $C_0$, as long as $n\ge C_0 \beta k/\dlp(\lambda)$, 
\begin{align}\label{eq:bottleneck-small-mass}
    \mu_{0,\Lambda_n^n}(A^c) \ge 1/2\,.
\end{align}
Indeed, using \cref{lem:reaching-height-k}, we get a loop of height-$k$ sites surrounding $\Lambda_{n/2}$, say, except with probability $e^{ - \beta n/C}$, and by definition such a loop must have size at least $\frac{n}{2}\ge r$. 

In order for a configuration to belong to $\partial A$, there must be a site $v\in \Lambda_n$, such that if its height is changed the configuration leaves $A$; this requires that $v$ be adjacent to some connected component of $\cV_{\ge k}$ having size at least $r/2$. In particular, 
\begin{align*}
    \partial A \subset \bigcup_{v} A \cap A_{r/2,v}^c\,,
\end{align*}
where $A_{r/2,v}^c$ is the event that there is a connected component of $\cV_{\ge k}$ of size at least $r/2$ incident to $v$. By a  union bound and \cref{eq:bottleneck-small-mass}, 
\begin{align*}
    \Phi_\star \le 2n^2 \max_{v} \mu_{0,\Lambda_n^n}(A_{r/2,v}^c \mid A)\,.
\end{align*}
On the event $A$, there must be a loop (possibly using boundary sites which all have height $0$) of height-$\{\le k-1\}$ sites interior to $B_v$, the $\ell^\infty$ ball of radius $r$ about $v$ in $\Lambda_n$ (otherwise there would be a $\{\ge k\}$ path from $v$ to $\partial_o B_v \setminus \partial_o \Lambda_n$, which would necessarily have length at least $r$, violating $A$). Exposing the outermost such loop, calling its strict interior $\cC_{k-1}\subset B_v$, this bounds the desired probability by 
\begin{align*}
    \mu_{0,\Lambda_n^n}(A_{r/2,v}^c\mid A) \le \max_{\cC_{k-1}\subset B_v}\max_{\phi:\|\phi\|_\infty \le k-1}\mu_{\phi,\cC_{k-1}}(A_{r/2,v}^c\mid A)\,.
\end{align*}
The events $A, A_{r/2,v}$ are both decreasing on this state space, and so by the FKG inequality, the probability in question is at most  $\mu_{\phi,\cC_{k-1}}(A_{r/2,v}^c)$ (without the conditioning on $A$). By monotonicity, this probability is maximized by the maximal boundary conditions $\phi$ which would be $k-1$, so we get 
\begin{align*}
    \Phi_\star \le 2n^2 \max_{v} \max_{\cC_{k-1}\subset B_v} \mu_{k-1,\cC_{k-1}}(A_{r/2,v}^c)\,.
\end{align*}
We control this last probability via a Peierls map. The event $A_{r/2,v}^c$ requires there to be an outermost contour of diameter at least $r/2$ (but at most $2r$ since $\cC_{k-1}\subset B_v$) confining $v$ in its interior. Considering the renormalized contour representation and performing the Peierls operation that deletes the outermost contour $\gamma_v$ confining $v$ in its interior, the weight change under application of this map is $W_{k-1}^\rn(\gamma_v)$ which by \cref{cor:non-elem-control-smaller-than-delta-tr} is at most $e^{ - (\beta - 5)|\gamma_v|}$ because $\diam(\gamma_v)\le 2r \le 1/(4\dlp(\lambda))$, which in turn is at most $(f^\tr_{k} - f^\tr_{k-1})^{-1}$ per \cref{cor:delta-f-delta-l-relation}.  Summing over the possible choices of $\ell$, we arrive at  
\begin{align*}
    \Phi_\star \le 2n^2 \sum_{\ell \ge r} 4^\ell e^{ - (\beta - 5)\ell}\le n^2 e^{ - (\beta - C)r}\,.
\end{align*}
This then implies the claimed bound on the inverse gap per \cref{eq:Cheeger-inequality}. 
    \end{proof}

\subsection{The case of the torus}
 
We describe how to get the analogous lower bound on the torus. 

\begin{proof}[\textbf{\emph{Proof of \cref{it-torus-O(1/f)} of \cref{mainthm:dynamical-torus}: lower bound}}]
Let $k$ be such that $\dl(\lambda)$ is attained by $\lambda_c^{(k-1)}$, whether by $\dl(\lambda) = \dlp(\lambda)$, in which case $\lambda =  \lambda_c^{(k-1)}- \dl(\lambda)$, or if $\dl(\lambda) = \dlm(\lambda)$ in which case $\lambda = \lambda_c^{(k-1)} + \dl(\lambda)$. Consider the following two events: 
\begin{align*}
    A_{r} & = \{\text{all connected components of $\cV_{\ge k}$ have size at most $r$}\}\,, \\ 
    A_{r}' & = \{\text{all connected components of $\cV_{\le k-1}$ have size at most $r$}\}\,.
\end{align*}
Similar to the proof with zero boundary conditions, fix $r = \min\{\frac{n}{8}, \frac{1}{\dl(\lambda)}\}$. Notice that $A_{r}$ and $A'_{r}$ are disjoint events because on $A_{r}$, the boundary in $\Lambda_{2r}$ of the $\{\ge k\}$ components intersecting $\partial_i \Lambda_{2r}$, say, forms a $\{\le k-1\}$ path of diameter at least $r$. As such, for every $\lambda$, one of these two events must have probability less than $1/2$. Thus, the bound of \cref{eq:bottleneck-small-mass} is ensured by the choice of $A = A_{r}$ vs.\ $A_{r}'$. 

From that point on, the argument is essentially identical to the argument in the previous subsection if $A= A_r$ with the balls $B_v$ now all being identical translates of one another. We had used that $\diam(\gamma_v)\le 2r\le 1/(4\dl(\lambda))$ in order to reason that $W_{k-1}^\rn(\gamma_v)$ had an exponentially small weight. This holds still if $\dl(\lambda) = \dlp(\lambda)$, while if $\dl(\lambda) = \dlm(\lambda)$ then $\lambda \in (\lambda_c^{(k-1)},\lambda_c^{(k)})$ whence the bound on the diameter isn't even needed for the exponentially small weight per \cref{prop:non-elem-control-entire-window}. The reasoning if $A = A_r'$ is symmetrical. 
\end{proof}

\appendix

\section{Monotonicity with generic fields and floors/ceilings}\label{sec:general-monotonicity-fkg}

Recall the general form of the SOS measure with arbitrary floors $\mathbf{a} = (a_v)_{v}$ and ceilings $\mathbf{b}= (b_v)_v$, boundary conditions $\phi$, and external field $\lambda$ from \cref{eq:general-SOS-measure}. 

\begin{proof}[\textbf{\emph{Proof of \cref{lem:SOS-monotonicity}}}]
It in fact suffices to check this for $V$ being a single vertex, say the origin, because given that, one can run a column Glauber dynamics (fully resampling the height at a vertex), whose $t\to\infty$ limit serves as the coupling attaining the stochastic domination expressed above. 
At a single vertex, we check the stochastic domination between the two distributions. The distribution can be expressed as 
\begin{align*}
    p(\varphi) := e^{ - \lambda \varphi_v} \big(\prod_{w\sim v} e^{ - \beta |\varphi_v - \phi_w|} \big) \mathbf{1}_{\varphi_v \ge a_v} \mathbf{1}_{\varphi_v \le b_v}\,,
\end{align*}
and $p'(\varphi)$ is defined analagously with the primed parameters. 
The ratio $p'(\varphi)/p(\varphi)$ can be broken up into the individual ratios in the above product. We claim that each of these are individually increasing in $\varphi_v$. 
 Since the product of increasing functions will also be increasing, that gives us that the ratio of the mass functions is increasing, from which the stochastic domination follows. 
 
 The ratio $e^{-(\lambda'- \lambda)\varphi_v}$ is increasing in $\varphi_v$ because $\lambda' - \lambda \le 0$. The ratios of the indicator functions are clearly increasing in $\varphi_v$ since 
\begin{align*}
    \frac{\mathbf{1}_{\varphi_v\ge a_v'}}{\mathbf{1}_{\varphi_v\ge a_v}} = \begin{cases}
        0 &  \varphi_v \in [a_v, a'_v] \\ 1 & \varphi_v \in [a_v',\infty)
    \end{cases}\,, \qquad \text{and}\qquad     \frac{\mathbf{1}_{\varphi_v\le b_v'}}{\mathbf{1}_{\varphi_v\le b_v}} = \begin{cases}
        1 &  \varphi_v \in (-\infty, b_v] \\ \infty & \varphi_v \in [b_v,b_v') 
    \end{cases}\,,
\end{align*}
(both measures are supported on $[a_v,b_v']$ so we don't care about the cases where it is $0$ divided by $0$). Finally,
\begin{align*}
    \frac{e^{ - | \varphi_v - \phi'_w|}}{e^{ - | \varphi_v - \phi_w|}}  = e^{|\varphi_v - \phi_w| - | \varphi_v - \phi'_w|} = \exp\left(\begin{cases}
      \phi_w - \phi'_w  &  \varphi_v \le \phi_w \\ (\varphi_v - \phi_w) - (\phi_w' - \varphi_v) & \varphi_v \in (\phi_w,\phi_w') \\  \phi_w' - \phi_w  &  \varphi_v \ge \phi_w'
    \end{cases} \right)\,.
\end{align*}
The middle term can be re-expressed as $2(\varphi_v - \frac{1}{2}(\phi_w + \phi'_w))$. It is easy to see with this writing that for $\varphi_v \in (\phi_w,\phi_w')$, this is in absolute value less than $\phi_w' - \phi_w$, so it is also increasing in $\varphi_v$. 
\end{proof}

\section{Exponential tails on non-elementary renormalized weights}\label{sec:exp-tails-nonelem-renormalized-weights}

In this section we follow the strategy of \cite{DiMa94} adapted to contours to provide the proof of \cref{lem:non-elem-contour-bound}. 

\subsection{Clusters of non-elementary contours}

We begin by canonically splitting any contour collection in $\mathscr{G}_{\eta,h,V}$ up into clusters of contours $\gamma$ that are non-elementary at $h(\gamma)$, and in between regions where the contributions are given by corresponding renormalized elementary partition functions. This follows the exposition leading up to Proposition 2.6 in \cite{CeMa1} (also \cite[Eq.~(2.33)]{DiMa94}). 

\begin{definition}\label{def:cluster-of-nonelems}
    [Cluster of non-elementary contours] A set of contours $\cC$ is called a \emph{non-elementary cluster} if one contour in $\cC$ nests all the others, along each nesting path ($\gamma_1 \subset \gamma_2 \subset ...$) only the heights exterior to the outermost contour and interior to the innermost contour are $h$, and every $\gamma\in \cC$ is $h(\gamma)$-non-elementary. 

    Any collection of contours $(\gamma_i)$ each of which are $h(\gamma_i)$-non-elementary can be uniquely decomposed into a set of clusters of non-elementary contours. 
\end{definition}

In words, given a set of $\gamma$ each of which is $h(\gamma)$-non-elementary, a cluster of non-elementary contours is an excursion away from height $h$. 
For a cluster of non-elementary contours $\mathcal C$, and each $\gamma\in \cC$, let
\begin{align*}
\mathsf{Ann}(\gamma;\mathcal C) &  = \Int(\gamma) \setminus \bigcup_{\gamma'\in \mathcal C: \gamma' \subsetneq \gamma} \Int(\gamma')\,,
\end{align*}
if $\exists \gamma'\in \mathcal C : \gamma'\subsetneq \gamma$, and let $\mathsf{Ann}(\gamma;\cC) = \emptyset$ otherwise. 
We can then define 
\begin{align*}
    \mathsf{Ann}(\cC) & = \bigcup_{\gamma} \mathsf{Ann}(\gamma;\cC)\,.
\end{align*}

\begin{definition}
    If $\mathcal C$ is a cluster of non-elementary contours, its $h$-renormalized weight is given by
    \begin{align}\label{eq:cluster-weight}         W^\rn_h(\mathcal C) = e^{ - \beta \sum_{\gamma \in \mathcal C} |\gamma|} \frac{\prod_{\gamma \in \mathcal C} Z^\elem_{S(\gamma),h(\gamma),\mathsf{Ann}(\gamma;\mathcal C)}}{Z^\rnelem_{h,\mathsf{Ann}(\mathcal C)}}\,,
    \end{align}
    where we recall $S(\gamma)=+$ if $\gamma$ is an up contour and $S(\gamma) = -$ if $\gamma$ is a down contour. 
    \end{definition}

\begin{remark}
        Notice that on the numerator of \cref{eq:cluster-weight}, it does not matter that we are writing $Z^\elem$ rather than $Z^\rnelem$ for the not-necessarily simply connected domain $\mathsf{Ann}(\gamma;\cC)$ since all its holes are necessarily $h(\gamma)$-non elementary and therefore cannot be confined by $h(\gamma)$-elementary contours.
\end{remark}

The following shows that these renormalized weights together with the renormalized weights of the elementary contours give a different way of writing the renormalized partition function (c.f., \cref{lem:part-function-renormalized-weights}). Let $\mathscr{C}_{\eta,h,V}$ be the set of all admissible collections of clusters of $h$-non-elementary contours. When we write the pair $(C,\Gamma)\in (\mathscr{C}_{\eta,h,V},\mathscr{G}^\rnelem_{\eta,h,V})$, we further impose that in each $\cC$, the contours of $\Gamma$ are compatible with the outermost and innermost contours of $\cC$ (i.e., the boundary of $\mathsf{Ann}(\cC)$).   

\begin{lemma}\label{lem:partition-function-non-elem-elem}
For every simply-connected $V$, every $h$, and every boundary signing $\eta$, 
    \begin{align*}
        Z_{\eta, h,V}^\rn = e^{ - \lambda h |V|} \sum_{(C,\Gamma) \in (\mathscr C_{\eta, h,V},\mathscr G_{\eta, h,V}^\elem)} \prod_{\cC\in C}  W^\rn_h(\mathcal C) \prod_{\gamma \in \Gamma}  W^\rn_h(\gamma)\,.
    \end{align*}
    \end{lemma}

\begin{proof}
    We can express the full SOS partition function as 
    \begin{align*}
        Z_{\eta,h,V} & = \sum_{C\in \mathscr{C}_{\eta,h,V}} Z^\rnelem_{h,\Ext(C)}  \prod_{\cC\in C} e^{ - \beta \sum_{\gamma \in \cC} |\gamma|} \prod_{\gamma \in \cC}  Z_{S(\gamma),h(\gamma),\mathsf{Ann}(\gamma;\cC)}^\elem \\ 
        & = \sum_{C\in \mathscr{C}_{\eta,h,V}} Z_{h,\Ext(\cC)}^\rnelem \prod_{\cC\in C} W_h^\rn(\cC) Z_{h,\mathsf{Ann}(\cC)}^\rnelem\,.
    \end{align*}
    Expanding out each $Z^\rnelem$ term per \cref{eq:elem-rnelem-partition-functions} yields the claim. 
    \end{proof}

Using our understanding of the elementary partition functions in the windows $I_h$ from \cref{eq:I-i}, we get the following upper bound on the renormalized weights of clusters of non-elementary contours. 

\begin{lemma}\label{lem:first-non-elem-bound}
    If $\lambda \in I_h$, for every cluster $\cC$ of non-elementary contours, 
    \begin{align*}
         W^\rn_h(\mathcal C) \le \exp\Big(- (\beta-1) \sum_{\gamma \in \cC} |\gamma| - \sum_{\gamma \in \cC} e^{ - 4\beta (h \wedge (h(\gamma)+1))-3\beta} |\mathsf{Ann}(\gamma;\cC)|\Big)\,.
    \end{align*}
\end{lemma}

\begin{proof}
    Notice first that 
    \begin{align*}
        Z_{h,\mathsf{Ann}(\cC)}^\rnelem \ge \prod_{\gamma \in \mathcal C} Z^\rnelem_{S(\gamma), h,\mathsf{Ann}(\gamma;\cC)}\,.
    \end{align*}
    since the latter is only more restrictive in terms of which contours are permitted in $\mathsf{Ann}(\cC)$. 
    Therefore, 
    \begin{align*}
         W^\rn_h(\cC) \le e^{ - \beta \sum_{\gamma\in \cC} |\gamma|}  \prod_{\gamma \in \cC} \frac{Z^\rnelem_{S(\gamma),h(\gamma),\mathsf{Ann}(\gamma;\cC)}}{Z^\rnelem_{S(\gamma),h,\mathsf{Ann}(\gamma;\cC)}}\,.
    \end{align*}
    Using \cref{lem:elementary-free-energies} and the lower bounds on the elementary free energies from \cref{lem:free-energy-lower-bound-in-windows} when $\lambda \in I_h$ on each of the terms in the product, we conclude (noting that $\sum_{\gamma \in \cC} |\partial_e \mathsf{Ann}(\gamma;\cC)|\le 2 \sum_{\gamma \in \cC} |\gamma|$). 
\end{proof}

\subsection{Weights of non-elementary contours}
For a cluster of non-elementary contours $\cC$, let $\gamma^\out$ be its outermost contour (there being only one by the definition of such clusters). 
We begin with the following bound on the total contribution from all clusters of non-elementary contours nesting a point $x\in V$, and having $|\gamma^\out|\ge r$. 
\begin{lemma}\label{lem:total-contribution-large-non-elem-in-I-i}
If $\lambda \in I_h$, then for every $r$, 
    \begin{align*}
        \sum_{\substack{\cC: x\in \Int(\gamma^\out) \\  |\gamma^\out| \ge r}}  W^\rn_h(\cC) \le \exp(- (\beta -4)r)\,.
    \end{align*}    
\end{lemma}

\begin{proof}

    We follow the general approach of~\cite[Lemmas 2.13--2.14]{DiMa94} to perform the sum by associating to each cluster of non-elementary contours a corresponding tree-like object. We associate to $\cC$ a \emph{witness} as follows:
    \begin{itemize}
        \item Take the set of (dual) edges associated to the collection of contours rooted at their respective heights, and color all such edges $\red$. 
        \item For each contour $\gamma$, and each contour $\gamma'$ nested in $\gamma$, add a straight line in the $e_1$ direction of $\blue$ faces at $h(\gamma)$ connecting $\gamma$ to $\gamma'$. Then for each $\gamma'$ nested in $\gamma$, erase all $\blue$ faces in their interiors, and continue the procedure for those $\gamma'$. 
    \end{itemize}
    This gives a treelike structure on contours, where a contour $\gamma$ is the parent of $\gamma'$ if they are connected by $\blue$ faces, and either $\gamma$ nests $\gamma'$ or $\gamma$ is closer (in this ordering) to their common nesting contour. For each $\gamma$, let $F_B(\gamma)$ be the set of $\blue$ faces between $\gamma$ and its children (these being distinct). We can upper bound the renormalized weight $W^\rn_h(\cC)$ associated to such a $\cC$ by the following: 
    \begin{align*}
         W^\rn_h(\cC) \le \prod_{\gamma} e^{ - (\beta - 1) |\gamma| - |F_B(\gamma)| e^{ - 4\beta(h\wedge h(\gamma)) - 3\beta}}\,.
    \end{align*}
    In order to see this, recall \cref{lem:first-non-elem-bound} and notice that for each $\gamma$ that is not a leaf in the tree construction, either the $\blue$ faces constituting $F_B(\gamma)$ are part of $\textrm{Ann}(\gamma;\cC)$, or they are part of $\textrm{Ann}(\gamma';\cC)$ where $\gamma'$ is the innermost nesting contour of $\gamma$. that is to say that they are either at height $h(\gamma)$ or at height $\max\{0,h(\gamma) \pm 1\}$. 

    At the same time, any child of $\gamma$ must be either $h(\gamma)$-non elementary or $h(\gamma')$-non-elementary where $\gamma'$ is innermost nesting contour. so in particular, it must have size at least $L_0:= e^{ 3\beta \max\{1, h(\gamma^\out)\}}$. 

    We now perform the count one generation of the tree corresponding to $\cC$, by enumerating over $\gamma^\out$ of size at least $r$, then for each point along $\gamma^\out$, deciding whether to start a blue path or not, enumerating over the length of the blue path $k$, then picking up a weight corresponding to all the possible choices of subtree to place there. Assuming (inductively) that the total weight of all subtrees whose outermost contour has length at least $L_0$ has a bound of $\exp( - (\beta - 4)L_0)$, we get 
    \begin{align*}
      \sum_{\substack{\cC: x\in \Int(\gamma^\out) \\ |\gamma^\out|\ge r}} W_h^\rn(\cC)  \le  \sum_{\ell \ge r} 4^\ell e^{ - (\beta -1)\ell } \big(1+ e^{ - (\beta -4) L_0}\sum_{k =1}^{\infty} e^{ - k e^{ - 4\beta (h\wedge  h(\gamma^\out)) - 3\beta}}\big)^{\ell}\,.
    \end{align*}
    The series in $k$ sums up to at most $e^{ 4\beta  h(\gamma^\out) + 3\beta}$ so this is at most 
    \begin{align*}
        \sum_{\ell \ge r} e^{ - (\beta -3) \ell} \exp\big(\ell e^{ - (\beta -4)e^{ 3\beta \max\{1,h(\gamma^\out)\}}} e^{ 4\beta h(\gamma^\out)+3\beta}\big)\,.
    \end{align*}
    For $\beta$ large, the second exponential is easily seen to be at most $\exp( \epsilon_\beta \ell)$, whence the sum is easily seen to be at most $\exp(- (\beta -4)r)$ as claimed. 
    \end{proof}

\begin{proof}[\textbf{\emph{Proof of \cref{lem:non-elem-contour-bound}}}]
If $\gamma$ is elementary, the bound follows from \cref{thm:main-elementary-general-bc}. Suppose now that $\gamma$ is an up non-elementary contour. Then,
\begin{align*}
     W^\rn_h(\gamma) = e^{-\beta|\gamma|} \frac{Z_{\eta,h+1,\Int(\gamma)}}{Z_{\eta,h,\Int(\gamma)}} = \mu_{+,h,\Int(\gamma)}\Big(\bigcap_{x\in \partial_i \Int(\gamma)} \{\varphi(x)\ge h+1\} \Big)\,.
\end{align*}
In order for $\varphi_x \ge h+1$ for all $x\in \partial_i \Int(\gamma)$, then there must be a positive contour $\gamma'$ in $\Gamma(\varphi)$ that lines up exactly with $\gamma$; 
 If $\gamma$ is not $h$-elementary, then by \cref{lem:partition-function-non-elem-elem}, this probability is bounded by the Peierls map that deletes its entire cluster $\cC$ (whose weight is $ W^\rn_h(\cC)$), whence \cref{lem:total-contribution-large-non-elem-in-I-i} performs the summation over the choice of the cluster and concludes the proof.
 The case where $\gamma$ is a down non-elementary contour is analogous. 
 \end{proof}

\section{Canonical paths bounds for SOS models}\label{sec:canonical-path-SOS}
We show that the inverse gap on domains of cut-width $m$ and ceiling $\ell$ is at most exponential in $m\ell$.  

\begin{proof}[\textbf{\emph{Proof of \cref{lem:canonical-path-bound-sos}}}]
    Recall (e.g.,~\cite[Corollary 13.21]{LP}) that for a reversible Markov chain with a finite state space $\Omega$, transition matrix $P$, and stationary distribution $\mu$, its spectral gap is lower bounded by the congestion
    \begin{align}\label{eq:congestion}
        \min_{\Gamma= (\gamma_{ab})_{a,b\in \Omega}} \max_{x,y\in \Omega:P(x,y)>0} \frac{1}{\mu(x)P(x,y)} \sum_{w,z: (x,y)\in \gamma_{w,z}}\mu(w)\mu(z)|\gamma_{w,z}|\,,
    \end{align}
    where each $\gamma_{ab}$ is a sequence of positive rate transitions starting at $a$ and ending at $b$. We first do the argument for $\Lambda_m$, then describe what changes for the annulus $A_{n,m}$. 
    
    Enumerate the vertices of $\Lambda_m$ in lexicographic order, denoted $v_1,v_2,...$. For two configurations $a,b$, the path $\gamma_{ab}$ is described by the updates where each of these vertices is sequentially processed, and by processed we mean that if $a_v\ne b_v$, then there is a minimal sequence of transitions to take $a_v$ to $b_v$. Evidently, $|\gamma_{a,b}|\le m^2\ell$ for all $a,b$. Further, for fixed $(x,y)$, there is an injection between pairs of configurations $(w,z)$ such that $(x,y)\in \gamma_{w,z}$ and $\Omega$, defined as follows. If $(x,y)$ entails a transition at site $v_i$, let $\tilde w$ agree with $w$ on $(v_j)_{j>i}$ and let $\tilde w$ agree with $z$ on $(v_j)_{j<i}$; finally let it take the value of $x$ on $v_i$. Similarly, let $\tilde z$ agree with $z$ on $(v_j)_{j>i}$, agree with $w$ on $(v_j)_{j<i}$, and take the value of $y$ at $v_i$. Given $i$, we can evidently recover $(w,z)$ from $(\tilde w,\tilde z)$. Moreover, the configuration $x$ is exactly $\tilde w$, so given $x$ and $i$ which we can read off from $(x,y)$, this gives an injection from $(w,z):(x,y)\in \gamma_{wz}$ to $\tilde z$, and the congestion can be rewritten as 
    \begin{align*}
        \frac{m^2 \ell}{P(x,y)} \sum_{\tilde z\in \Omega}\frac{\mu(w)\mu(z)}{\mu(\tilde w)\mu(\tilde z)} \mu(\tilde z)\,.
    \end{align*}
    For any valid transition, $P(x,y)$ is clearly lower bounded by $e^{ - 4\beta \ell -\lambda}$. In the ratio of weights between $w,z$ and $\tilde w,\tilde z$ the external field quantities cancel out exactly, as do all gradients except those along the cut between $(v_j)_{j<i}$ and $(v_j)_{j>i}$. In the lexicographic ordering on $\Lambda_m$, this cut is at most $m$ for every $i$, so the largest that ratio can be is $e^{ \beta m\ell}$. The sum over $\tilde z$ then gives a $1$ since $\mu$ is a probability measure, altogether yielding the desired. 

    We can absorb the $m^2$ and $\ell$, and the polynomial in $m^2\ell$ factor that comes from translating spectral gap into mixing time, by adjusting the constant in the exponent, getting the desired. 

    In order to get the claimed bound for the annulus, follow the same reasoning, but instead of using the lexicographic ordering, we use one that attains a cut-width of at most $2m$. This ordering will be obtained by starting from the $m\times m$ block in one corner, performing the lexicographic order there, then processing around the annulus one row/column (whichever is thinner, depending on the orientation) one at a time. The extra factor of $n$ in the bound here comes from the translation of spectral gap to mixing time. 
\end{proof}

\subsection*{Acknowledgements}
The authors thank the anonymous referees for their careful reading of the manuscript.
R.G.\ acknowledges the support of NSF DMS-2246780 and E.L.\ acknowledges the support of NSF DMS-2054833.

\bibliographystyle{abbrv}
\bibliography{references}
\end{document}